\documentclass[]{amsart}
\usepackage[letterpaper,hmargin=27mm,vmargin=27mm,bmargin=27mm,tmargin=22mm]{geometry}
\usepackage[utf8]{inputenc}
\usepackage{amsmath,amsthm,amssymb}
\usepackage{mathtools}
\usepackage[T1]{fontenc}
\usepackage[english]{babel}
\usepackage{mathrsfs}
\usepackage{yfonts}
\usepackage{pgf,tikz}
\usetikzlibrary{arrows}
\usepackage{listings}
\usepackage{graphicx}
\usepackage{float}
\usepackage{caption}
\usepackage{subcaption}
\usepackage[normalem]{ulem}
\usepackage{changepage}
\usepackage[final]{pdfpages}
\usepackage{hyperref}
\usepackage{lipsum}
\usepackage{url}
\usepackage{algorithm}
\usepackage{algpseudocode}
\usepackage{pgfplots}
\usepackage[toc]{glossaries}
\usepackage{afterpage}
\usepackage{mathrsfs}
\usepackage{yfonts}
\usepackage{pgf,tikz}
\usetikzlibrary{arrows}
\usepackage{authblk}
\usepackage{bbm}
\usepackage{enumitem}   
\usepackage{makecell}

\usepackage[
backend=biber, %
style=numeric, %
 url=false, %
 doi=false, %
 eprint=true, %
 giveninits=true, %
 isbn=false,
 maxbibnames=5, %
 sortcites=true, %
]{biblatex}

\AtEveryBibitem{\clearlist{language}}
\DeclareFieldFormat{title}{#1}
\newtheorem{theorem}{Theorem}[section]

\newtheorem{lemma}[theorem]{Lemma}
\newtheorem{definition}[theorem]{Definition}
\newtheorem{remark}[theorem]{Remark}

\theoremstyle{definition}

\newcommand{\norm}[1]{\left\lVert#1\right\rVert}
\newcommand{\abs}[1]{\left\lvert#1\right\rvert}

\DeclareMathOperator{\im}{Im}

\DeclareMathOperator{\supp}{supp}

\DeclareMathOperator{\Tr}{Tr}

\DeclareMathOperator{\GOE}{GOE}

\newcommand{\Dim}{\Delta \im}
\newcommand{\di}{\mathop{}\!{d}}

\newcommand{\dt}{\frac{\di}{\di t}}
\newcommand{\dF}{\frac{\di}{\di t}}
\DeclareMathOperator{\ii}{i}
\DeclarePairedDelimiter\ceil{\lceil}{\rceil}
\DeclarePairedDelimiter\floor{\lfloor}{\rfloor}
\DeclarePairedDelimiter\inner{\langle}{\rangle}

\newcommand{\e}{\mathrm{e}}

\newcommand{\Cinf}{C_{\mathrm{inf}}}
\newcommand{\Csup}{C_{\mathrm{sup}}}
\newcommand{\msc}{m_{\mathrm{sc}}}
\DeclareMathOperator{\EX}{\mathbb{E}}%
\DeclareMathOperator{\indicator}{\mathbbm{1}} %
\DeclareMathOperator{\Prob}{\mathbb{P}}%
\DeclareMathOperator{\disSum}{\mathring{\sum}}

\newcommand{\lse}{\mathrm{LSE}}
\newcommand{\func}{\mathcal{L}}
\newcommand{\inFac}{\mathrm{IF}}
\newcommand{\Fasump}{100}

\newcommand{\C}{{\mathbb C}}
\newcommand{\R}{{\mathbb R}}
\newcommand{\ie}{\emph{i.e., }}
\newcommand{\eg}{\emph{e.g., }}
\newcommand{\cf}{\emph{c.f., }}

\numberwithin{equation}{section}

\makeatletter
\@ifundefined{endofdump}{\def\endofdump{}}{}
\makeatother
\endofdump
\makeglossaries

\begin{document}

\begin{center}

	\begin{minipage}{0.85\textwidth}
		\vspace{2.5cm}
		
		\begin{center}
			\large\bf
			Order statistics for edge eigenvectors of Wigner matrices
		\end{center}
	\end{minipage}
\end{center}

\vspace{0.5cm}

\begin{center}
	\begin{minipage}[t]{0.32\textwidth}
		\centering
		{Zhigang Bao} \\
		{\small The University of Hong Kong \\ \texttt{zgbao@hku.hk}}
	\end{minipage}
	\hfill
	\begin{minipage}[t]{0.32\textwidth}
		\centering
		{Teodor Bucht} \\
		{\small KTH Royal Institute of Technology \\ \texttt{tbucht@kth.se}}
	\end{minipage}
	\hfill
	\begin{minipage}[t]{0.32\textwidth}
		\centering
		{Kevin Schnelli} \\
		{\small KTH Royal Institute of Technology \\ \texttt{schnelli@kth.se}}
	\end{minipage}
	
	\vspace{0.6cm}
	
	{\small June 15, 2026}
\end{center}

\vspace{0.8cm}

\begin{center}
	\begin{minipage}{0.85\textwidth}
		\small
		\textbf{Abstract.} 
		In this paper, we establish a general comparison theorem for the order statistics of the edge eigenvectors for generalized Wigner matrices. Consequently, we derive the Gumbel law for the maximal edge eigenvector component and prove the universality of the Gaussian fluctuations of the order statistics in an intermediate regime close to the maximum. In addition, our comparison result also implies a quantitative first order estimate for moderately small order statistics.
	\end{minipage}
\end{center}

\vspace{0.5cm}

\pagenumbering{arabic}

\section{Introduction}
For a vector $(y_i)_{i=1}^n$ of i.i.d.\ standard Gaussians, it is a classic result from extreme value theory that the maximal squared entry, $\max_i y_i^2$, converges in distribution after appropriate rescaling to the Gumbel law. Precisely we have that,
\begin{equation}
	\lim_{n \to \infty} \Prob\left(\frac{\max_{i} y_i^2 - 2 \log n + \log\log n + \log\pi}{2} \leq x\right) = \e^{-\e^{-x}}, \quad x \in \mathbb{R}.
	\label{standard_iid_result}
\end{equation}

For the Gaussian orthogonal ensemble, $\GOE_N$, it follows from orthogonal invariance that individual eigenvectors are uniformly distributed on the unit sphere $\mathbb{S}^{N-1}$. Consequently, the maximum of the squares of the entries of an eigenvector for $\GOE_N$ converges to the Gumbel law. Furthermore, it was shown in \cite{jiang_2005maxima_entries_haar} that the maximal entry of all eigenvectors for the $\GOE_N$ convergence to the Gumbel law.

In light of the above properties it is natural to conjecture that the maximum of the squares of the entries of an eigenvector for generalized Wigner matrices also converges to the Gumbel law. Determining whether the distribution of the maximal entry is universal was stated as an open question in the survey paper \cite{vu2016_survey}. Our first main result, Theorem \ref{gumbel_convergence_prop}, answers this question for the edge eigenvectors.

There is a considerable body of literature concerning universality of eigenvalue and eigenvector statistics of random matrices. For Wigner matrices, early references on universality of local eigenvalue statistics include \cite{erdos_schlein_yau_2011_universality, tao2010local}, see also \cite{dynamical_approach_rmt} for a comprehensive overview. Along with the development of universality of eigenvalue statistics, it was also proved that the eigenvectors satisfy the sup-norm delocalization, i.e.\ $\sqrt{N} \norm{u_\alpha}_{\infty} \leq C (\log N)^{c}$  with very high probability, see for instance \cite{erdos_schlein_yau_complete_delocalization,vu_kee2015,bloemendal_erdos_knowles_yau_yin_2014,BenigniL.2022Odfg}. Notably, \cite{BenigniL.2022Odfg} showed optimal  delocalization in the sense that for any $\epsilon > 0$, it holds with high probability that $\sqrt{N} \norm{u_\alpha}_{\infty} \leq \sqrt{(2 + \epsilon) \log N}$. For the universality of entry distribution of eigenvectors the first works~\cite{knowles_yin_2013,tao_vu_2012_eigenvectors} showed universality for a finite set of entries of any eigenvector under sufficiently strong moment matching conditions. The moment matching assumptions were removed in the landmark work \cite{BourgadeP.2017TEMF} where the dynamical approach used to show universality of eigenvalue statistics was extended to the study of eigenvectors through the eigenvector moment flow. For further developments we also refer to~\cite{yau_marcinik_2022_noisy_eigenvectors, cipolloni_erdos_schroder_2022, cipolloni_henheik_kolupaiev_2023}. Very recently, \cite{benigni2026quantitativeeigenvectoruniversalitygeneralized} showed a quantitative universality result for the eigenvectors of generalized matrices, through a novel analysis of the Dyson vector flow.

The sup-norm delocalization results mentioned above bound the amount of mass that can be carried by a single eigenvector entry. A different type of delocalization is the so called \textit{no-gaps delocalization} introduced in \cite{RudelsonVershynin2016}. The no-gaps delocalization instead gives a bound on how spread out the mass of the eigenvectors are. More specifically, no-gaps delocalization results bound the following quantity from below,  \begin{equation}
	\min_{\substack{ I \subset [[1, N]] \\ \#_I = \floor{(1 - c)N}}} \sum_{i \in I} \abs{u_{\alpha}(i)}^2,
	\label{no_gap_delocalization}
\end{equation} where $u_\alpha$ is any eigenvector, $I$ is a subset of indices of size $\#_I$, and $c$ is a parameter that can either be constant or $N$-dependendent converging to one. In \cite{RudelsonVershynin2016} lower bounds for \eqref{no_gap_delocalization} are proven in the regime $1 - c \geq 1/N$.

In \cite{vu2016_survey} it is shown that the convergence in distribution results from \cite{BourgadeP.2017TEMF} imply first order convergence for \eqref{no_gap_delocalization} for a constant $c$. Further from the results in~\cite{BourgadeP.2017TEMF} one can derive first order convergence of the order statistics of the $\ceil{c N}$-th largest entry of $u_\alpha$, denoted $(N \abs{u_\alpha}^2)_{(\ceil{c N})}$. Our second main result, Theorem \ref{no_gap_deloc_prop} gives a quantitative convergence result for $(N \abs{u_\alpha}^2)_{(\ceil{c N})}$ for edge eigenvectors $u_{\alpha}$, albeit at a very small rate, from which we also get a quantitative no-gaps delocalization estimate.

Our proof relies on Green function comparison, first introduced in \cite{green_function_comparison_reference}. The strategy used to handle non-homogeneous variances is inspired by \cite{schnelli_xu_generalized_22}. To bound error terms we crucially rely on the local law from \cite{bloemendal_erdos_knowles_yau_yin_2014}. We use a smoothed order statistics functional that compares all the entries simultaneously. The smoothed functional we use is introduced in Section \ref{smoothed_order_statistics_section}, and is the quantile of a smoothed empirical cdf with the sigmoid as kernel. This functional allows us to compare any order statistics in our comparison theorem. In particular, we can obtain the limiting fluctuations of intermediate order statistics, see \eqref{intermediate_fluctuation_universality}.

\subsection{Notation}
Inequalities involving $N$-dependent quantities are assumed to hold for sufficiently large $N$. We denote the $l_1$ norm of vectors $u \in \mathbb{R}^k$ as $\norm{u}_1 = \sum_{i=1}^k \abs{u(i)}$. For a multi-index $\beta$ we will use $\#_{\beta}$ to denote its length, and for a finite set $A$ we shall denote its cardinality as $\#_{A}$. We frequently use $\sum_\mathcal{I}$ to denote a sum with index set $\{i_1, \ldots, i_{\#_{\mathcal{I}}}\} = \mathcal{I}$ summed from $1$ to $N$. We denote the part of $\sum_{\mathcal{I}}$ where all summation indices assume distinct values as $\disSum_{\mathcal{I}}$. We denote the set of non-negative integers as $\mathbb{N}$. For a square $N \times N$ matrix $A$, we denote its normalized trace $\frac{1}{N} \Tr A$ as $\underline{A}$. We denote the canonical basis of $\R^n$ by $(e_i)_{i=1}^n$. 

We shall use the following notion of stochastic domination. Properties of stochastic domination can be found in \cite[Proposition 6.5]{dynamical_approach_rmt}.

\begin{definition}
	Let $X = X(N)$ and $Y = Y(N)$ be two sequences of non-negative random variables. We say that $Y$ \textit{stochastically dominates} $X$, if for all $\epsilon > 0$ and $D > 0$ we have \begin{equation*}
		\Prob(X(N) > N^{\epsilon} Y(N)) \leq N^{-D},
	\end{equation*} for all $N \geq N_0(\epsilon, D)$. We denote this as $X \prec Y$ or $X = O_{\prec}(Y)$.
	\label{stochastic_domination_def}
\end{definition}
Note that the $N_0(\epsilon, D)$ in the above definition is allowed to depend on $X$ and $Y$. We will point out in the text when \eg one $N_0(\epsilon, D)$ holds uniformly for all $X$ and $Y$ satisfying a certain bound. If we have an additional time parameter in our $X$ and $Y$, say $X = X(t, N)$ and $Y = Y(t, N)$, then when we say that a stochastic domination holds uniformly in \eg $t \in \mathbb{R}_{\geq 0}$ we mean that for all $\epsilon > 0$ and $D > 0$ there exists $N_0 = N_0(\epsilon, D, X, Y)$ such that for $N \geq N_0$, \begin{equation*}
	\sup_{t \in \R_{\geq 0}} \Prob(X(t, N) > N^{\epsilon} Y(t, N)) \leq N^{-D}.
\end{equation*}

\subsection{Assumptions and main results}
In this paper we consider \textit{generalized Wigner matrices}.
\begin{definition}
	We call an $N \times N$ symmetric random matrix generalized Wigner if it satisfies \begin{enumerate}
		\item The elements $(h_{ij})_{i \geq j}$ are independent random variables with $\EX[h_{ij}] = 0$. 
		\item The variances of $h_{ij}$ are given by the variance profile matrix $V = (V_{ij})_{1 \leq i, j \leq N}$, where $V_{ij} = \EX[h_{ij}^2]$. We require $V$ to satisfy the normalization condition \begin{equation}
			\sum_{i = 1}^N V_{ij} = 1, \quad 1 \leq j \leq N. 
			\label{sum_condition}
		\end{equation} Further, we assume there exist two constants $\Cinf, \Csup > 0$ independent of $N$ such that \begin{equation}
			\Cinf \leq \inf_{i,j} \{N V_{ij}\} \leq \sup_{i,j} \{N V_{ij}\} \leq \Csup.
			\label{mean_field_condition}
		\end{equation} 
		\item All moments of $\sqrt{N} h_{ij}$ are uniformly bounded, \ie for any $k \geq 3$ there exists $C_k$ independent of $N$ such that for all $1 \leq i, j \leq N$, \begin{equation}
			\EX[(\sqrt{N} h_{ij})^k] \leq C_k.
			\label{moment_condition}
		\end{equation}
	\end{enumerate}
	\label{general_wigner_def}
	\end{definition}

Throughout the paper, $H$ will be a generalized Wigner matrix. Denote the ordered eigenvalues of $H$ as \\ $\lambda_1 \leq \dots \leq \lambda_N$, and the corresponding eigenvectors $u_{1}, \ldots, u_{N}$. In this paper we are concerned with the behavior of the eigenvectors corresponding to the eigenvalues at the edge of the spectrum. Further, our results are isotropic in the sense that they hold for the eigenvectors in an arbitrary basis. More precisely, we have the following definitions. 

\begin{definition}
Let $Q = Q(N)$ be a deterministic sequence of arbitrary orthogonal matrices. Denote the row vectors of $Q$ by $q_1, \ldots, q_N$.
\end{definition}

The eigenvectors of $Q H Q^T$ are $Q u_1, \ldots, Q u_N$, and for an index $\alpha \in [[1, N]]$ the entries of $Q u_{\alpha}$ are precisely $\inner{q_1, u_{\alpha}}, \ldots, \inner{q_N, u_{\alpha}}$. Next we define the set of products of entries that our results hold for. 

\begin{definition}
Let $\epsilon > 0$, and $J = J(N, \epsilon)$ be a sequence of index sets,
\begin{equation}
    J \subset \{(\alpha, i, j) : 1 \leq i, j \leq N,  1 \le \alpha \le \floor{N^{\epsilon/10}} \},
	\label{J_def_eq}
\end{equation} corresponding to products of entries of the edge eigenvectors. We assume $\#_{J} \to \infty$ as $N \to \infty$. Further, let $E_J = E_J(N)$ be the set, 
\begin{equation}
  E_J := \{N \inner{u_{\alpha}, q_i} \inner{u_{\alpha}, q_j} : (\alpha, i,j) \in J\}.
\end{equation}
 Let $(N \inner{u_{\alpha}, q_i} \inner{u_{\alpha}, q_j})_{(k)}$ denote the $k$-th largest entry of $E_J$, \ie $(N \inner{u_{\alpha}, q_i} \inner{u_{\alpha}, q_j})_{(1)}$ denotes the largest entry of $E_J$, and $(N \inner{u_{\alpha}, q_i} \inner{u_{\alpha}, q_j})_{(\#_J)}$ the smallest. 
\end{definition}

Our first result is the following Theorem, that states that the largest squared entry of an edge eigenvector of $H$ is Gumbel distributed in the limit, and also that the entries of $E_J$ (properly rescaled) tend to a Poisson point process. Furthermore, it also gives a result on the intermediate order statistics, \ie that the limiting fluctuations of the $k$-th largest entry are Gaussian, when $k \to \infty$ and $k / \#_J \to 0$. 

\begin{theorem}[Universality of largest and intermediate order statistics for edge eigenvectors]
	Let $\epsilon > 0$ and choose $J$ such that it does not contain any cross terms, \ie \begin{equation*}
	J \subset \{(\alpha, i, i) : 1 \leq i \leq N,  1 \le \alpha \le \floor{N^{\epsilon/10}}\},
	\end{equation*} and let, \begin{equation*}
	 b_n = 2 \log n - \log \log n - \log \pi, \quad n \in \mathbb{N}.
	\end{equation*}
	Then there exists a $\delta > 0$ such that for $\epsilon < \delta$, the largest entry of $E_J$ converges in law to a Gumbel distribution, \begin{equation}
	\lim_{N \to \infty }\Prob\left(\frac{\max_{(\alpha,i,i) \in J} N |\inner{u_{\alpha},q_i}|^2 - b_{\#_{J}}}{2} \leq x\right) = \e^{-\e^{-x}}, \quad \forall x \in \R.
  \label{gumbel_convergence_result}
	\end{equation} 
	Furthermore, the point process of the rescaled entries of $E_J$ converges to the Poisson process with intensity measure $\e^{-x} \di x$, \ie for all non-negative continuous functions $g : \R \to \R$ with compact support we have, \begin{equation}
		\lim_{N \to \infty } \EX\Bigg[\exp \Big(- \sum_{(\alpha, i, i) \in J} g\Big(\frac{1}{2}(N |\inner{u_{\alpha}, q_i}|^2 - b_{\#_{J}})\Big) \Big)\Bigg] = \exp \Big(- \int_{\R} \big(1 - \e^{-g(x)}\big) \e^{-x} \di x \Big).
		\label{point_process_convergence_result}
	\end{equation} 
	Finally, let $F_{\chi^2(1)}$ and $f_{\chi^2(1)}$ denote the cdf and pdf of the $\chi^2(1)$ distribution, and let $\Phi$ denote the cdf of a standard Gaussian. For $k = k(N) \leq N^{\epsilon/2}$, satisfying $k \to \infty$ and $k / \#_{J} \to 0$ as $N \to \infty$, we have with, \begin{equation*}
		d_{\#_J} := F_{\chi^2(1)}^{-1}\left(1 - \frac{k}{\#_J}\right), \quad c_{\#_J} := \frac{\sqrt{k}}{\#_J f_{\chi^2(1)}(d_{\#_J})},
	\end{equation*} that,
	\begin{equation}
		\lim_{N \to \infty} \Prob \left( \frac{(N|\inner{u_{\alpha}, q_i}|^2)_{(k)} - d_{\#_J}}{c_{\#_J}} \leq x \right) = \Phi(x), \quad \forall x \in \R.
		\label{intermediate_fluctuation_universality}
	\end{equation}
	\label{gumbel_convergence_prop}
\end{theorem}

\begin{remark}
	The reason we are not able to provide an explicit lower bound for $\delta$ in Theorem \ref{gumbel_convergence_prop} is that we lack an explicit eigenvalue repulsion estimate that bounds the probability that eigenvalues at the edge are closer than $N^{-2/3 - \epsilon}$ for some explicit $\epsilon > 0$. If we had such a level repulsion result, we believe that our proof would give \eqref{intermediate_fluctuation_universality} for $k \leq N^{c}$ for some explicit $c > 0$. 
\end{remark}

Our second result is a quantitative law of large numbers for general order statistics of the edge eigenvectors' entries. 

\begin{theorem}[Law of large numbers for order statistics and no gap delocalization]
	Let $\epsilon > 0$ and choose $J$ to be $\{(\alpha, i, i), 1 \leq i \leq N\}$ for some $\alpha = \alpha(N) < N^{\epsilon/10}$. Let $\mathfrak{c} \in (0, 1]$ be a constant, and denote the cdf of the $\chi^2(1)$-distribution by $F_{\chi^2(1)} : \R_{\geq 0} \to [0, 1]$. The $\mathfrak{c}$-quantile is given by,
	\begin{equation*}
		u^{\star}_{\mathfrak{c}} := F_{\chi^2(1)}^{-1}(1 - \mathfrak{c}). 
	\end{equation*}	
	Then there exists a $\delta > 0$ such that for $\epsilon < \delta$ and sufficiently large $N$, with probability $\geq 1 - N^{-\epsilon/2}$, 
	\begin{equation}
		(N \abs{\inner{u_{\alpha},q_i}}^2)_{(\ceil{ \mathfrak{c}N})} = u^{\star}_{\mathfrak{c}} + O(N^{-\epsilon/4}).
		\label{order_statistic_concentration}
	\end{equation}
	A direct consequence of \eqref{order_statistic_concentration} is the following quantitative convergence: With probability $\geq 1 - N^{-\epsilon/4}$, 
	\begin{equation}
		\inf_{\substack{I \subset [[1, N]] \\ \#_{I} = \floor{(1-\mathfrak{c})N}}} \sum_{i \in I} \abs{\inner{u_{\alpha}, q_i}}^2 = \int_{\mathfrak{c}}^1 u^{\star}_{c} \di c + O(N^{-\epsilon/4}).
		\label{limit_of_smallest_sub_vector}
	\end{equation}
	\label{no_gap_deloc_prop}
\end{theorem}

\begin{remark} A non-quantitative convergence of the first order of the order statistics can be derived directly from the joint Gaussian law for finitely many eigenvector components in \cite{BourgadeP.2017TEMF}, as the latter implies the convergence of the empirical measure of the components, which further implies the convergence of the order statistics in the middle. Nevertheless, it is not clear whether a quantitative convergence with an effective rate can be obtained by such a method. Despite being weak, our rate could potentially be further improved along the approach proposed in this paper. We leave this for future study.
\end{remark}

Our main result is the following comparison theorem for order statistics of edge eigenvectors. Both of the above theorems are derived from it.
\begin{theorem}
Let $\epsilon > 0$ be a (small) real number. Recall $J = J(N, \epsilon)$ from \eqref{J_def_eq}, and let $K = K(N) \leq \#_{J}$ be a sequence of natural numbers. Let $F = F_N : \mathbb{R}^K \to \mathbb{R}$ be a sequence of smooth functions satisfying \begin{align}
\abs{F(x)} & \leq C, \nonumber \\
\sum_{1 \leq i_1, \dots, i_m \leq K} \abs{\partial_{i_1} \dots \partial_{i_m} F(x)} & \leq C_m N^{m \epsilon/4}, \label{F_deriv_assumptions_eq}
\end{align} for all $x \in \mathbb{R}^K$ and $m \in \mathbb{N}$, where $C_m$ are $m$-dependendent constants. Further, let $(k_l)_{l=1}^K$ be an ($N$-dependent) list of indices, 
\begin{equation*}
	1 \leq k_1 < \cdots < k_K \leq \#_J.
\end{equation*}
Then there exists a $\delta > 0$ such that for $\epsilon < \delta$ we have, \begin{equation}
	\abs{\EX[F(((N \inner{u_{\alpha}, q_i} \inner{u_{\alpha}, q_j})_{(k_l)})_{l=1}^K)] - \EX^{\GOE}[F(((N \inner{u_{\alpha}, q_i} \inner{u_{\alpha}, q_j})_{(k_l)})_{l=1}^K)]} \leq N^{-\epsilon/2},
	\label{main_result_eq}
\end{equation} for $N \geq N_0(\epsilon, C, (C_m)_{m \in \mathbb{N}})$. The expectation $\EX$ is with respect to a generalized Wigner matrix, and $\EX^{\GOE}$ is with respect to the eigenvectors of the $\GOE_N$ which is equal in law to the Haar measure. 
	\label{main_result}
\end{theorem}

The proof is presented in Section \ref{smoothed_order_statistics_section}.

\subsection{Proofs of Theorems \ref{gumbel_convergence_prop} and \ref{no_gap_deloc_prop}}
In this section we deduce Theorems \ref{gumbel_convergence_prop} and \ref{no_gap_deloc_prop} from Theorem \ref{main_result}.

\begin{proof}[Proof of Theorem \ref{gumbel_convergence_prop}]
	Let $\delta > 0$ be the constant given by Theorem \ref{main_result}. We shall show \eqref{gumbel_convergence_result} and \eqref{point_process_convergence_result} for $\epsilon$ with $0 < \epsilon < \delta$.
	
	For i.i.d. standard Gaussians $(y_i)_{i=1}^n$ we have from classic extreme value theory (see \eg Table 3.4.4 in \cite{Embrechts1997}) that,
	\begin{equation}
		\Prob\Big(\frac{1}{2} (\max_{1 \leq i \leq n} y_i^2 - b_n) \leq x \Big) \to \e^{-\e^{-x}}, \quad \forall x \in \R. 
		\label{iid_gumbel_conv}
	\end{equation} In other words, $\chi^2(1)$ belongs to the maximum domain of attraction of the Gumbel distribution. This, in particular implies that the binomial point process, \begin{equation*}
		\sum_{i=1}^{n} \delta_{\frac{1}{2} (y_{i}^2 - b_n)},
	\end{equation*} converges in distribution to a Poisson point process with intensity measure $\e^{-x} \di x$,
	\begin{equation}
		\lim_{N \to \infty} \EX \Bigg[\exp \Big(- \sum_{i=1}^{n} g\Big(\frac{1}{2} (y_i^2 - b_n)\Big)\Big) \Bigg] = \exp \Big(- \int_{\R} \big(1 - \e^{-g(x)}\big) \e^{-x} \di x \Big).
		\label{iid_laplace_conv}
	\end{equation} For a reference, see Chapter 2 and Appendix A of \cite{mikosch_wintenberger_2024}. 

	By Theorem 5 from \cite{jiang_2005maxima_entries_haar}, we have that the entries of the edge eigenvectors of the $\GOE$ can be estimated in high probability by i.i.d. Gaussians. In particular let $W_N$ be  drawn from $\GOE_N$. Denote the eigenvalues of $W_N$ by $\lambda_1(W_N) \leq \cdots \leq \lambda_{N}(W_N)$, and the corresponding eigenvectors $u_1(W_N), \dots, u_N(W_N)$. Then \cite{jiang_2005maxima_entries_haar} gives that there exists $(y_{i \alpha})_{1 \leq i \leq N, 1 \leq \alpha \leq \floor{N^{\epsilon/10}}}$ (on the same probability space as $W_N$), that have marginal distributions of standard i.i.d. Gaussians, and for any $D > 0$ we have, \begin{equation}
		\max_{\substack{1 \leq i \leq N \\ 1 \leq \alpha \leq \floor{N^{\epsilon/10}}}}|N ((u_{\alpha}(W_N))_{i})^2 - y_{i \alpha}^2| \leq N^{-1/8 + \epsilon},
		\label{iid_good_approx}
	\end{equation} with probability $\geq 1 - N^{-D}$. 

	Now we are ready to show \eqref{gumbel_convergence_result}. Let $\widehat{F}$ be a bounded smooth function with all derivatives globally bounded. We use Theorem \ref{main_result} with $F(\cdot) = \widehat{F}(\frac{1}{2} (\cdot - b_{\#_{J}}))$, $K \equiv 1$, $k_1 \equiv 1$. Since $K$ is bounded, \eqref{F_deriv_assumptions_eq} is satisfied. So from \eqref{main_result_eq} we have that \begin{equation}
		\lim_{N \to \infty} \abs{\EX\Big[\widehat{F}\Big(\frac{1}{2}((N|\inner{u_{\alpha}, q_i}|^2)_{(1)} - b_{\#_{J}})\Big)\Big] - \EX\Big[\widehat{F}\Big(\frac{1}{2}((N|\inner{u_{\alpha}(W_N), q_i}|^2)_{(1)} - b_{\#_{J}})\Big)\Big]} = 0.
		\label{main_proof_eq1}
	\end{equation} Next, we note that from the invariance of the Haar measure we have,
	 \begin{equation}
	\EX\Big[\widehat{F}\Big(\frac{1}{2}((N|\inner{u_{\alpha}(W_N), q_i}|^2)_{(1)} - b_{\#_{J}})\Big)\Big] = \EX\Big[\widehat{F}\Big(\frac{1}{2}((N|\inner{u_{\alpha}(W_N), e_i}|^2)_{(1)} - b_{\#_{J}})\Big)\Big].
	\label{main_proof_eq2}
	\end{equation} Now we estimate the $u_{\alpha}(W_N)$ by i.i.d. Gaussians. Applying \eqref{iid_good_approx}, and using Lipschitz continuity of $\widehat{F}$ we have that,
	\begin{equation}
		\lim_{N \to \infty} \abs{\EX\Big[\widehat{F}\Big(\frac{1}{2}((N|\inner{u_{\alpha}(W_N), e_i}|^2)_{(1)} - b_{\#_{J}})\Big)\Big] - \EX\Big[\widehat{F}\Big(\frac{1}{2}((y_{i \alpha}^2)_{(1)} - b_{\#_{J}})\Big)\Big]} = 0, 
		\label{main_proof_eq3}
	\end{equation} where $(y_{i \alpha}^2)_{(1)}$ denotes $\max_{(\alpha, i, i) \in J} y_{i \alpha}^2$. Finally, by choosing a sequence of $\widehat{F}$ that approximates the indicator $\indicator (\cdot \leq x)$, we may combine \eqref{main_proof_eq1} -- \eqref{main_proof_eq3} with \eqref{iid_gumbel_conv} to obtain \eqref{gumbel_convergence_result}. 

	Next we will show the intermediate fluctuation result \eqref{intermediate_fluctuation_universality}. We begin by noting that for i.i.d. standard Gaussians $(y_i)_{i=1}^{\#_J}$ we have that for $k \to \infty$ and $k / \#_J \to 0$ as $N \to \infty$, 
	\begin{equation}
	\lim_{N \to \infty} \Prob\left(\frac{(y_i^2)_{(k)} - d_{\#_J}}{c_{\#_J}} \leq x\right) = \Phi(x). 
	\label{iid_intermediate_result}
	\end{equation} 
	Similar to the $\max$ case, we again let $\widehat{F}$ be a bounded smooth function with all derivatives globally bounded. We then use Theorem \ref{main_result} with $F(\cdot) = \widehat{F}(\frac{1}{c_{\#_J}}(\cdot - d_{\#_J}))$, $K \equiv 1$, and $k_1 = k$. Using the formulas,
	\begin{equation*}
		f_{\chi^2(1)}(x) = \frac{1}{\sqrt{2 \pi x}} \e^{-x/2}, \quad F_{\chi^2(1)}(x) = 2 \Phi(\sqrt{x}) - 1, \quad x > 0,
	\end{equation*}	
	along with the standard limit, \begin{equation*}
		\lim_{x \to \infty} \frac{x (1 - \Phi(x))}{\Phi'(x)} = 1,
	\end{equation*} one can derive the following asymptotics for $c_{\#_J}$,
	 \begin{equation}
		\lim_{N \to \infty} \frac{\frac{2}{\sqrt{k}}}{c_{\#_J}} = 1. 
		\label{c_asymptotics}
	\end{equation} From \eqref{c_asymptotics}, $K = 1$, and that we restricted $k \leq N^{\epsilon/2}$, we see that $F$ satisfies \eqref{F_deriv_assumptions_eq}. Then we can just proceed as in \eqref{main_proof_eq1}--\eqref{main_proof_eq3} but use \eqref{iid_intermediate_result} instead of \eqref{iid_gumbel_conv}, and in the end choose a sequence of $\widehat{F}$ approximating $\indicator(\cdot \leq x)$. This finishes the proof of \eqref{intermediate_fluctuation_universality}.  

	The remainder is to show \eqref{point_process_convergence_result}. We show \eqref{point_process_convergence_result} under the additional assumption that $g$ is smooth. The general case is then obtained by a standard approximation argument. We rewrite the Laplace functional of \begin{equation*}
		\sum_{(\alpha, i, i) \in J} \delta_{\frac{1}{2}(N |\inner{u_{\alpha}, q_i}|^2 - b_{\#_{J}})},
	\end{equation*} as,
 \begin{equation*}
		\EX\Bigg[\exp \Big(- \sum_{(\alpha, i, i) \in J} g\Big(\frac{1}{2}(N |\inner{u_{\alpha}, q_i}|^2 - b_{\#_{J}})\Big) \Big)\Bigg] = \EX\Bigg[\exp \Big(- \sum_{k = 1}^{\#_{J}} g\Big(\frac{1}{2}((N |\inner{u_{\alpha}, q_i}|^2)_{(k)} - b_{\#_{J}})\Big) \Big)\Bigg]. 
	\end{equation*} Theorem \ref{main_result} cannot be applied directly to the above right side, since \eqref{F_deriv_assumptions_eq} is not satisfied. We shall show that we can approximate the Laplace functional as, 
	\begin{equation}
		\begin{multlined}
		\EX\Bigg[\exp \Big(- \sum_{k = 1}^{\#_{J}} g\Big(\frac{1}{2}((N |\inner{u_{\alpha}, q_i}|^2)_{(k)} - b_{\#_{J}})\Big) \Big)\Bigg] \\ 
		= \EX\Bigg[\exp \Big(- \sum_{k = 1}^{\min(\#_{J}, \floor{\log N })} g\Big(\frac{1}{2}((N |\inner{u_{\alpha}, q_i}|^2)_{(k)} - b_{\#_{J}})\Big) \Big)\Bigg] + O(N^{-c_1}),
		\end{multlined}
		\label{laplace_truncation}
	\end{equation} 
	for some constant $c_1 > 0$. Note that if $\#_{J} \leq \floor{\log N}$, we have obvious equality in \eqref{laplace_truncation}, so assume for the rest of this argument that $\#_{J} > \floor{\log N}$. Let $\widehat{k} := \min(\#_J, \floor{\log N}) + 1$, and $\widehat{C}$ be a constant such that $\supp g \subset (-\widehat{C}, \infty)$. In the event that $\frac{1}{2}((N |\inner{u_{\alpha}, q_i}|^2)_{(\widehat{k})} - b_{\#_{J}}) \leq - \widehat{C}$, we have exact equality in \eqref{laplace_truncation}. Hence, we need to show that there exists a $c_1 > 0$ such that 
	\begin{equation}
		\Prob \Big( (N |\inner{u_{\alpha}, q_i}|^2)_{(\widehat{k})} > - \widehat{C} + b_{\#_J} \Big) \leq N^{-c_1}. 
		\label{sub_goal}
	\end{equation} To accomplish this, apply Theorem \ref{main_result} with $K \equiv 1$, $k_1 = \widehat{k}$, and $F$ being a smooth bump function that satisfies 
	\begin{equation*}
		\begin{cases}
			F(x) = 0, & \text{for } x \leq - 1 - 2 \widehat{C} + b_{\#_J} , \\
			0 \leq F(x) \leq 1, & \text{for } -1 - 2 \widehat{C} + b_{\#_J}  \leq x \leq- 2 \widehat{C} + b_{\#_J} , \\
			F(x) = 1, & \text{for } x \geq - 2 \widehat{C} + b_{\#_J} .
		\end{cases}
	\end{equation*}
	Using this notation and \eqref{main_result_eq} we obtain that, 
	\begin{equation}
		\begin{multlined}
		\Prob \Big( (N |\inner{u_{\alpha}, q_i}|^2)_{(\widehat{k})} > - 2 \widehat{C} + b_{\#_J} \Big) \\ 
		\leq \EX[F((N |\inner{u_{\alpha}, q_i}|^2)_{(\widehat{k})})] \\ 
		\leq \EX[F((N |\inner{u_{\alpha}(W_N), q_i}|^2)_{(\widehat{k})})] + N^{-\epsilon/2}. 
		\end{multlined}
		\label{main_proof_eq4}
	\end{equation} 
	Next we repeat the argument of estimating the $\GOE$ eigenvector entries with i.i.d. Gaussians, \eqref{iid_good_approx}. Let $\mathcal{N}_{J}$ denote the binomial point process,
	\begin{equation*}
		\mathcal{N}_{J} := \sum_{(\alpha, i, i) \in J} \delta_{\frac{1}{2}(y_{i \alpha}^2 - b_{\#_{J}})}.
	\end{equation*}	
	By using \eqref{iid_good_approx} we may continue the computations in \eqref{main_proof_eq4} to obtain, 
	\begin{equation*}
		\begin{multlined}
			\Prob \Big( (N |\inner{u_{\alpha}, q_i}|^2)_{(\widehat{k})} > - 2 \widehat{C} + b_{\#_J} \Big) \\ 
			\leq \EX[F((y_{i \alpha}^2)_{(\widehat{k})})] + O(N^{-\epsilon/2}) \\ 
			\leq \Prob\Big(\mathcal{N}_J(-(1 + \widehat{C}), \infty) \geq \widehat{k}\Big) + O(N^{-\epsilon/2})
		\end{multlined}
	\end{equation*}
	This last term may be bounded using $\widehat{k} \geq \log N$ and Bennett's inequality (see Theorem 2.9 of \cite{concentration_inequalities}). This yields, 
	\begin{equation}
		\Prob\Big(\mathcal{N}_J(-(1 + \widehat{C}), \infty) \geq \widehat{k}\Big) \leq O(N^{-1}).
		\label{bennett_bound}
	\end{equation}
	Hence, we have shown that \eqref{laplace_truncation} is true for $c_1 = \epsilon/3$. 	

	Equipped with \eqref{laplace_truncation} we can use Theorem \ref{main_result} with $K = \min(\#_{J}, \floor{\log N})$,  $k_l = l$ for $1 \leq l \leq K$, and \begin{equation*}
	F(x_1, \ldots, x_K) = \exp\left( - \sum_{l = 1}^K g\left(\frac{1}{2}(x_l - b_{\#_J})\right) \right).	
	\end{equation*} Since $K \leq \log N$ and $F$ is compactly supported, \eqref{F_deriv_assumptions_eq} is satisfied. Therefore, we have (with a potentially smaller $\delta$) that, \begin{equation*}
		\begin{multlined}
		\EX\Bigg[\exp \Big(- \sum_{k = 1}^{\#_{J}} g\Big(\frac{1}{2}((N |\inner{u_{\alpha}, q_i}|^2)_{(k)} - b_{\#_{J}})\Big) \Big)\Bigg] \\ 
		=  \EX\Bigg[\exp \Big(- \sum_{k = 1}^{\min(\#_{J}, \floor{\log N})} g\Big(\frac{1}{2}((N |\inner{u_{\alpha}(W_N), q_i}|^2)_{(k)} - b_{\#_{J}})\Big) \Big)\Bigg] + O(N^{-\epsilon/2}).
		\end{multlined}
	\end{equation*} Again, estimating the eigenvector entries of $W_N$ by i.i.d. Gaussians, \eqref{iid_good_approx}, we obtain,
	\begin{equation*}
		\begin{multlined}
		\EX\Bigg[\exp \Big(- \sum_{k = 1}^{\#_{J}} g\Big(\frac{1}{2}((N |\inner{u_{\alpha}, q_i}|^2)_{(k)} - b_{\#_{J}})\Big) \Big)\Bigg] \\ 
		 =  \EX\Bigg[\exp \Big(- \sum_{k = 1}^{\min(\#_{J}, \floor{\log N})} g\Big(\frac{1}{2}((y_{i \alpha}^2)_{(k)} - b_{\#_{J}})\Big) \Big)\Bigg] + O(N^{-\epsilon/2}). 
		\end{multlined}
	\end{equation*} Finally, we may use \eqref{bennett_bound} to remove the truncation, and conclude the proof by using \eqref{iid_laplace_conv}. 
\end{proof}

Finally, we present the proof of Theorem \ref{no_gap_deloc_prop}.

\begin{proof}[Proof of Theorem \ref{no_gap_deloc_prop}]
	Let $f$ be an even smooth function that is non-decreasing for $x \geq 0$, and satisfies,
	\begin{equation}
		f(x) = \begin{cases}
			0, & \text{for } 0 \leq x \leq 1, \\
			1, & \text{for } 2 \leq x.
		\end{cases}
	\end{equation} Further, recall the constant $\delta >0$ from Theorem \ref{main_result} and assume that $\epsilon < \delta$. Let $$f_{\mathfrak{c}, N}(x) := f( N^{\epsilon/4} (x - u^{\star}_{\mathfrak{c}})). $$
	We shall prove \eqref{order_statistic_concentration} by showing that,
	\begin{equation}
		\Prob\Big(\abs{(N\abs{\inner{u_{\alpha}, q_i}}^2)_{\ceil{\mathfrak{c} N}} - u^{\star}_{\mathfrak{c}}} \geq 2 N^{-\epsilon/4} \Big) = O(N^{-\epsilon/2}).
		\label{goal_eq_in_proof}
	\end{equation} We bound the probability as follows,
	\begin{equation*}
		\begin{split}
			\Prob\Big(\abs{(N\abs{\inner{u_{\alpha}, q_i}}^2)_{\ceil{\mathfrak{c} N}} - u^{\star}_{\mathfrak{c}}} \geq 2 N^{-\epsilon/4} \Big) \leq \EX[f_{\mathfrak{c}, N}((N\abs{\inner{u_{\alpha}, q_i}}^2)_{\ceil{\mathfrak{c} N}})].
		\end{split}
	\end{equation*} On this last expectation we use Theorem \ref{main_result} with $K = 1$, $k_1 = \ceil{\mathfrak{c} N}$ and $F = f_{\mathfrak{c}, N}$. Note that by construction, $f_{\mathfrak{c}, N}$ satisfies the assumptions \eqref{F_deriv_assumptions_eq}. Hence, we obtain, 
	\begin{equation*}
		\Prob\Big(\abs{(N\abs{\inner{u_{\alpha}, q_i}}^2)_{\ceil{\mathfrak{c} N}} - u^{\star}_{\mathfrak{c}}} \geq 2 N^{-\epsilon/4} \Big) \leq \EX^{\GOE}[f_{\mathfrak{c}, N}((N\abs{\inner{u_{\alpha}, q_i}}^2)_{\ceil{\mathfrak{c} N}})] + O(N^{-\epsilon/2}).
	\end{equation*}  Next we use the same approximation of the eigenvector entries by i.i.d. Gaussians $(y_{i \alpha})_{i=1}^N$ as in \eqref{iid_good_approx}. Exactly as in the proof of Theorem \ref{gumbel_convergence_prop} we use the Haar invariance for the eigenvectors of the $\GOE$ to replace $\inner{u_{\alpha}, q_i}$ by $\inner{u_{\alpha}, e_{i}}$, and then use that the derivative of $f$ is bounded by $O(N^{\epsilon/2}) \ll N^{1/8 - 3\epsilon/2}$ to see that 
	\begin{equation*}
		\Prob\Big(\abs{(N\abs{\inner{u_{\alpha}, q_i}}^2)_{\ceil{\mathfrak{c} N}} - u^{\star}_{\mathfrak{c}}} \geq 2 N^{-\epsilon/4} \Big) \leq \EX[f_{\mathfrak{c}, N}((y_{i \alpha}^2)_{\ceil{\mathfrak{c} N}})] + O(N^{-\epsilon/2}).
	\end{equation*} Finally, we see that this last expectation can be bounded as,
	\begin{equation*}
		 \EX[f_{\mathfrak{c}, N}((y_{i \alpha}^2)_{\ceil{\mathfrak{c} N}})]	\leq \Prob(\abs{(y_{i \alpha}^2)_{\ceil{\mathfrak{c} N}} - u^{\star}_{\mathfrak{c}}} \geq N^{-\epsilon/4}).
	\end{equation*} But $(y_{i \alpha}^2)_{\ceil{\mathfrak{c} N}}$ is the sample quantile of i.i.d. $\chi^2(1)$-variables. In particular, the pdf of $\chi^2(1)$ is bounded from below in a compact neighborhood around the quantile $u_{\mathfrak{c}}^{\star}$, which implies that $(y_{i \alpha}^2)_{\ceil{\mathfrak{c} N}}$ concentrates at a rate $N^{-1/2}$ around $u^{\star}_{\mathfrak{c}}$. So we even have that $\Prob(\abs{(y_{i \alpha}^2)_{\ceil{\mathfrak{c} N}} - u^{\star}_{\mathfrak{c}}} \geq N^{-\epsilon/4}) \leq N^{-D}$ for any large $D > 0$. This may be shown, \eg by using Hoeffding's inequality. With this we have shown \eqref{goal_eq_in_proof} and have finished the proof. 

	Finally, \eqref{limit_of_smallest_sub_vector} follows from noting that $$\inf_{\substack{I \subset [[1, N]] \\ \#_{I} = \floor{(1-\mathfrak{c})N}}} \sum_{i \in I} \abs{\inner{u_{\alpha}, q_i}}^2 = \sum_{k = N - \floor{(1 - \mathfrak{c})N} + 1}^N (\abs{\inner{u_{\alpha}, q_i}}^2)_{(k)}.$$
	By using that the summands are non-increasing one may estimate the sum from above and below by picking $N^{\epsilon/4}$ indices in a uniform grid of $[[N - \floor{(1 - \mathfrak{c})N} - 1, N]]$. Then, using \eqref{order_statistic_concentration} both the upper and lower bound yields a Riemann sum for the integral in \eqref{limit_of_smallest_sub_vector}. 
\end{proof}

\section{Green function comparison}

 The Green function of $H$ is defined as $G(z) := (H - z I)^{-1}$. For two vectors $u, v \in \mathbb{R}^N$ we denote $G_{uv} := \inner{u, G v}$. When we write $G_{ua}$, where $a \in [[1, N]]$ is an index we will understand this to mean $G_{ua} = G_{u e_a}$, where $e_a$ is the $a$-th standard basis vector. 

Following \cite{knowles_yin_2013} we can estimate the entries of $Q u_\alpha$ as a functional of the Green function as follows. 

\begin{lemma}[Lemma 3.2 in \cite{knowles_yin_2013}]
	Let $\epsilon > 0$, and let $\alpha$ be an index $< A := \floor{N^{ \epsilon / 10}}$ and $\varrho = \varrho_{\alpha} : \mathbb{R} \to \mathbb{R}^+$ a smooth cutoff function concentrated around $\alpha - 1$, satisfying \begin{equation*}
		\varrho(x) = 1 \quad \mathrm{if} \quad \abs{x - \alpha + 1} \leq 1/3, \qquad \varrho(x) = 0 \quad \mathrm{if} \quad \abs{x - \alpha + 1} > 2/3.
	\end{equation*} Let \begin{equation*}
		E_L := -2 - 2N^{-2/3 + \epsilon}, \quad E^- := E - N^{2 \epsilon} \eta,
	\end{equation*} where \begin{equation}
		\eta := N^{-2/3 - 6 \epsilon}, \qquad \tilde{\eta} := N^{-2/3 - 14 \epsilon} . 
		\label{eta_defs}
	\end{equation} Further, let $I := [-2 - N^{-2/3+ \epsilon}, -2 + N^{-2/3+ \epsilon}]$. 
	Then there exists $\epsilon_0$ such that for $\epsilon \in (0, \epsilon_0)$, it holds for $N \geq N_0(\epsilon)$ with probability $\geq 1 - N^{- \epsilon}$, that \begin{equation}
	\max_{\alpha < A} \max_{i,j} \left| N \inner{u_{\alpha}, q_i} \inner{u_{\alpha}, q_j}  -  \frac{N}{\pi} \int_I 
	\im G_{q_i q_j}(E + \ii \eta) \varrho_{\alpha}[\Tr(\indicator_{[E_L, E^-]} \star \; \theta_{\tilde{\eta}})(H)] \di E \right| 
	\leq C N^{-\epsilon},
	\label{entry_size_functional}
	\end{equation} with the constant $C$ not depending on $\epsilon$, and $\star$ denotes convolution. 
	\label{green_function_estimate_lemma}
\end{lemma}

We introduce the shorthand notation for the quantity in \eqref{entry_size_functional},
\begin{equation}
	\mathcal{X}_{ij}^{(\alpha)} := \frac{N}{\pi} \int_{I} \im G_{q_i q_j}(E + \ii \eta) \varrho_{\alpha}[\Tr (\chi \star \theta_{\tilde{\eta}})(H)] \di E.
	\label{cal_X_def}
\end{equation} In this notation, Lemma \ref{green_function_estimate_lemma} states that
\begin{equation}
    \max_{(\alpha, i, j) \in J} \left| N \inner{u_\alpha, q_i} \inner{u_\alpha, q_j} - \mathcal{X}_{ij}^{(\alpha)} \right| \le C N^{-\epsilon}, 
	\label{green_function_estimate_simpler_notation}
\end{equation} holds with probability $\geq 1 - N^{-\epsilon}$. 
We also introduce, 
\begin{equation*}
	X := (\mathcal{X}_{ij}^{\alpha})_{(\alpha, i, j) \in J},
\end{equation*}
as the vector of length $\#_{J}$ holding all $\mathcal{X}_{ij}^{(\alpha)}$.

Next we wish to interpolate our Wigner matrix with a Gaussian matrix. This is done in the following Theorem, which is our main technical result. Our main result, Theorem \ref{main_result} follows by using Theorem \eqref{main_gfct} together with a suitable chosen smooth estimate of order statistics, see Section \ref{smoothed_order_statistics_section} for details.

\begin{theorem}

	Let $\widetilde{\func}_{\tau} = (\func_{\tau}^1, \dots, \func_{\tau}^K) : \R^{\#_{J}} \to \R^K$ be an $N$-dependent family of smooth functions parametrized by $\tau > 0$, whose components $\func_{\tau}^l, l = 1, \dots, K$ each satisfy the bound, 
	\begin{align}
	\sum_{1 \leq i_1, \cdots, i_k \leq \#_{J}} \abs{\partial_{i_1} \dots \partial_{i_k} \func_{\tau}^l(x)} \leq \frac{C_k}{\tau^{k-1}}, \label{func_assumptions}
	\end{align} 
	for all $k \in \mathbb{N}$, and all $x \in \R^{\#_J}$, where $C_k$ are $k$-dependendent constants. 

	Recall $\epsilon_0 > 0$ from Lemma \ref{green_function_estimate_lemma}. There exists $\epsilon_1$ with $0 < \epsilon_1 \leq \epsilon_0$ such that for $\epsilon \in (0, \epsilon_1)$, the following holds: 
	
	Let $F : \R^K \to \R$ be a smooth function such that for each $k \in \mathbb{N}$, there is an $N_0(k)$ such that for all $N \geq N_0(k)$ and all $x \in \R^K$, 
	\begin{align}
		\label{deriv_F_bound_assump}
	\sum_{1 \leq i_1, \dots, i_k \leq K} \abs{\partial_{i_1} \dots \partial_{i_k} F(x)} \leq N^{\Fasump k \epsilon}. 
	\end{align} Further, choose $\tau := N^{-\epsilon}$. Then for $N \geq N_0(\epsilon, (C_k)_{k=1}^{\infty})$, 
	\begin{equation}
	\abs{(\EX - \EX^{\GOE})[(F \circ \widetilde{\func}_{\tau}) (X)]}	\leq N^{-1/9}.
	\label{main_gfc_theorem_eq}
	\end{equation}
	\label{main_gfct}
\end{theorem}

\begin{remark}
	The bound $N^{-1/9}$ in \eqref{main_gfc_theorem_eq} is not optimal. Our proof implies that the inequality holds for $N^{-1/9}$ replaced by $N^{-1/6 + c}$, for any positive constant $c > 0$. However, the bottleneck error in Theorem \ref{main_result} is coming from the approximation error of the eigenvector entries in Lemma \ref{green_function_estimate_lemma}. 
\end{remark}

The proof outline is similar to the Green function comparison in \cite{schnelli_xu_generalized_22}. We shall first compare our generalized Wigner matrix $H$ to a Gaussian matrix with the same variance profile as $H$, and then compare the Gaussian matrix with a custom variance profile to the $\GOE$ to obtain \eqref{main_gfc_theorem_eq}. A more detailed overview is presented below in Section \ref{comparison_outline_section}. Compared to the Green function comparison in \cite{schnelli_xu_generalized_22}, Theorem \ref{main_gfct} introduces substantial technical difficulties. Since $K$ and $\#_J$ can be large (of order $N^2$) the number of terms explodes if they are not tracked effectively. The general orthogonal basis $(q_i)_{i=1}^N$ also complicates the proof. In particular, the local law bounds provided below in Lemmas \ref{local_law_bound_lemma} and \ref{step2_local_law_bound_lemma} require considerable work, partially since the functional $\widetilde{\func}_\tau$ prohibits us from using the standard Ward identity-bound.    

\subsection{Preliminaries and notation}

First we gather some information regarding, $\msc$, the Stieltjes transform of the semicircle law, which can be defined as the unique analytic solution $\C^+ \to \C^+$ to the equation, 
\begin{equation*}
	1 + z \msc(z) + (\msc(z))^2 = 0
\end{equation*}

\begin{lemma}
	The imaginary part of the Stieltjes transform of the semicircular law satisfies
	\begin{equation}
		\abs{\im \msc(z)} \sim \begin{cases}
			\sqrt{\kappa + b}, & \mathrm{if } a \in [-2, 2], \\
			\frac{b}{\sqrt{\kappa + b}}, & \mathrm{otherwise},
		\end{cases}
		\label{msc_imaginary_bounds}
	\end{equation}
	uniformly in $z \in \{a + \ii b : \abs{a} \le 5, 0 < b \le 10\}$, with $\kappa := \min\{\abs{a - 2}, \abs{a + 2}\}$. Moreover, $\abs{\msc(z)} \le 1$ holds on the same spectral domain.
\end{lemma}

We will use the following isotropic local law from \cite{bloemendal_erdos_knowles_yau_yin_2014}.

\begin{lemma}[Theorem 2.12 in \cite{bloemendal_erdos_knowles_yau_yin_2014}]
Fix a small $\omega > 0$, and define the domain,
\begin{equation*}
	\mathbb{S}  = \mathbb{S}(\omega, N) := \{z = a + \ii b \in \mathbb{C}: \abs{a} \leq \omega^{-1}, N^{-1 + \omega} \leq b \leq \omega^{-1}  \}.
\end{equation*} Then for any deterministic unit vectors $u, v \in \mathbb{R}^{N}$, 
\begin{equation}
	\abs{\inner{u, G(v)} - \inner{u, v} \msc(z)} \prec \sqrt{\frac{\im \msc(z)}{N b}} + \frac{1}{N b},
	\label{local_law}
\end{equation} uniformly in $z \in \mathbb{S}$. 
\end{lemma}

We shall frequently use cumulant expansions for computing and estimating expectations. 

\begin{lemma}[Lemma 3.1 in \cite{cumulant_expansion_ref}]
	Let $h$ be a random variable with finite moments. The $k$-th cumulant of $h$, denoted by $c^{(k)}$, is defined as \begin{equation*}
		c^{(k)} := (-\ii)^k \left. \left(\frac{\di^k}{\di t^k} \log \EX[\e^{\ii t h}]\right) \right|_{t=0}.
	\end{equation*}
	Let $f : \R \to \C$ be a smooth function. Then for any fixed $l \in \mathbb{N}$ we have \begin{equation}
		\EX [h f(h)] = \sum_{k=1}^l \frac{1}{k!} c^{(k+1)} \EX [f^{(k)}(h)] + R_{l+1},
		\label{cumulant_expansion_formula}
	\end{equation} provided that all the expectations in \eqref{cumulant_expansion_formula} exist. The error term satisfies \begin{equation}
		\abs{R_{l+1}} \leq C_l \EX [\abs{h}^{l+1}] \sup_{\abs{x} \leq M} \abs{f^{(l)}(x)} + C_l \EX \left[ \abs{h}^{l+1} \indicator_{\{\abs{h} > M\}} \right] \norm{f^{(l)}}_{\infty},
		\label{cumulant_expansion_error_bound}
	\end{equation} where $M > 0$ is an arbitrary fixed cutoff and $C_l$ is a constant depending only on $l$.
	\label{cumulant_expansion_lemma}
\end{lemma}

\subsection{Comparison outline and proof of Theorem \ref{main_gfct}}
\label{comparison_outline_section}
The proof of Theorem \ref{main_gfct} is based on interpolating between three different ensembles, in the end connecting a generalized Wigner matrix and the $\GOE$. 

Recall the variance profile $V$ of $H$ from Definition \ref{general_wigner_def}. Let $W^V = (w^{(v)}_{ij})_{1 \leq i, j \leq N}$ be a symmetric random matrix with independent Gaussian entries (up to symmetry) that has $V$ as variance profile matrix. That is, $\EX[((W^V)_{ij})^2] = V_{ij}$. For ease of notation we introduce, 
\begin{equation}
	\widetilde{F} := F \circ \widetilde{\func}_{\tau},
	\label{F_tilde_def}
\end{equation} to denote the composition inside the expectation in \eqref{main_gfc_theorem_eq}. The comparison steps are as follows:

	\noindent\textbf{Step 1: Comparing $H$ to Gaussian with same variance profile}\\
	 In this step the goal is to show the following Lemma.
	\begin{lemma}
	Recall $\widetilde{F}$ from \eqref{F_tilde_def}, $X = (\mathcal{X}_{ij}^{(\alpha)})_{(\alpha, i, j) \in J}$, $\epsilon_0$ from Lemma \ref{green_function_estimate_lemma}, $\tau = N^{-\epsilon}$, and the assumptions on $\widetilde{\mathcal{L}}_{\tau}$ in \eqref{func_assumptions} from Theorem \ref{main_gfct}. Then there exist $\epsilon_2$ with $0 < \epsilon_2 < \epsilon_0$ such that for $\epsilon \in (0, \epsilon_2)$ the following holds: 

	Assume that $F$ satisfies \eqref{deriv_F_bound_assump}. Then for sufficiently large $N$,
	\begin{equation}
		\abs{\EX^H[\widetilde{F}(X)] - \EX^{W^V}[\widetilde{F}(X)]} \leq N^{-1/8}. 
		\label{first_comparison_step}
	\end{equation} 
	\label{first_comprison_step_lemma}
	\end{lemma}

	To show Lemma \ref{first_comprison_step_lemma}, we interpolate between $W^V$ and $H$ with the flow,
	\begin{equation}
		H^{(1)}(t) := \e^{-\frac{t}{2}} W^V + \sqrt{1 - \e^{-t}} H, \quad t \in \mathbb{R}^+. 
		\label{interpolation_step_1}
	\end{equation} The proof is presented below in Section \ref{first_comparison_section}.

\noindent\textbf{Step 2: Comparing Gaussian with variance profile to $\GOE$}\\
	For the second step we first need to handle the variance on the diagonal of the $\GOE$. As in \cite{schnelli_xu_generalized_22} we define \begin{equation}
		\widetilde{V} := ((1 + \delta_{ij}) V_{ij})_{1 \leq i, j \leq N}. 
	\end{equation} Next we show that this small perturbation from $V$ to $\widetilde{V}$ does not significantly affect our observable: 
	\begin{lemma}
	Recall $\widetilde{F}$ from \eqref{F_tilde_def}, $X = (\mathcal{X}_{ij}^{(\alpha)})_{(\alpha, i, j) \in J}$, $\epsilon_0$ from Lemma \ref{green_function_estimate_lemma}, $\tau = N^{-\epsilon}$, and the assumptions on $\widetilde{\mathcal{L}}_{\tau}$ in \eqref{func_assumptions} from Theorem \ref{main_gfct}. Then there exist $\epsilon_3$ with $0 < \epsilon_3 < \epsilon_0$ such that for $\epsilon \in (0, \epsilon_3)$ the following holds: 

	Assume that $F$ satisfies \eqref{deriv_F_bound_assump}, and choose $\tau = N^{-\epsilon}$. Then for sufficiently large $N$,
	\begin{equation}
		\abs{\EX^{\widetilde{V}}[\widetilde{F}(X)] - \EX^V[\widetilde{F}(X)]} \leq N^{-1/8}.
		\label{diagonal_variance_perturbation_comparison}
	\end{equation} 
	\label{diagonal_variance_perturbation_lemma}
	\end{lemma} 
	The interpolation, 
	\begin{equation}
		H^{(3)}(t) := \e^{-\frac{t}{2}} W^V + \sqrt{1 - \e^{-t}} W^{\widetilde{V}}, \quad t \in \mathbb{R}^+,
		\label{diagonal_var_interpolation_eq}
	\end{equation} is used to prove Lemma \ref{diagonal_variance_perturbation_lemma}. The proof is relatively simple compared to the other interpolation steps. So we omit the full details and sketch the proof in Appendix \ref{easy_interpolation_lemma_proof_section}. 
	
	The second major step is to compare $W^{\widetilde{V}}$ with the $\GOE$. We wish to show the following Lemma.

	\begin{lemma}
	Recall $\widetilde{F}$ from \eqref{F_tilde_def}, $X = (\mathcal{X}_{ij}^{(\alpha)})_{(\alpha, i, j) \in J}$, $\epsilon_0$ from Lemma \ref{green_function_estimate_lemma}, $\tau = N^{-\epsilon}$, and the assumptions on $\widetilde{\mathcal{L}}_{\tau}$ in \eqref{func_assumptions} from Theorem \ref{main_gfct}. Then there exist $\epsilon_4$ with $0 < \epsilon_4 < \epsilon_0$ such that for $\epsilon \in (0, \epsilon_4)$ the following holds: 

	Assume that $F$ satisfies \eqref{deriv_F_bound_assump}, and choose $\tau = N^{-\epsilon}$. Then for sufficiently large $N$,
	\begin{equation}
	\abs{\EX^{\widetilde{V}}[\widetilde{F}(X)] - \EX^{\GOE}[\widetilde{F}(X)]} \leq N^{-1/4}.	
	\label{second_comparison_step}
	\end{equation} 
	\label{gaussian_to_gaussian_lemma}
	\end{lemma} 

	Let $W = (w_{ij})_{1 \leq i, j \leq N}$ be a matrix sampled from the $\GOE$. We show Lemma \ref{gaussian_to_gaussian_lemma} by using the interpolating flow given by \begin{equation}
		H^{(2)}(t) := \e^{-\frac{t}{2}} W + \sqrt{1 - \e^{-t}} W^{\widetilde{V}}, \quad t \in \mathbb{R}^+. 	
		\label{gaussian_to_gaussian_interpolating_flow}
	\end{equation} The proof is presented in Section \ref{second_comparison_section}.

Equipped with the three Lemmas \ref{first_comprison_step_lemma} -- \ref{gaussian_to_gaussian_lemma}, Theorem \ref{main_gfct} follows directly. 

\begin{proof}[Proof of Theorem \ref{main_gfct}]
Note that the assumptions of Lemmas \ref{first_comprison_step_lemma} -- \ref{gaussian_to_gaussian_lemma} are the same as in Theorem \ref{main_gfct}. Hence, the above Lemmas give that for,
\begin{equation*}
\epsilon_1 := \min(\epsilon_2, \epsilon_3, \epsilon_4),
\end{equation*}
we have for $\epsilon \in (0, \epsilon_1)$ and $\tau = N^{-\epsilon}$ that for sufficiently large $N$,
\begin{equation*}
	\abs{\EX^{H}[\widetilde{F}(X)] - \EX^{\GOE}[\widetilde{F}(X)]} \leq N^{-1/8} + N^{-1/8} + N^{-1/4}.
\end{equation*} Hence, see that \eqref{main_gfc_theorem_eq} is satisfied, thereby finishing the proof.
\end{proof}

\subsection{Interpolation preliminaries}

In this subsection we introduce notation for the interpolating flows and present some results we will need later on. Let $H(t) = (h_{ab}(t))_{a, b = 1}^N, t \in \R^+$ denote any smooth interpolating flow (\eg \eqref{interpolation_step_1} or \eqref{gaussian_to_gaussian_interpolating_flow}). Define the time dependent Green function,
\begin{equation*}
	 G = G(t, z) := (H(t) - zI)^{-1}.\end{equation*}

Introduce the set of basis vectors $u, v$ for which $G_{uv}$ is used,
	\begin{equation*}
		\mathbb{B} := \{e_1, \dots, e_N, q_1, \dots, q_N\}.
	\end{equation*} 
	Further, introduce the random control parameter $\Psi = \Psi(t)$ by setting, \begin{equation}
		\begin{split}
			\Psi(t) & := \sup_{\substack{E \in [E_L, -2 + N^{-2/3 + \epsilon}] \\
			u,v \in \mathbb{B} \\
			1 \leq a \neq b \leq N}}
			\max\Big\{\abs{G_{uv}(t, E + \ii \eta) - \inner{u, v} m_{sc}(E + \ii \eta)}, \abs{\im G_{uv}(t, E + \ii \eta)}, \\
			& \quad \quad  2 \abs{G_{ab} (t, E + \ii \tilde{\eta})}, \abs{2 \im G_{uv}(t, E + \ii \tilde{\eta})}, N^{-1/3}, \im \msc(E + \ii \eta)\Big\}.
		\end{split}
		\label{Psi_def}
	\end{equation}
	
	The $H(t)$ will change depending on what interpolation we consider, but the interpolations we consider all satisfy the isotropic local law, implying the following Lemma, 
	\begin{lemma} Let $H(t)$ be the interpolation \eqref{interpolation_step_1}, \eqref{diagonal_var_interpolation_eq}, or \eqref{gaussian_to_gaussian_interpolating_flow}, and recall $\Psi$ from \eqref{Psi_def} and $\eta, \tilde{\eta}$ from \eqref{eta_defs}. Then, uniformly for $t \in \R_{\geq 0}$,
	\begin{align}
		\Psi(t) \prec \frac{1}{N \tilde{\eta}}.
		\label{Psi_bounds}
	\end{align}
		Further, for any (large) $D > 0$ it holds uniformly for $t \in \R^+$ that, \begin{equation}
		\sup_{\substack{E \in [E_L, -2 + N^{-2/3 + \epsilon}] \\
			u,v \in \mathbb{B}}} \max(\abs{G_{uv}(t, E + \ii \eta)}, \abs{G_{uv}(t, E + \ii \tilde{{\eta}})}) \leq 1, 
			\label{abs_bdd_by_one}
	\end{equation} with probability $\geq 1 - N^{-D}$ for $N \geq N_0(D, H)$.
	\label{uniform_local_law}
	\end{lemma}	
	\begin{proof}
		For a fixed time $t$ both \eqref{Psi_bounds} and \eqref{abs_bdd_by_one} follow directly from Lemma \ref{local_law}. The uniformity for $t \in \R^+$ follows from that there exist constants $\Cinf, \Csup$, and $(C_k)_{k=1}^{\infty}$ such that $H(t)$ satisfies the conditions \eqref{mean_field_condition} and \eqref{moment_condition} uniformly in $t$.  
	\end{proof}
	We will use Lemma \ref{uniform_local_law} several times throughout the paper without explicitly referring to it.

	Next we introduce the corresponding time dependent analogues of the observables introduced previously. First, introduce the time dependent version of $\mathcal{X}_{ij}^{(\alpha)}$,
	\begin{align}
		\mathcal{X}_{ij}^{(\alpha)}(t) & := \frac{N}{\pi} \int_{I} \im G_{q_i q_j}(t, E + \ii \eta) \varrho_{\alpha}[\Tr (\chi \star \theta_{\tilde{\eta}})(H(t))] \di E.
	\end{align}		
From here on we will usually omit the arguments of $G$. Further, introduce the notation, \begin{align}
	Y(t) & :=  \int_{E_L}^{E^-} \im \sum_j G_{jj} (t, y + \ii \tilde{\eta}) \di y, \notag \\
	X(t) & := (\mathcal{X}_{ij}^{(\alpha)}(t))_{(\alpha, i, j) \in J}. \notag \\
\end{align} 

From the local law we see that both $\mathcal{X}_{ij}^{(\alpha)}$ and $Y(t)$ are $O(1)$ (up to factors of $N^{C \epsilon}$). This can be seen since the size of the integration interval $I$ is $O(N^{-2/3})$, and both $\im G_{q_i q_j}$ and $\im G_{jj}$ are of size $O(N^{-1/3})$ (up to factors of $N^{C \epsilon}$). 

Using the new notation, the observable we compare for our generalized Wigner matrix and the $\GOE$ in \eqref{main_gfc_theorem_eq} is precisely $\EX[\widetilde{F}(X(t))]$ for $t = 0$ and $t = \infty$, respectively. Our strategy of bounding the difference $|\EX[\widetilde{F}(X(\infty))] -  \EX[\widetilde{F}(X(0))]|$ is to bound the derivative $\abs{\frac{\di}{\di t} \EX[\widetilde{F}(X(t))]}$ and then integrate. 

Hence, we wish to compute the time derivative $\frac{\di}{\di t} \EX[\widetilde{F}(X(t))]$. To do this we collect the time derivatives of the different components, $Y(t)$ and $\mathcal{X}_{ij}^{(\alpha)}(t)$. We recall first the derivative rules for $G_{uv}$ (using the notation $G_{uv} = \inner{u, G v}$): \begin{equation}
	\frac{\partial G_{uv}}{\partial h_{ab}(t)} = - \frac{1}{1 + \delta_{ab}} (G_{ua} G_{bv} + G_{ub} G_{av}).
	\label{derivative_rule_gen}
\end{equation} 
Next, introduce the operator $\Dim$ for evaluating integrals. It is defined through \begin{equation*}
\Dim (f(t, z)) := \im f(t, E^- + \ii \tilde{\eta}) - \im f(t, E_L + \ii \tilde{\eta}). 
\end{equation*} 
Using $\Dim$, we have \begin{equation}
	\begin{split}
	\frac{\partial}{\partial h_{ab}(t)} Y(t) & = \im \int_{E_L}^{E^-} \sum_{j=1}^N \frac{\partial}{\partial h_{ab}(t)} G_{jj}(y) \di y \\
	& = - \frac{2}{1 + \delta_{ab}} \im \int_{E_L}^{E^-} \sum_{j} G_{ja}(y) G_{bj}(y) \di y. \\
	& = -  \frac{2}{1 + \delta_{ab}} \Dim G_{ab}.
	\end{split} 
	\label{Y_deriv_rule}
\end{equation} 
Next we have \begin{equation}
	\begin{split}
	\frac{\partial \mathcal{X}_{ij}^{(\alpha)}(t)}{\partial h_{ab}(t)} & = \frac{N}{\pi} \int_I (\frac{\partial \im G_{q_i q_j}}{\partial h_{ab}(t)}) \varrho_{\alpha}(Y(t)) + \im G_{q_i q_j} \varrho_{\alpha}'(Y(t)) \frac{\partial Y(t)}{\partial h_{ab}(t)} \di E \\
	& = - \frac{2}{1 + \delta_{ab}} \frac{N}{\pi} \int_I (\im G_{q_i a} G_{b q_j}) \varrho_{\alpha}(Y(t)) \di E - \frac{2}{1 + \delta_{ab}} \frac{N}{\pi} \int_I \im G_{q_i q_j} \varrho_{\alpha}'(Y_t)  (\Dim G_{ab})
	\di E.
	\end{split}
	\label{cal_X_deriv_rule}
\end{equation} 
Additionally, we need to compute the result of $\frac{\partial}{\partial h_{ab}(t)}$ acting on $\partial_{\bar{\mathfrak{j}}} \func_{\tau}^{l} (X(t))$ (where $\bar{\mathfrak{j}}$ is a multi-index). In this case we have \begin{equation}
	\frac{\partial}{\partial h_{ab}(t)} \partial_{\bar{\mathfrak{j}}} \func_\tau^l (X(t)) = \sum_{\mathfrak{i} \in J} \partial_{\mathfrak{i}, \bar{\mathfrak{j}}} \func_\tau^l(X(t)) \frac{\partial \mathcal{X}_{\mathfrak{i}}}{\partial h_{ab}(t)},
	\label{func_deriv_rule}
 \end{equation} where $\mathfrak{i}$ is a fresh multi-index $\in J$, and by $\mathcal{X}_{\mathfrak{i}}$ we mean $\mathcal{X}_{ij}^{(\alpha)}$ given by $\mathfrak{i} = (\alpha, i, j)$. By $\mathfrak{i}, \bar{\mathfrak{j}}$ we mean the concatenation, \ie $\partial_{\mathfrak{i}, \bar{\mathfrak{j}}} = \partial_{\mathfrak{i}} \partial_{\bar{\mathfrak{j}}}$. 

On a high level we can now expand $\dt \EX[\widetilde{F}(X(t))]$ as,
\begin{equation}
		\dt \EX[\widetilde{F}(X(t))] = \sum_{a \leq b} \EX\left[\dot{h}_{ab}(t) \frac{\partial \widetilde{F}(X(t))}{\partial h_{ab}(t)}\right] = \sum_{a,b} \frac{1 + \delta_{ab}}{2} \EX \left[\dot{h}_{ab}(t) \frac{\partial \widetilde{F}(X(t))}{\partial h_{ab}(t)}\right].
		\label{basic_time_derivative_computation}
\end{equation} The last equality in \eqref{basic_time_derivative_computation} follows from $H$ being symmetric. The reason for converting $\sum_{a \leq b}$ to $\sum_{a,b}$ is that it slightly simplifies the computations below. Below in Sections \ref{first_comparison_section} and \ref{second_comparison_section} we will cumulant expand the last expression in \eqref{basic_time_derivative_computation} and use the derivative rules introduced above to expand it further.

\section{Smoothed order statistics and proof of Theorem \ref{main_result}}

\label{smoothed_order_statistics_section}
A key component in the proof of Theorem \ref{main_result} is the following smoothed order statistic functionals we introduce below. In previous works on extreme order statistics, notably in \cite{Landon2020,BenigniL.2022Odfg}, the LogSumExp functional with parameter $\tau$, $\lse_{\tau} : \R^n \to \R$, \begin{equation*}
	\lse_{\tau}(x_1,\dots,x_n) := \tau \log \sum_{i=1}^n \e^{x_i / \tau}, 
\end{equation*} was used as a smoothed max to compare extreme order statistics. The $\lse_{\tau}$ satisfies the assumptions \eqref{func_assumptions}, and it is also very close to $\max_{i} x_i$, the error is $O(\tau \log n)$. However, to estimate order statistics that are not the largest (or smallest) the $\lse_{\tau}$ is not very effective. The traditional approach to approximate the $k$-th largest value of a vector $(x_i)_{i=1}^n$ is to consider, \begin{equation}
	\tau \log \sum_{i_1 < \cdots < i_k} \e^{\sum_{j=1}^k x_{i_j}/ \tau} - \tau \log \sum_{i_1 < \cdots < i_{k-1}} \e^{\sum_{j=1}^{k-1} x_{i_j}/ \tau}.
	\label{logsumexp_kth_largest}
\end{equation} The intuition behind the above expression is that the first term estimates the largest sum consisting of $k$ $x_i$'s, which is exactly $x_{(1)} + x_{(2)} + \dots + x_{(k)}$ and the second term estimates the largest sum consisting of $k-1$ $x_i$'s, which is exactly $x_{(1)} + x_{(2)} + \dots + x_{(k-1)}$. So the difference is precisely $x_{(k)}$. This works well for small $k$, but the sums $\sum_{i_1 < \cdots < i_k}$ contains $O(n^k)$ terms, so the error between $x_{(k)}$ and \eqref{logsumexp_kth_largest} becomes $O(\tau k \log n)$. To \eg estimate the median using this approach is therefore not feasible, one would have to choose $\tau \ll n^{-1}$ to make the error small. 

However, we wish to estimate the $k$-th largest entry of a vector in $\mathbb{R}^n$, for any $k$. To accomplish this, we introduce the $k/n$-th quantile of the sigmoid-smoothed empirical cdf, $t^*_{\tau, k} : \mathbb{R}^{n} \to \mathbb{R}$, that approximates the $k$-th largest value $x_{(k)}$ of a vector $\bar{x} \in \mathbb{R}^n$. The same definition appears frequently in the statistics literature, see for instance \cite{sheather_marron_1990, azzalini1981}, where it is used as a quantile estimator, but they consider other kernels as well as the sigmoid we use. We have not found any paper using it to explicitly estimate the $k$-th largest value, and we have not found any version of Lemma \ref{t_star_lemma} in the existing literature. 
\begin{definition}
Let $\sigma : \mathbb{R} \to \mathbb{R}$ be the sigmoid function, $\sigma(x) = 1/(1 + \e^{-x})$, and denote $\bar{x} = (x_1, \dots, x_n)$. Then $t^*_{\tau, k}(\bar{x})$ is defined as the unique $t \in \mathbb{R}$ solving the equation, \begin{equation}
	\sum_{i=1}^n \sigma\Big(\frac{t - x_i}{\tau}\Big) = n - k + \frac{1}{2}. 
	\label{t_star_def}
\end{equation} 
From the implicit function theorem we get that $t^*_{\tau, k} : \mathbb{R}^n \to \mathbb{R}$ is well-defined and $C^{\infty}$, since $\sigma$ is $C^{\infty}$ and strictly increasing.
\end{definition}
The following Lemma states the key properties of $t^*_{\tau, k}$. 
\begin{lemma}
	For any $\bar{x} \in \mathbb{R}^n$ we have, \begin{equation}
	\abs{t^*_{\tau, k}(\bar{x}) - x_{(k)}} \leq \tau \log (2 n). 
	\label{t_star_good_approx_eq}
	\end{equation}	Secondly, for a fixed integer $l$, there exists a constant $C_l$ not depending on $n$ or $k$ such that \begin{equation}
	\sum_{j_1, \dots, j_l} \abs{\frac{\partial^l}{\partial x_{j_1} \cdots \partial x_{j_l}} t^*_{\tau, k}(\bar{x})} \leq \frac{C_l}{\tau^{l-1}},
	\label{t_star_derivative_l1_bound}
	\end{equation} where the indices $j_1, \ldots, j_l$ are iterated from $1$ to $n$.
	\label{t_star_lemma}
\end{lemma}

\begin{proof}
We begin by showing \eqref{t_star_good_approx_eq}. First, assume that $t^*_{\tau, k}(\bar{x}) > x_{(k)} + \tau \log (2 n)$. We will show that this leads to a contradiction. We use \eqref{t_star_def}, which defines $t^*_{\tau, k}$, 
\begin{equation*}
	\begin{split}
\sum_{i=1}^n \sigma \Big(\frac{t^*_{\tau, k} - x_i}{\tau}\Big) & =\sum_{i=1}^n \sigma \Big(\frac{t^*_{\tau, k} - x_{(i)}}{\tau}\Big) \\
& = \sum_{i=1}^{k-1} \sigma \Big(\frac{t^*_{\tau, k} - x_{(i)}}{\tau}\Big) + \sum_{i=k}^n \sigma \Big(\frac{t^*_{\tau, k} - x_{(i)}}{\tau}\Big) \\
& \geq 0 + \sum_{i=k}^n \sigma \Big(\frac{t^*_{\tau, k} - x_{(k)}}{\tau}\Big) \\
& > \sum_{i=k}^n \sigma (\log (2n)). 
\end{split}
\end{equation*} Using that $\sigma(x) \geq 1 - \e^{-x}$ we can continue to lower bound the above expression and get, \begin{equation*}
\begin{split}
\sum_{i=k}^n \sigma (\log (2n)) & \geq \sum_{i=k}^n 1 - \e^{-\log(2n)} \\
& = (n-k+1) \big(1 - \frac{1}{2n}\big) \\
& \geq n - k + \frac{1}{2}. 
\end{split}
\end{equation*} Combining the two computations above we obtain that $\sum_{i=1}^n \sigma \Big(\frac{t^*_{\tau, k} - x_i}{\tau}\Big) > n - k + \frac{1}{2}$, which contradicts \eqref{t_star_def} and shows that $t^*_{\tau, k}(\bar{x}) > x_{(k)} + \tau \log (2 n)$ is impossible. 

The proof in the other direction is analogous. Assume $t^*_{\tau, k}(\bar{x}) < x_{(k)} - \tau \log (2 n)$. Then, \begin{equation*}
	\begin{split}
	\sum_{i=1}^n \sigma \Big(\frac{t^*_{\tau, k} - x_i}{\tau}\Big) & = \sum_{i=1}^k \sigma \Big(\frac{t^*_{\tau, k} - x_{(i)}}{\tau}\Big) + \sum_{i=k+1}^n \sigma \Big(\frac{t^*_{\tau, k} - x_{(i)}}{\tau}\Big) \\
	& \leq \sum_{i=1}^k \sigma \Big(\frac{t^*_{\tau, k} - x_{(i)}}{\tau}\Big) + \sum_{i=k+1}^n \sigma \Big(\frac{t^*_{\tau, k} - x_{(i)}}{\tau}\Big) \\
	& \leq \sum_{i=1}^k \sigma \Big(\frac{t^*_{\tau, k} - x_{(k)}}{\tau}\Big) + n-k \\
	& < k \sigma (-\log (2n)) + n-k \\
	& \leq n-k+\frac{1}{2}. 
	\end{split}
\end{equation*} In the last step we used that $\sigma(x) \leq \e^{x}$. So the proof of \eqref{t_star_good_approx_eq} is done.

Now we show \eqref{t_star_derivative_l1_bound}. Define $\mathcal{F}_{\tau, k} : \mathbb{R}^{n + 1} \to \mathbb{R}$, as $$\mathcal{F}_{\tau, k}(t, \bar{x}) := \sum_{i=1}^n \sigma\Big(\frac{t - x_i}{\tau}\Big) - \big(n - k + \frac{1}{2}\big), $$ by the implicit function theorem, the first order partial derivatives of $t^*_{\tau, k}$ are given by \begin{equation*}
	\frac{\partial}{\partial x_i} t^*_{\tau, k} = \frac{- \big( \frac{\partial}{\partial x_i} \mathcal{F}_{\tau, k}\big) (t^*_{\tau, k}, \bar{x})}{- \big( \frac{\partial}{\partial t} \mathcal{F}_{\tau, k}\big) (t^*_{\tau, k}, \bar{x})} = \frac{\sigma'\Big(\frac{t^*_{\tau, k} - x_i}{\tau}\Big)}{\sum_{j=1}^n \sigma'\Big(\frac{t^*_{\tau, k} - x_j}{\tau}\Big)}.
\end{equation*}
Introduce $z_i :=\frac{t^*_{\tau, k} - x_i}{\tau}$, and $p_i := \frac{\partial}{\partial x_i} t^*_{\tau, k}$. From the definition of $t^{*}_{\tau, k}$ it is easy to check that,
\begin{equation}
	t^{*}_{\tau, k}(x_1 + d, \ldots, x_n + d) = t^{*}_{\tau, k}(x_1, \ldots, x_n) + d.
	\label{shift_invariance}
\end{equation}
From \eqref{shift_invariance} it follows that (note $\abs{p_i} = p_i$ since $p_i$ is non-negative), 
\begin{equation}
	\sum_{i=1}^n |p_i| = 1,
	\label{first_derivative_sum}
\end{equation} since $\sum_{i=1}^n p_i$ is $\sqrt{n}$ times the directional derivative of $t^*_{\tau, k}$ in the direction $(1, \ldots, 1)/\sqrt{n}$.    

We note the derivative, \begin{equation*}
\frac{\partial z_i}{\partial x_l} = \frac{1}{\tau} (p_l - \delta_{il}). 
\end{equation*} Next, note that since $\sigma' = \sigma (1 - \sigma)$, the higher derivatives of $\sigma$ can be expressed as, \begin{equation}
\sigma^{(q+1)}(x) = \sigma'(x) P_{q} (\sigma(x)),
\label{P_q_def}
\end{equation} for a polynomial $P_{q}$ of degree $q$. Next introduce, \begin{equation*}
v_i^{(q)} := \frac{1}{\tau^{q}} P_{q}(\sigma(z_i)).
\end{equation*} Note that $v_i^{(q)}$ does not denote the $q$-th  derivative. Since $0 < \sigma < 1$, we have, \begin{equation}
|v_i^{(q)}| \leq \frac{C_q}{\tau^{q}},
\label{v_bound}
\end{equation} for $C_q := \sup_{x \in [0, 1]} P_q(x)$. Now we are ready to compute the derivatives of $p_i$, \begin{equation}
 \begin{split}
	\frac{\partial}{\partial x_l} p_i & = \frac{\partial}{\partial x_l} \frac{\sigma'(z_i)}{\sum_{j} \sigma'(z_j)} \\
	& = \frac{\sigma'(z_i) P_1(\sigma(z_i)) \frac{\partial z_i}{\partial x_l} \big(\sum_{j=1} \sigma'(z_j)\big) - \sigma'(z_i) \sum_{j} \sigma'(z_j) P_1(\sigma(z_j)) \frac{\partial z_j}{\partial x_l}}{\big(\sum_{j=1} \sigma'(z_j)\big)^2} \\
	& = p_i P_q(\sigma(z_i)) \frac{1}{\tau} (p_l - \delta_{il}) - p_i \sum_j p_j P_1(\sigma(z_j)) \frac{1}{\tau} (p_l - \delta_{jl}) \\
	& = p_i p_l (v_i^{(1)} + v_l^{(1)} - \sum_{j} p_j v_j^{(1)}) - \delta_{il} p_i v_i^{(1)}.
 \end{split}
 \label{p_deriv_rule}
\end{equation} Before computing the derivatives of $v_i^{(q)}$ we need a formula for $P_q'$. We differentiate $\sigma^{(q+1)}(x)/\sigma'(x) = P_q(\sigma(x))$ and obtain, \begin{equation*}
\begin{split}
	(P_q(\sigma(x)))' & = P_q'(\sigma(x)) \sigma'(x) \\
	& = \frac{\sigma^{(q+2)}(x) \sigma'(x) - \sigma^{(q+1)}(x) \sigma'(x) P_1(\sigma(x))}{(\sigma'(x))^2} \\
	& = \frac{\sigma^{(q+2)}(x)}{\sigma'(x)} - \frac{\sigma^{(q+1)}(x)}{\sigma'(x)} P_1(\sigma(x)) \\
	& = P_{q+1}(\sigma(x)) - P_q(\sigma(x)) P_1(\sigma(x)).
\end{split}
\end{equation*}
Now we can compute $\frac{\partial}{\partial x_l} v_i^{(q)}$, \begin{equation}
	\begin{split}
	\frac{\partial}{\partial x_l} v_i^{(q)} &= \frac{\partial}{\partial x_l} \frac{P_q(\sigma(z_i))}{\tau^q} \\	
	& = \frac{1}{\tau^q} P_q'(\sigma(z_i)) \sigma'(z_i) \frac{\partial z_i}{\partial x_l} \\ 
	& = \frac{1}{\tau^q} (P_{q+1}(\sigma(z_i)) - P_q(\sigma(z_i)) P_1(\sigma(z_i))) \frac{1}{\tau} (p_l - \delta_{il}) \\
	& = (p_l - \delta_{il}) (v_i^{(q+1)} - v_i^{(q)} v_i^{(1)}).
	\end{split}
	\label{v_deriv_rule}
\end{equation} 

To finish off the proof, we show that for all $l \geq 1$, 
\begin{equation}
	\sum_{j_1, \dots, j_l} \frac{\partial^l}{\partial x_{j_1} \cdots \partial x_{j_l}} t^*_{\tau, k}(\bar{x}), 
	\label{derivative_tensor_sum}
\end{equation}
 can be expressed as a finite linear combination of terms of the form \begin{equation}
\sum_{(i_1, \dots, i_s) \in [[1, n]]^s}  p_{i_1} \cdots p_{i_s} v_{y_1}^{(q_1)} \cdots v_{y_m}^{(q_m)}, \quad \mathrm{with } \sum_{\mu=1}^m q_{\mu} = l - 1, 
\label{linear_comb_form}
\end{equation} and the formula is independent of $k$, and the $y_1, \dots, y_m$ each represent one of the indices $i_1 , \dots, i_s$. From this combined with \eqref{first_derivative_sum} and \eqref{v_bound}, \eqref{t_star_derivative_l1_bound} follows immediately.  

We show that \eqref{derivative_tensor_sum} can be expressed as a finite linear combination of \eqref{linear_comb_form} by induction. For $l = 1$ we have $\sum_{j_1} \frac{\partial}{\partial x_{j_1}} t^*_{\tau, k}(\bar{x}) = \sum_{j_1} p_{j_1}$. If we assume the claim is true for any $l \geq 1$ we show that it is also true for $l+1$: Apply $\sum_{j_{l+1}} \frac{\partial}{\partial x_{j_{l+1}}}$ to \eqref{linear_comb_form}. Then we see from the derivative rules \eqref{p_deriv_rule} and \eqref{v_deriv_rule} that if $\frac{\partial}{\partial x_{j_{l+1}}}$ acts on a $p_i$, we get four new terms, all of which have one extra $v^{(1)}$. Similarly, if $\frac{\partial}{\partial x_{j_{l+1}}}$ acts on a $v_y^{(q)}$, we see from \eqref{v_deriv_rule} that we again get four new terms, of the form \eqref{linear_comb_form}, two of which contain $v_y^{(q+1)}$ and the other two $v_y^{(q)}v_y^{(1)}$.  
\end{proof}

Equipped with $t^*_{\tau, k}$ and Theorem \ref{main_gfct} we can prove Theorem \ref{main_result} 

\begin{proof}[Proof of Theorem \ref{main_result}]
 Let $\epsilon_1 > 0$ be the constant given by Theorem \ref{main_gfct}, and set $$\delta := \min(\epsilon_1, 1/18).$$ Our goal is now to show \eqref{main_result_eq} for $\epsilon < \delta$.
 
 For $l \in [[1, K]]$ choose, 
 $$\mathcal{L}^l_{\tau} = t^*_{\tau, k_l},$$
 where $t^*_{\tau, k}$ is from \eqref{t_star_def}. By Lemma \ref{t_star_lemma}, these $\mathcal{L}^{\cdot}_{\tau}$ satisfy the assumptions in \eqref{func_assumptions}, and since $\delta \leq \epsilon_1$, $F$ satisfies \eqref{deriv_F_bound_assump}, and have that, 
 \begin{equation}
	\abs{(\EX - \EX^{\GOE})[(F \circ \widetilde{\func}_{\tau}) ((\mathcal{X}_{ij}^{(\alpha)})_{(\alpha, i, j) \in J})]}	\leq N^{-1/9}.
	\label{gfc_repeat}
 \end{equation} 
Denote the $k$-th largest element of $\{\mathcal{X}_{ij}^{\alpha}, (\alpha, i, j) \in J\}$ as $(\mathcal{X}_{ij}^{(\alpha)})_{(k)}$. By a standard ordering inequality and Lemma \ref{green_function_estimate_lemma}, we have uniformly in $k$, that with probability $\geq 1 - N^{-\epsilon}$,
\begin{equation}
	\abs{(\mathcal{X}_{ij}^{(\alpha)})_{(k)} - (N \inner{u_{\alpha}, q_i} \inner{u_{\alpha}, q_j})_{(k)}} \leq  \max_{(\alpha, i, j) \in J} \left| N \inner{u_\alpha, q_i} \inner{u_\alpha, q_j} - \mathcal{X}_{ij}^{(\alpha)} \right| < C N^{-\epsilon}.
	\label{order_ineq}
\end{equation}
By \eqref{t_star_good_approx_eq} we have uniformly in $l \in [[1, K]]$ that, 
\begin{equation}
	\abs{\mathcal{L}_{\tau}^{l}((\mathcal{X}_{ij}^{\alpha})_{(\alpha, i, j) \in J}) - (\mathcal{X}_{ij}^{(\alpha)})_{(k_l)}} \leq \tau \log(2 (\#_{J})).  
	\label{t_star_good_approx}
\end{equation} 
Combining \eqref{order_ineq} and \eqref{t_star_good_approx} yields that with probability $\geq 1 - N^{-\epsilon}$ and uniformly in $l \in [[1, K]]$, 
\begin{equation*}
	\abs{\mathcal{L}_{\tau}^{l}((\mathcal{X}_{ij}^{\alpha})_{(\alpha, i, j) \in J}) - (N \inner{u_{\alpha}, q_i} \inner{u_{\alpha}, q_j})_{(k_l)}} \leq \tau \log(2 (\#_{J})) + C N^{-\epsilon}.  
	\label{combined_ineq}
\end{equation*} 
Next we use the mean value theorem bound $\abs{F(x) - F(y)} \leq \norm{\nabla F}_{1} \norm{x - y}_{\infty}$, with $x = (\mathcal{L}_{\tau}^{l}((\mathcal{X}_{ij}^{\alpha})_{(\alpha, i, j) \in J}))_{l=1}^K$ and $y = ((N \inner{u_{\alpha}, q_i} \inner{u_{\alpha}, q_j})_{(k_l)})_{l=1}^K$ and apply $\EX$. Using the boundedness of $F$ and the bound on $\norm{\nabla F}_{1}$ gives that for sufficiently large $N$ we have (with a modified $C$),
\begin{equation}
	\abs{\EX[F((\mathcal{L}_{\tau}^{l}((\mathcal{X}_{ij}^{\alpha})_{(\alpha, i, j) \in J}))_{l=1}^K)] - \EX[F(((N \inner{u_{\alpha}, q_i} \inner{u_{\alpha}, q_j})_{(k_l)})_{l=1}^K)]} \leq C N^{\epsilon/4} (\tau \log(2 (\#_{J})) + C N^{-\epsilon}).
	\label{scnd_to_last}
\end{equation} 
Note that \eqref{scnd_to_last} holds also with $\EX^{\GOE}$ instead of $\EX$. Hence, combining \eqref{scnd_to_last} and \eqref{gfc_repeat} yields, 
\begin{equation}
	\begin{multlined}
	\abs{\EX[F(((N \inner{u_{\alpha}, q_i} \inner{u_{\alpha}, q_j})_{(k_l)})_{l=1}^K)] - \EX^{\GOE}[F(((N \inner{u_{\alpha}, q_i} \inner{u_{\alpha}, q_j})_{(k_l)})_{l=1}^K)]}  \\ \leq 2 C N^{\epsilon/4}( \tau \log(2 (\#_{J})) + C N^{-\epsilon}) + N^{-1/9}.   
	\end{multlined}
	\label{last_step}
\end{equation} 
Finally, recall $\tau = N^{-\epsilon}$, and use $\#_J \leq N^{3}$ and $\delta \leq 1/18$ to finish the proof. 
\end{proof}

	\section{Interpolating from generalized Wigner to Gaussian with variance profile} 
	\label{first_comparison_section}

	In this section we prove Lemma \ref{first_comprison_step_lemma}. Recall the interpolation \eqref{interpolation_step_1}. In the remainder of this section we shall denote $H^{(1)}(t)$ as $H(t)$. Recall $\widetilde{F} = F \circ \widetilde{\func}_\tau$, and $\tau = N^{-\epsilon}$. Our goal is to show \eqref{first_comparison_step}, which is equivalent to showing, \begin{equation*}
		\abs{\EX[\widetilde{F}(X(\infty))] - \EX[\widetilde{F}(X(0))]} \leq N^{-1/8}. 
	\end{equation*}
	Our strategy to achieve this is to bound $\dt \EX[\widetilde{F}(X(t))]$ and integrate from 0 to $\infty$. Next we continue the computations started in \eqref{basic_time_derivative_computation} by cumulant expanding using Lemma \ref{cumulant_expansion_lemma}. Because we constructed $H(t)$ to have constant second moment, we obtain cancellation of the variance terms when cumulant expanding $\EX[\dot{h}_{ab}(t) \frac{\partial \widetilde{F}(X(t))}{\partial h_{ab}(t)}]$. Denoting the normalized $p$-th cumulant of $h_{ab}(t)$ as $\kappa_{p}(a,b)$ we get that \begin{equation}
		\sum_{a,b} \frac{1 + \delta_{ab}}{2} \EX\left[\dot{h}_{ab}(t) \frac{\partial \widetilde{F}(X(t))}{\partial h_{ab}(t)}\right] =  \sum_{p+1 = 3}^{P} \sum_{a,b} \frac{1 + \delta_{ab}}{2} \frac{\e^{-t} (1 - \e^{-t})^{\frac{p-1}{2}}}{2} 	\frac{\kappa_{p+1}(a, b)}{N^{\frac{p+1}{2}}} \EX \left[\frac{\partial^{p+1} \widetilde{F}(X(t))}{(\partial h_{ab}(t))^{p+1}}\right] + O(N^{-D}),
		\label{cumulant_expansion1}
	\end{equation} 
	where $P$ is a finite but sufficiently large cut-off to make the bound on the remainder term $R_{l+1}$ in \eqref{cumulant_expansion_error_bound} $\leq N^{-D}$. Showing that a $P$ exists for any $D > 0$ is a routine calculation so we omit the proof. See \cite[Appendix A.2]{bucht_schnelli_xu_2025} for how this can be done. In the remainder of the section we shall without explicitly referring to it, use the fact that we can make the error term in cumulant expansions $O(N^{-D})$ for any $D > 0$ by expanding to a sufficiently large order.   

	Before introducing the general form of the terms in the expansion in \eqref{cumulant_expansion1} we write out what we get from taking the first derivative, \ie 
	\begin{equation}
		\begin{split}
		& \frac{\partial \widetilde{F}(X(t))}{(\partial h_{ab}(t))}  = \frac{\partial F (\widetilde{\func}_\tau (X(t)))}{\partial h_{ab}(t)}  \\
		& = \sum_{\beta \in [[1, K]]} (\partial_{\beta} F) (\widetilde{\func}_\tau (X(t))) \frac{\partial (\func_\tau^{\beta}(X(t)))}{\partial h_{ab}(t)} \\
		&  = \sum_{\beta \in [[1, K]]} (\partial_{\beta} F) (\widetilde{\func}_\tau (X(t))) \sum_{\mathfrak{i} \in J} \partial_{\mathfrak{i}} \func^{\beta}_\tau (X(t)) \frac{\partial \mathcal{X}_{\mathfrak{i}}(t)}{\partial h_{ab}(t)} \\
		& =  \sum_{\beta \in [[1, K]]} (\partial_{\beta} F) (\widetilde{\func}_\tau (X(t))) \sum_{(i, j, \alpha) \in J} \partial_{(i, j, \alpha)} \func^{\beta}_\tau (X(t)) \\
		& \quad \quad \times \left(- \frac{2}{1 + \delta_{ab}} \frac{N}{\pi} \int_I (\im G_{q_i a} G_{b q_j} + G_{q_j a} G_{b q_i}) \varrho_{\alpha}(Y(t)) \di E - \frac{2}{1 + \delta_{ab}} \frac{N}{\pi} \int_I \im G_{ q_{i} q_{j}} \varrho_{\alpha}'(Y(t))  (\Dim G_{ab})\di E\right).
		\end{split}
		\label{one_derivative_example}
	\end{equation}	

	By iteratively applying the derivative rules all higher derivatives in \eqref{cumulant_expansion1} can be computed. The result of computing the higher derivatives in \eqref{cumulant_expansion1} is a finite sum of terms of the form 
	
	\begin{equation} \hspace*{-0cm}
		\begin{split}
			& \frac{g(t)}{N^{\frac{p+1}{2}}} \sum_{a,b} \frac{\kappa_{p+1}(a,b)}{(1 + \delta_{ab})^{p}} \sum_{\beta \in [[1, K]]^{\#_{\beta}}}  \EX \Bigg[ (\partial_{\beta} F)(\widetilde{\func}_\tau (X(t))) \prod_{k=1}^{\#_{\beta}} \Bigg( \sum_{\mathcal{J} \in J^{r_k}} (\partial_{\mathcal{J}} \func_\tau^{\beta_k})(X(t)) \\
			& \times \prod_{\mathfrak{j} \in \mathcal{J}} \Big( \frac{N}{\pi} \int_I \Big(\im \prod_{i=1}^{n_1(\mathfrak{j})} G_{u_i v_i}\Big) \varrho_{\alpha(\mathfrak{j})}^{(n_2(\mathfrak{j}))}(Y(t)) \prod_{l=1}^{n_2(\mathfrak{j})} \Dim \Big(\prod_{j=1}^{\hat{n}_l(\mathfrak{j})}  G_{\hat{u}_j \hat{v}_j} \Big) \di E\Big)\Bigg)\Bigg].
		\end{split}
		\label{general_term_expansion}
	\end{equation}
	The $p$ is the same as in \eqref{cumulant_expansion1}. Further, $\beta$ is a multi-index of length $\#_{\beta}$ with entries $\beta_1, \dots, \beta_{\#_{\beta}} \in [[1, \dots, K]]$, corresponding to the partial derivatives of $F$. The $r_k, 1 \leq k \leq \#_\beta$ are positive integers counting the number of times a derivative has acted on $\func_{\tau}^{\beta_k}$. To get the term \eqref{general_term_expansion}, the derivative $\partial / \partial h_{ab}(t)$ acts exactly $\#_{\beta}$ times on $F$. Furthermore, $\mathcal{J}$ is a multi-index of elements in $J^{r_k}$, representing the partial derivatives of $\func_\tau^{\beta_k}$. The length $\#_{\mathcal{J}}$ of $\mathcal{J}$ increases by 1 if a $\partial/\partial h_{ab}(t)$ acts on the $\func$-factor. 
	
	The $\mathfrak{j} \in \mathcal{J}$ are on the form $\mathfrak{j} = (\alpha, i, j) \in J$, we use $\alpha(\mathfrak{j})$, $i(\mathfrak{j})$, $j(\mathfrak{j})$ to denote the components of $\mathfrak{j}$. The indices $u_i, v_i$ inside the factor indexed by $\mathfrak{j}$ are either the vectors $q_{i(\mathfrak{j})}$ or $q_{j(\mathfrak{j})}$, or the indices $a$ or $b$, with $q_{i(\mathfrak{j})}$ and $q_{j(\mathfrak{j})}$ appearing exactly once and $a$ and $b$ appearing equally many times. The indices $ \hat{u}_j, \hat{v}_j$ are $a$ or $b$, with $a$ and $b$ appearing equally many times. In \eqref{general_term_expansion}, $g(t)$ denotes a time-dependent factor of the form $g(t) = C \e^{-t} (1 - \e^{-t})^{(p-1)/2}$. 

	Next we discuss what happens when the derivative $\partial / \partial h_{ab}(t)$ acts on different factors of \eqref{general_term_expansion} (inside the expectation). Introduce the notation $n(k,\mathfrak{j})$ for the total number of $G$-factors inside the inner $\int_I$. In particular, we have \begin{equation*}
		n(k, \mathfrak{j}) = n_1(\mathfrak{j}) + \sum_{l=1}^{n_2(\mathfrak{j})} \hat{n}_l(\mathfrak{j}). 
	\end{equation*}
	
	We have the following different factors that the derivate can hit.  \begin{itemize}
		\item $\partial/\partial h_{ab}(t)$ acts on $(\partial_\beta \widetilde{F})(X(t))$: We will obtain three new terms corresponding to the three terms on the last row of \eqref{one_derivative_example}. The new terms will have $\#_{\beta} \to \#_{\beta} + 1$. The fresh factor corresponding to $k = \#_{\beta} + 1$ has $r_k = 1$ and $n(k, \mathfrak{j}) = 2$. 

		\item $\partial/\partial h_{ab}(t)$ acts on $\partial_{\mathcal{J}} \func^{\beta_k}_\tau (X(t))$: We get $r_k \to r_k + 1$ and hence a fresh factor in the product $\prod_{\mathfrak{j} \in \mathcal{J}}$. Inside the fresh factor we get exactly the three terms described by the last row in \eqref{one_derivative_example}. Both of these terms have $n(k, \hat{\mathfrak{j}}) = 2$. In total the term we started with gets replaced by $3$ new terms. 
		\item $\partial/\partial h_{ab}(t)$ acts inside an integral $\int_I (\dots) \di E$: It can act on three different factors. Either $$\im \prod_{i=1}^{n_1(\mathfrak{j})} G_{u_i v_i}, \text{ or } \varrho_{\alpha(\mathfrak{j})}^{(n_2(\mathfrak{j}))}(Y(t)),  \text{ or } \prod_{l=1}^{n_2(\mathfrak{j})} \Dim (\prod_{j=1}^{\hat{n}_l(\mathfrak{j})}  G_{\hat{u}_j \hat{v}_j} ).$$  
		In the first case the derivative will act on one $G_{u_i v_i}$ and $n_1(\mathfrak{j}) \to n_1(\mathfrak{j}) + 1$. In the second case $n_2(\mathfrak{j}) \to n_2(\mathfrak{j}) + 1$ and a factor $(\Dim G_{ab})$ is added. Finally, in the third case $\partial/\partial h_{ab}(t)$ acts on one of the $G_{\hat{u}_j \hat{v}_j}$ and the $\hat{n}_l(\mathfrak{j}) \to \hat{n}_l(\mathfrak{j}) + 1$. All in all we obtain at most two fresh terms, and $n(k, \mathfrak{j})$ is increased by exactly 1.  
	\end{itemize}

	Note that in all of the above cases we always get a factor $1/(1 + \delta_{ab})$ for each derivative. Next we want to bound the total amount of terms on the form \eqref{general_term_expansion} we can at most have for a fixed $p$. We can bound this by counting the number of different factors the derivative can hit, and how many new terms get generated in each of the different cases. There is always one $(\partial_\beta F)(\widetilde{\func}_\tau (X(t)))$-factor. If the derivative acts on this we obtain $3$ new terms. Further there are $\#_{\beta}$ factors of $\sum_{\mathcal{J}_k} \partial_{\mathcal{J}_k} \func_\tau^{\beta_k} (X(t))$. Each of these generate $3$ new terms. So in total we have $\#_{\beta} \times 3$ new terms in this case. Finally, there are at most $\sum_{k=1}^{\#_{\beta}} \sum_{\mathfrak{j} \in \mathcal{J}_k} n(k, \mathfrak{j})$ factors the derivative can hit inside the integrals $\int_I (\dots) \di E$. Finally, if the derivative hits one of these factors we get at most $2$ new terms. Hence, the total number of terms of the form \eqref{general_term_expansion} in the $(p+1)$-th term of \eqref{cumulant_expansion1} is bounded by a constant $C_p$.

	Our next goal is to bound the terms in the expansion of \eqref{cumulant_expansion1} of the form \eqref{general_term_expansion} individually. To accomplish this we will use naive local law bounds along with unmatched expansion techniques similar to \cite{Schnelli_Xu_2022,schnelli_xu_generalized_22,bucht_schnelli_xu_2025}. To carry out the expansion we introduce a more general form of \eqref{general_term_expansion}. 
	\begin{definition}[General form of terms in expansion]
		The terms we consider are defined by the data $\#_{\mathcal{I}} \in \mathbb{Z}_{\geq 1}$, a positive integer $\#_{\beta}$, a time factor $g : \mathbb{R}_{\geq 0} \to \mathbb{R}$, a bounded sequence, $\kappa = \kappa_N : [[1, N]]^{\#_{\mathcal{I}}} \to \mathbb{R}$, a constant $d_N \in \mathbb{R}$, and a tuple of positive integers $(r_k)_{k=1}^{\#_{\beta}}$. We define $\mathcal{I}$ as a multi-index of free summation indices $i_1, \dots, i_{\#_{\mathcal{I}}}$ that are iterated from $1$ to $N$. The above data defines a sequence of functions $\R^+ \to \R$ by, 
		\begin{equation}
			\begin{split}
			& t \mapsto g(t) \frac{N^{d_N}}{N^{\#_{\mathcal{I}}}} \sum_{\mathcal{I}} \kappa(\mathcal{I}) \sum_{\beta \in [[1, K]]^{\#_{\beta}}}\EX\Bigg[(\partial_\beta F) (\widetilde{\func}_\tau (X(t))) \prod_{k=1}^{\#_{\beta}} \Bigg(\sum_{\mathcal{J} \in J^{r_k}} \partial_{\mathcal{J}} \func_{\tau}^{\beta_k} (X(t)) \prod_{s=1}^{r_k} \Big( \inFac(k, s, \mathcal{I}, \mathcal{J})\Big)\Bigg)\Bigg],
			\end{split}
			\label{general_term}
		\end{equation} where $\inFac(k, s, \mathcal{I}, \mathcal{J})$ denotes an \textit{inner factor}. Denote $\mathcal{J}(s) = (\alpha(\mathcal{J}(s)), i(\mathcal{J}(s)), j(\mathcal{J}(s)))$, let $q_i$ and $q_j$ be the basis vectors $q_{i(\mathcal{J}(s))}$ and $q_{j(\mathcal{J}(s))}$ respectively and $\alpha = \alpha(\mathcal{J}(s))$. Each inner factor is defined by the following ($k$ and $s$ dependent) data: 
		\begin{itemize}
			\item Integers $n_1 \in \mathbb{Z}_{\geq 1}$, $n_2 \in \mathbb{Z}_{\geq 0}$, and the tuple $(\hat{n}_l)_{l=1}^{n_2} \in \mathbb{Z}_{\geq 1}^{n_2}$. 
			\item $\mathcal{I} \cup \{q_i, q_j\}$-valued indices $u_i, v_i$, $1 \leq i \leq n_1$. We assume that $q_{i}$ and $q_{j}$ occur exactly once in $(u_i, v_i)_{i=1}^{n_1}$. 
			\item For $1 \leq l \leq n_2$: $\mathcal{I}$-valued indices $\hat{u}_{j}, \hat{v}_{j}$, $1 \leq j \leq \hat{n}_l$.
		\end{itemize}
		The above data uniquely defines the inner factor, \begin{equation}
			\inFac(k, s, \mathcal{I}, \mathcal{J}) := \frac{N}{\pi} \int_{I} \im \prod_{i=1}^{n_1} G_{u_i v_i} \varrho^{(n_2)}_{\alpha} \prod_{l=1}^{n_2} \Big(\Dim \Big(\prod_{j=1}^{\hat{n}_l} G_{\hat{u}_j \hat{v}_j}\Big) \Big)\di E.
			\label{inner_fac_def}
		\end{equation} A term of the form \eqref{general_term} shall usually be called $T$, and we shall write \eg $r_k(T)$, $\mathcal{I}(T)$, $n_1(k, s)(T)$, $g(t)(T)$, etc., to denote the defining data of $T$. 
	\end{definition}

 	Next we define the \textit{degree} of a term on the form \eqref{general_term}, which is similar to the notion used in \cite{Schnelli_Xu_2022,schnelli_xu_generalized_22} but slightly adjusted to our current setting. \begin{definition} 
		Consider an inner factor $\inFac(k, s, \mathcal{I}, \mathcal{J})$ as in \eqref{general_term}. Let $d(k, s)$ denote the number of off-diagonal factors in $\im \prod_{i=1}^{n_1(k, s)} G_{u_i v_i}$, not counting the $G$-factors containing a $q$-index. Further, let $\hat{d}(k,s, l)$ be the number of off-diagonal factors in $\prod_{j=1}^{\hat{n}_l(k,s)}  G_{\hat{u}_j \hat{v}_j} $. 
		We define the degrees $d$ and $\hat{d}$ of the whole term \eqref{general_term} as \begin{align*}
			d & := \sum_{k=1}^{\#_{\beta}} \sum_{s=1}^{r_k} d(k, s), \\
			\hat{d} & :=  \sum_{k=1}^{\#_{\beta}} \sum_{s=1}^{r_k} \sum_{l=1}^{n_2(k, s)} \hat{d}(k, s, l). 
		\end{align*}
		Let $d_q(j)$ denote the total number of times one of $u_i, v_i$ is a $q$-index and the other $i_j \in \mathcal{I}$. Further, let $r_{\sum} := \sum_{k=1}^{\#_{\beta}} r_k$, and let $\mathfrak{r}_1$ be the number of inner factors with $d = 0, n_1 = 1$, and let $\mathfrak{r}_2$ be the number of inner factors with $d = 0, n_1 \geq 2$ and let $\mathfrak{r}_3$ be the number of inner factors with $d > 0, n_1 \geq 2$. Finally, introduce $$\mathfrak{c} := \sum_{k=1}^{\#_{\beta}} \sum_{s=1}^{r_k} \sum_{l=1}^{n_2(k, s)} \max\{1, \hat{d}(k, s, l)\},$$ and let $$n_2^{(\mathrm{tot})} := \sum_{k=1}^{\#_{\beta}} \sum_{s=1}^{r_k} n_2(k, s).$$
		\label{degree_def}
	\end{definition}
	
	We introduce more notation to easier reason about terms of the form \eqref{general_term_expansion}. First, abbreviate \begin{equation*}
		\mathfrak{w}_{\mathcal{J}} = \mathfrak{w}_{\mathcal{J}}^{k} := \partial_{\mathcal{J}} \func_\tau^{\beta_k} (X(t)). 
	\end{equation*}

	Next we bound $\abs{\inFac(k, s, \mathcal{I}, \mathcal{J})}$ in high probability by using the $\im$ in front of $\prod_{i=1}^{n_1(k, s)}$. Introduce a new small constant $\xi_6 > 0$. We see that we have \begin{equation}
		\abs{\inFac(k, s, \mathcal{I}, \mathcal{J})} \leq C N \cdot N^{-2/3 + \epsilon} \Psi  \leq \frac{N^{\xi_6 \epsilon}}{N \tilde{\eta}} N^{1/3 + 2 \epsilon} =: \mathcal{M},
		\label{inFac_high_prob_bound}
	\end{equation} where the last inequality holds with probability $\geq 1 - N^{-D}$ for sufficiently large $N$. Note that $\mathcal{M}$ is deterministic, and its size is $N^{C \epsilon}$. 
 
Next we provide a lemma that bounds the inner factors in terms of the random control factor $\Psi$. 
\begin{lemma} 
	Consider an inner factor $\inFac(k, s, \mathcal{I}, \mathcal{J})$ on the form \eqref{inner_fac_def}. Recall the notation $q_i$ and $q_j$ for $q_{i(\mathcal{J}(s))}$ and $q_{j(\mathcal{J}(s))}$. If $n_1(k,s) \geq 2$ we shall wlog assume that $u_1 = q_i$ and $v_2 = q_j$. Assume that all indices in $\mathcal{I}$ take on pairwise distinct values. Then we have for sufficiently small $\epsilon > 0$ and $N \geq N_0(\epsilon)$ that with probability $\geq 1 - N^{-D}$ for any large $D$, 
	\begin{equation}
		\begin{split}
		& \abs{\inFac(k, s, \mathcal{I}, \mathcal{J})} \leq N^{1/3 + 2 \epsilon} B_1 B_2, \\
		& \text{where } B_1 := \begin{cases}
			\Psi & \text{if } d(k, s) = 0, n_1(k, s) = 1, \\
			\Psi(\Psi + \abs{\inner{q_i, v_1}} + \abs{\inner{u_2, q_j}}) & \text{if } d(k, s) = 0, n_1(k, s) \geq 2, \\
			\Psi^{d(k, s)} (\Psi + \abs{\inner{q_i, v_1}}) (\Psi + \abs{\inner{u_2, q_j}}) & \text{if } d(k, s) > 0, n_1(k, s) \geq 2, \\
		\end{cases} \\
		& \text{and } B_2 := \Psi^{\sum_{l=1}^{n_2(k, s)} \max\{1, \hat{d}(k, s, l)\}}.
		\end{split}
		\label{distinct_indices_inner_factor_bound}
	\end{equation}
	 When the indices in $\mathcal{I}$ are allowed to take on the same values, we instead have for any $\epsilon > 0$ and $N \geq N_0 (\epsilon)$ that with probability $\geq 1 - N^{-D}$ for any large $D$, 
	\begin{equation}
		\begin{split}
		& \abs{\inFac(k, s, \mathcal{I}, \mathcal{J})} \leq N^{1/3 + 2 \epsilon} B_1 B_2, \\
		& \text{where } B_1 := \begin{cases}
			\Psi & \text{if } n_1(k, s) = 1, \\
			\Psi(\Psi + \abs{\inner{q_i, v_1}} + \abs{\inner{u_2, q_j}}) & \text{if } n_1(k, s) \geq 2, \\
		\end{cases} \\
		& \text{and } B_2 := \Psi^{n_2(k, s)}.
		\end{split}
		\label{incidental_indices_inner_factor_bound}
	\end{equation}
\end{lemma}
\begin{proof}

The strategy is to bound the absolute value of the integrand uniformly in $E$ and use the mean value bound. The length of the integration interval $I$ is $2 N^{-2/3 + \epsilon}$, so the mean value bound will be $2 N/\pi \cdot N^{-2/3 + \epsilon} = 2 N^{1/3 + \epsilon} / \pi$ times the uniform bound on the integrand. The uniform bound on the integrand will be given by $B_1 B_2$ up to constants. We add a factor $N^{\epsilon}$ to absorb constants. 

We begin by showing \eqref{distinct_indices_inner_factor_bound}. For $n_2 = 1$, note that $\abs{\im G_{q_i q_j}} \leq \Psi$. For the case $n_2 \geq 2$, expand $G_{q_i v_1}$ and $G_{u_2 q_j}$ as $G_{uv} = \inner{u,v} \msc + \mathcal{E}_{uv}$, with the random error defined as $\mathcal{E}_{uv} := G_{uv} - \inner{u, v} \msc$. By \eqref{Psi_def} we have $\abs{\mathcal{E}_{uv}} \leq \Psi$, for $u, v \in \mathbb{B}$. Writing it out we have, \begin{equation*}
	\im \prod_{i=1}^{n_1} G_{u_i v_i} = \im G_{q_i v_1} G_{u_2 q_j} \prod_{l=3 }^{n_1} G_{u_l v_l} = \im (\inner{q_i, v_1} \msc + \mathcal{E}_{q_i v_1}) (\inner{u_2, q_j} \msc + \mathcal{E}_{u_2 q_j}) \prod_{l=3 }^{n_1} G_{u_l v_l}.
\end{equation*} Expanding the first parenthesis we obtain four terms, 
\begin{align}
	\im \prod_{i=1}^{n_1} G_{u_i v_i} &= \im (\inner{q_i, v_1} \msc + \mathcal{E}_{q_i v_1}) (\inner{u_2, q_j} \msc + \mathcal{E}_{u_2 q_j}) \prod_{l=3 }^{n_1} G_{u_l v_l} \nonumber \\ 
	&= \im \inner{q_i, v_1} \inner{u_2, q_j} \msc^2 \prod_{l=3 }^{n_1} G_{u_l v_l} \label{t1} \\ 
	&\quad + \im \inner{q_i, v_1} \mathcal{E}_{u_2 q_j} \msc \prod_{l=3 }^{n_1} \label{t2} G_{u_l v_l} \\ 
	&\quad + \im \mathcal{E}_{q_i v_1} \inner{u_2, q_j} \msc \prod_{l=3 }^{n_1} \label{t3} G_{u_l v_l} \\ 
	&\quad + \im \mathcal{E}_{q_i v_1} \mathcal{E}_{u_2 q_j} \prod_{l=3 }^{n_1} \label{t4} G_{u_l v_l}.
\end{align}

When $d = 0$ we bound with probability $\geq 1 - N^{-D}$, $\abs{\eqref{t1}} \leq C \abs{\inner{q_i, v_1}} \abs{\inner{u_2, q_j}} \Psi$, with the $\Psi$ coming from the imaginary parts of $G_{u_l v_l}$ and $\msc$. For $\abs{\eqref{t2}}$, we just use that $\abs{\mathcal{E}_{u_2 q_j}} \leq \Psi$ and \eqref{abs_bdd_by_one} to see $\abs{\eqref{t2}} \leq C \abs{\inner{q_i, v_1}} \Psi$ with probability $\geq 1 - N^{-D}$. Similarly, we get $\abs{\eqref{t3}} \leq C \abs{\inner{u_2, q_j}} \Psi$ with probability $\geq 1 - N^{-D}$. Finally, we bound $\eqref{t4}$ using $\abs{\mathcal{E}_{u_2 q_j}}, \abs{\mathcal{E}_{q_i v_1}} \leq \Psi$ to see $\abs{\eqref{t4}} \leq \Psi^2$ with probability $\geq 1 - N^{-D}$. Note that since $\abs{\inner{q_i,v_1}}, \abs{\inner{u_2,q_j}} \leq 1$ the bound on $\abs{\eqref{t1}}$ can be absorbed into the bound for $\abs{\eqref{t2}}$. Hence, in the case $n_2 \geq 2, d(k, s) = 0$ we have with probability $\geq 1 - N^{-D}$,
\begin{equation*}
	\abs{\im \prod_{i=1}^{n_1} G_{u_i v_i}} \leq C \Psi (\Psi + \abs{\inner{q_i, v_1}} + \abs{\inner{u_2, q_j}}). 
\end{equation*} 

In the case $d > 0$, we just use $\abs{\msc} \leq 1$ and bound the off-diagonals in $\prod_{l=3}^{n_1} G_{u_l v_l}$ by $\Psi$, and the diagonal factors by $1$ in high probability. This gives $\abs{\im \prod_{i=1}^{n_1} G_{u_i v_i}} \leq C B_1$ for $d > 0, n_1 \geq 2$. So we have shown that with high probability, $\abs{\im \prod_{i=1}^{n_1} G_{u_i v_i}} \leq C B_1$ in all three cases in \eqref{distinct_indices_inner_factor_bound}. To show the bound, 
\begin{equation*}
	\abs{\prod_{l=1}^{n_2(k, s)} \Dim \prod_{j=1}^{\hat{n}_l(k, s)} G_{\hat{u}_j \hat{v}_j}} \leq C B_2, \text{ with probability } \geq 1 - N^{-D},
\end{equation*} note that each factor $\Dim \prod_{j=1}^{\hat{n}_l(k, s)} G_{\hat{u}_j \hat{v}_j}$ can be bounded by $C \Psi$ if $\hat{d}(k, s, l) = 0$ or $C \Psi^{\hat{d}(k, s, l)}$ if $\hat{d}(k, s, l)~>~0$. Hence, the proof of \eqref{distinct_indices_inner_factor_bound} is finished. 

Finally, we note that \eqref{incidental_indices_inner_factor_bound} follows by \eqref{distinct_indices_inner_factor_bound} in the case of $d(k, s) = 0$ and $\hat{d}(k, s, l) = 0$ for $1 \leq l \leq n_2(k, s)$.  \end{proof}

	Next we use the bound on the inner factors to bound entire terms. 

	 \begin{lemma}
		Let $T$ be a term of the form \eqref{general_term} and recall $d_q(j)$ from Definition \ref{degree_def}.  Assume wlog that the $d_q(j)$ are ordered as $d_q(1) \geq d_q(2) \geq \dots \geq d_q(\#_{\mathcal{I}})$, and recall the high probability bound $\mathcal{M}$ in \eqref{inFac_high_prob_bound}. There exist non-negative integers $p_j, j = 1,\ldots,\#_{\mathcal{I}}$, and $\hat{p}_j, j=1, \dots, \#_{\mathcal{I}}-1$ satisfying,
		\begin{equation}
			\begin{split}
			& \sum_{j=1}^{\#_{\mathcal{I}}} p_j \leq \mathfrak{r}_2 + 2 \mathfrak{r}_3, \mathrm{ and } \, p_j \leq d_q(j), \; j = 1,\ldots,\#_{\mathcal{I}}, \\ 
			& \sum_{j=1}^{\#_{\mathcal{I}}} \hat{p}_j \leq \mathfrak{r}_2 + \mathfrak{r}_3, \mathrm{ and } \, \hat{p}_1 \leq d_q(1) + d_q(2), \, \hat{p}_j \leq d_q(j+1), \; j=2,\ldots,\#_{\mathcal{I}} - 1, 
			\end{split}
			\label{p_conditions}
		\end{equation}
		such that for any (small) $\xi_7 > 0$ and sufficiently large $N$ we have 
		\begin{equation}
			\begin{split}
			& \abs{g(t) \frac{N^{d_N}}{N^{\#_{\mathcal{I}}}} \sum_{\mathcal{I}} \kappa(\mathcal{I}) \sum_{\beta \in [[1, K]]^{\#_{\beta}}} \EX\Bigg[(\partial_\beta F) (\widetilde{\func}_{\tau}(X(t))) \prod_{k=1}^{\#_{\beta}} \Bigg(\sum_{\mathcal{J} \in J^{r_k}} \mathfrak{w}_{\mathcal{J}}^{k} \prod_{s=1}^{r_k} \inFac(k, s, \mathcal{I}, \mathcal{J}) \Bigg)\Bigg]} \\
			& \leq g(t) N^{\xi_7} N^{d_N} \mathcal{M}^{r_{\sum}} \Bigg( \prod_{k=1}^{\#_{\beta}} \frac{N^{\Fasump \epsilon}}{\tau^{r_k - 1}} \Bigg) \Bigg(\Big(\frac{1}{N \tilde{\eta}}\Big)^{d + \mathfrak{r}_2 + \mathfrak{r}_3 + \mathfrak{c} - \sum_{j=1}^{\#_{\mathcal{I}}} \max\{p_j-3, 0\}}  + \frac{1}{N} \Big(\frac{1}{N \tilde{\eta}}\Big)^{\mathfrak{r}_2 + \mathfrak{r}_3 + n_2^{(\mathrm{tot})} - \sum_{j=1}^{\#_{\mathcal{I}}-1} \max\{\hat{p}_j - 3, 0\}} \Bigg).
			\end{split}
			\label{local_law_bound}
		\end{equation}
		\label{local_law_bound_lemma}
	\end{lemma}

	\begin{proof}
	
		Denote the distinct summation indices part of $T$ by $\mathring{T}$, \ie \eqref{general_term} but $\sum_{\mathcal{I}}$ replaced by $\disSum_{\mathcal{I}}$. We will first show the bound for $\mathring{T}$. We return in the end to bounding $T - \mathring{T}$. 
		
		Now, use that $\kappa_{\mathcal{I}}$ is bounded, and that $(\partial_{\beta} F)(\widetilde{\func}_{\tau})$ does not depend on $\mathcal{I}$, so we can swap the summations over $\beta$ and $\mathcal{I}$, also move in the factor $1/N^{\#_{\mathcal{I}}}$. Finally, expand the product $\prod_{k=1}^{\#_{\beta}}$. We then get,
		\begin{equation*}
		|\mathring{T}| \leq C g(t) N^{d_N}  \EX \Bigg[\sum_{\beta \in [[1, K]]^{\#_{\beta}}} \abs{(\partial_{\beta} F)(\widetilde{\func}_{\tau})} \frac{1}{N^{\#_{\mathcal{I}}}} \disSum_{\mathcal{I}}  \sum_{\mathcal{J}_1, \dots, \mathcal{J}_{\#_{\beta}}} \abs{ \mathfrak{w}_{\mathcal{J}_1} \cdots \mathfrak{w}_{\mathcal{J}_{\#_{\beta}}}} \prod_{k=1}^{\#_{\beta}} \prod_{s=1}^{r_k} \abs{\inFac(k, s, \mathcal{I}, \mathcal{J}_k) }\Bigg] =: (\star).
		\end{equation*}	
		Using the assumption on the derivatives of $F$, \eqref{deriv_F_bound_assump}, and that the $\mathfrak{w}$ do not depend on $\mathcal{I}$, we obtain, \begin{equation}
			(\star) \leq C g(t) N^{d_N} N^{(\#_{\beta}) (\Fasump \epsilon)} \EX \Bigg[\max_{\beta \in [[1, K]]^{\#_{\beta}}}   \sum_{\mathcal{J}_1, \dots, \mathcal{J}_{\#_{\beta}}} \abs{ \mathfrak{w}_{\mathcal{J}_1} \cdots \mathfrak{w}_{\mathcal{J}_{\#_{\beta}}}} \frac{1}{N^{\#_{\mathcal{I}}}} \disSum_{\mathcal{I}} \prod_{k=1}^{\#_{\beta}} \prod_{s=1}^{r_k} \abs{\inFac(k, s, \mathcal{I}, \mathcal{J}_k) }\Bigg].
			\label{second_to_final_step}
	\end{equation}
	Next we focus on bounding the averages of product of inner factors, 
	\begin{equation*}
	\frac{1}{N^{\#_{\mathcal{I}}}} \disSum_{\mathcal{I}} \prod_{k=1}^{\#_{\beta}} \prod_{s=1}^{r_k} \abs{\inFac(k, s, \mathcal{I}, \mathcal{J}_k) }.	
	\end{equation*}
	Insert the bounds from \eqref{distinct_indices_inner_factor_bound}, and expand all products. We obtain at most $3^{r_{\sum}}$ terms when expanding the products. Each of the terms can be written as \begin{equation}
		(N^{1/3 + 2 \epsilon} \Psi)^{\mathfrak{r}_1 + \mathfrak{r}_2 + \mathfrak{r}_3} \Psi^{d - \mathfrak{r}_3} \Psi^{\mathfrak{c}} \Psi^{p_{\Psi}} \Big(\prod_{l}^{p_1} \abs{\inner{q_{1l}, i_1}} \Big) \cdots \Big(\prod_{l}^{p_{\#_{\mathcal{I}}}} \abs{\inner{q_{\#_{\mathcal{I}} l}, i_{\#_{\mathcal{I}}}}} \Big),
		\label{inner_factor_expansion_terms}
	\end{equation} where the $q_{1l}, q_{2l}, \ldots q_{\#_{\mathcal{I}} l}$ are the vectors that appear together with the indices $i_1, \ldots, i_{\#_{\mathcal{I}}}$ in the Green function factors, and $p_j$ is the number of times $i_j$ occurs in a scalar product, and $p_{\Psi}$ is the number of $\Psi$-factors we pick up from the factors in $B_1$ in \eqref{distinct_indices_inner_factor_bound} that contain a $\Psi$ plus one or two scalar products. Note that $p_j \leq d_{q}(j)$. Since in each factor from the second and third case of $B_1$ in \eqref{distinct_indices_inner_factor_bound} we must pick up either a $\Psi$ or a scalar product, we have \begin{equation}
		p_{\Psi} + p_1 + \dots + p_{\#_{\mathcal{I}}} = \mathfrak{r}_2 + 2 \mathfrak{r}_3. 	
		\label{p_distribution}
	\end{equation} The expression in \eqref{inner_factor_expansion_terms} is factored in the $i_j$ so for each term as in \eqref{inner_factor_expansion_terms} we have \begin{equation*}
		\begin{split}
		& \frac{1}{N^{\#_{\mathcal{I}}}} \disSum_{\mathcal{I}}  (N^{1/3 + 2 \epsilon} \Psi)^{\mathfrak{r}_1 + \mathfrak{r}_2 + \mathfrak{r}_3} \Psi^{d - \mathfrak{r}_3} \Psi^{\mathfrak{c}} \Psi^{p_{\Psi}} \Big(\prod_{l}^{p_1} \abs{\inner{q_{1l}, i_1}} \Big) \cdots \Big(\prod_{l}^{p_{\#_{\mathcal{I}}}} \abs{\inner{q_{\#_{\mathcal{I}} l}, i_{\#_{\mathcal{I}}}}} \Big)	\\
		& \quad \leq (N^{1/3 + 2 \epsilon} \Psi)^{\mathfrak{r}_1 + \mathfrak{r}_2 + \mathfrak{r}_3} \Psi^{d - \mathfrak{r}_3} \Psi^{\mathfrak{c}} \Psi^{p_{\Psi}} \Big(\frac{1}{N} \sum_{i_1} \prod_{l}^{p_1} \abs{\inner{q_{1l}, i_1}} \Big) \cdots \Big(\frac{1}{N} \sum_{i_{\#_{\mathcal{I}}}} \prod_{l}^{p_{\#_{\mathcal{I}}}} \abs{\inner{q_{\#_{\mathcal{I}}l}, i_{\#_{\mathcal{I}}}}} \Big).
		\end{split}
	\end{equation*} Note that to obtain the above factorization we bound $\sum_{\mathcal{I}}$ by $\disSum_{\mathcal{I}}$, which we can do since all the summands are non-negative. 

	Using Cauchy-Schwarz, $\abs{\inner{q_{\cdot \cdot}, i_j}} \leq 1$, and $\sum_{i_j} \abs{\inner{q_{\cdot \cdot}, i_j}}^2 = 1$ we have for $j = 1, \ldots, \#_{\mathcal{I}}$ that, \begin{equation}
		\frac{1}{N} \sum_{i_j} \prod_{l}^{p_j} \abs{\inner{q_{jl}, i_j}} \leq \begin{cases}
			1 & p_j = 0, \\
			\frac{1}{\sqrt{N}} & p_j = 1, \\
			\frac{1}{N} & p_j \geq 2.
		\end{cases}
		\label{inner_power_bounds}
	\end{equation}
	Using \eqref{inner_power_bounds} on each factor, we obtain the bound, 
	\begin{equation}
		\begin{split}
		& \Psi^{p_{\Psi}} \prod_{j=1}^{\mathcal{I}}\Big(\frac{1}{N} \sum_{i_j} \prod_{l}^{p_j} \abs{\inner{q_{jl}, i_j}} \Big) \leq \Psi^{p_{\Psi}} \prod_{j = 1}^{\#_{\mathcal{I}}} \Big(\frac{1}{\sqrt{N}}\Big)^{\max\{p_j, 2\}}. 
		\end{split}
		\label{worst_p_dist_analysis}
	\end{equation} 

	Assume wlog that $p_{\Psi}$, $p_1, \ldots, p_{\#_{\mathcal{I}}}$ is the configuration giving the worst bound on the right of \eqref{worst_p_dist_analysis}. We will show that they satisfy \eqref{local_law_bound}. Rewrite the right side of \eqref{worst_p_dist_analysis} as follows, 
	\begin{equation}
		\begin{split}
			\Psi^{p_{\Psi}} \prod_{j = 1}^{\#_{\mathcal{I}}} \Big(\frac{1}{\sqrt{N}}\Big)^{\max\{p_j, 2\}} & = \Psi^{\mathfrak{r}_2 + 2 \mathfrak{r}_3} \Psi^{p_{\Psi} - (\mathfrak{r}_2 + 2 \mathfrak{r}_3)} \prod_{j=1}^{\#_{\mathcal{I}}} \Big(\frac{1}{\sqrt{N}}\Big)^{\max\{p_j, 2\}} \\
			& = \Psi^{\mathfrak{r}_2 + 2 \mathfrak{r}_3} \Psi^{- \sum_{j=1}^{\#_{\mathcal{I}}} p_j} \prod_{j=1}^{\#_{\mathcal{I}}} \Big(\frac{1}{\sqrt{N}}\Big)^{\max\{p_j, 2\}} \\ 
			& = \Psi^{\mathfrak{r}_2 + 2 \mathfrak{r}_3} \prod_{j=1}^{\#_{\mathcal{I}}} \Big(\frac{1}{\sqrt{N}}\Big)^{\max\{p_j, 2\}} \Psi^{-p_j}. 
		\end{split}
		\label{bounding_scalar_terms_step1}
	\end{equation} 
	Since $\Psi \geq N^{-1/3}$ we see that for $0 \leq p_j \leq 3$ we have 
	\begin{equation*}
		\Big(\frac{1}{\sqrt{N}}\Big)^{p_j} \Psi^{-p_j} \leq 1. 	
	\end{equation*} 
	Consequently, we can continue from \eqref{bounding_scalar_terms_step1} and obtain,
	\begin{equation}
		\begin{split}
		\Psi^{p_{\Psi}} \prod_{j = 1}^{\#_{\mathcal{I}}} \Big(\frac{1}{\sqrt{N}}\Big)^{\max\{p_j, 2\}} & \leq \Psi^{\mathfrak{r}_2 + 2 \mathfrak{r}_3} \prod_{\substack{1 \leq j \leq \#_{\mathcal{I}} \\ p_j \geq 4}} \frac{1}{N} \Psi^{-p_j} \\ 
		& \leq \Psi^{\mathfrak{r}_2 + 2 \mathfrak{r}_3} \prod_{\substack{1 \leq j \leq \#_{\mathcal{I}} \\ p_j \geq 4}} \Psi^{-(p_j - 3)} \\ 
		& = \Psi^{\mathfrak{r}_2 + 2 \mathfrak{r}_3} \Psi^{-\sum_{j=1}^{\#_{\mathcal{I}}} \max\{p_j - 3, 0\}}. \\ 
		\end{split}
		\label{bounding_scalar_terms_step2}
	\end{equation}
	For the second inequality in \eqref{bounding_scalar_terms_step2} we used that $1/N \leq \Psi^{3}$. 	
	
	Plugging \eqref{worst_p_dist_analysis} and \eqref{bounding_scalar_terms_step2} into \eqref{inner_factor_expansion_terms} we have that,
	\begin{equation}
		\begin{split}
			& (N^{1/3 + 2 \epsilon} \Psi)^{\mathfrak{r}_1 + \mathfrak{r}_2 + \mathfrak{r}_3} \Psi^{d - \mathfrak{r}_3} \Psi^{\mathfrak{c}} \Psi^{p_{\Psi}} \Big(\frac{1}{N} \sum_{i_1} \prod_{l}^{p_1} \abs{\inner{q_{1l}, i_1}} \Big) \cdots \Big(\frac{1}{N} \sum_{i_{\#_{\mathcal{I}}}} \prod_{l}^{p_{\#_{\mathcal{I}}}} \abs{\inner{q_{\#_{\mathcal{I}}l}, i_{\#_{\mathcal{I}}}}} \Big) \\
			& \quad \leq (N^{1/3 + 2 \epsilon} \Psi)^{\mathfrak{r}_1 + \mathfrak{r}_2 + \mathfrak{r}_3} \Psi^{d - \mathfrak{r}_3} \Psi^{\mathfrak{c}} \Psi^{\mathfrak{r}_2 + 2 \mathfrak{r}_3} \Psi^{-\sum_{j=1}^{\#_{\mathcal{I}}} \max\{p_j - 3, 0\}} \\ 
			& \quad = (N^{1/3 + 2 \epsilon} \Psi)^{r_{\sum}} \Psi^{d + \mathfrak{r}_2 + \mathfrak{r}_3 + \mathfrak{c}} \Psi^{-\sum_{j=1}^{\#_{\mathcal{I}}} \max\{p_j - 3, 0\}}.
		\end{split}
		\label{product_in_fac_bound}
	\end{equation} 

	Recall the assumption in \eqref{func_assumptions}, which gives that, \begin{equation*}
		\sum_{\mathcal{J}_1, \dots, \mathcal{J}_{\#_{\beta}}} \abs{ \mathfrak{w}_{\mathcal{J}_1} \cdots \mathfrak{w}_{\mathcal{J}_{\#_{\beta}}}} \leq C \prod_{k=1}^{\#_{\beta}} \frac{1}{\tau^{r_k - 1}}.
	\end{equation*}
	By using that $N^{1/3 + 2 \epsilon} \Psi \leq \mathcal{M}$ with probability $\geq 1 - N^{-D}$, and \eqref{Psi_bounds}, we insert the bound from \eqref{product_in_fac_bound} into \eqref{second_to_final_step} and obtain after absorbing the constants into $N^{\xi_7}$, 
	\begin{equation*}
		|\mathring{T}| \leq g(t) N^{\xi_7} N^{d_N} \mathcal{M}^{r_{\sum}} \Bigg( \prod_{k=1}^{\#_{\beta}} \frac{N^{\Fasump \epsilon}}{\tau^{r_k - 1}} \Bigg) \Bigg(\Big(\frac{1}{N \tilde{\eta}}\Big)^{d + \mathfrak{r}_2 + \mathfrak{r}_3 + \mathfrak{c} -\sum_{j=1}^{\#_{\mathcal{I}}} \max\{p_j - 3, 0\}} \Bigg). 
	\end{equation*} This is the first term of \eqref{local_law_bound}. 

	Finally, we bound $|T - \mathring{T}|$. The strategy is similar to the above. Bound the sum over $\mathcal{I}$ by $\binom{\#_{\mathcal{I}}}{2}$ sums where in each of the sums two indices in $\mathcal{I}(T)$ are grouped together to form terms $\hat{T}_k$. Hence, we want to bound $\binom{\#_{\mathcal{I}(T)}}{2}$ terms $\hat{T}_k$, $1 \leq k \leq \binom{\#_{\mathcal{I}(T)}}{2}$ of the form \eqref{general_term} with $d_N(\hat{T}_k) = d_N(T) - 1$, using \eqref{incidental_indices_inner_factor_bound}. We use the same strategy as in \eqref{second_to_final_step} and bound the averages of the products of inner factors of the $\hat{T}_k$, but instead of using \eqref{distinct_indices_inner_factor_bound} we use \eqref{incidental_indices_inner_factor_bound}. Note that the number of inner factors that fall into the second case of $B_1$ in \eqref{incidental_indices_inner_factor_bound} is exactly $\mathfrak{r}_2 + \mathfrak{r}_3$. The bounds of the average inner factors are \begin{equation}
		(N^{1/3 + 2 \epsilon} \Psi)^{r_{\sum}} \Psi^{n_2^{(\mathrm{tot})}} \Psi^{\hat{p}_{\Psi}} \Big(\prod_{l=1}^{\hat{p}_1} \abs{\inner{q_{1l}, i_1}} \Big) \cdots \Big(\prod_{l=1}^{\hat{p}_{\#_{\mathcal{I}(\hat{T}_k)}}} \abs{\inner{q_{\#_{\mathcal{I}(\hat{T}_k)} l}, i_{\#_{\mathcal{I}(\hat{T}_k)}}}} \Big),
		\label{incidental_inner_prod_bound}
	\end{equation} where the $\hat{p}_j$ sum to $\mathfrak{r}_2 + \mathfrak{r}_3 - \hat{p}_{\Psi}$. By repeating the argument using \eqref{inner_power_bounds} and choosing $\hat{p}_j$ as the configuration maximizing \eqref{incidental_inner_prod_bound} and wlog assuming they are ordered as $\hat{p}_1 \geq \dots \geq \hat{p}_{\#_{\mathcal{I}(\hat{T}_k)}}$, we obtain, \begin{equation}
		\begin{split}
		& \sum_{\mathcal{I}(\hat{T}_k)} (N^{1/3 + 2 \epsilon} \Psi)^{r_{\sum}} \Psi^{n_2^{(\mathrm{tot})}} \Psi^{\hat{p}_{\Psi}} \Big(\prod_{l=1}^{\hat{p}_1} \abs{\inner{q_{1l}, i_1}} \Big) \cdots \Big(\prod_{l=1}^{\hat{p}_{\#_{\mathcal{I}(\hat{T}_k)}}} \abs{\inner{q_{\#_{\mathcal{I}(\hat{T}_k)} l}, i_{\#_{\mathcal{I}(\hat{T}_k)}}}} \Big) \\ 
		& \quad \leq (N^{1/3 + 2 \epsilon} \Psi)^{r_{\sum}} \Psi^{n_2^{(\mathrm{tot})}} \Psi^{\mathfrak{r}_2 + \mathfrak{r}_3 - \sum_{j=1}^{\#_{\mathcal{I}(\hat{T}_k)}} \max\{\hat{p}_j - 3, 0\}}.
		\end{split}
		\label{second_to_last_step}
	\end{equation} Note that by the ordering assumption on $\hat{p}_j$ and since the indices in $\mathcal{I}(T)$ are grouped to create $\hat{T}_k$, they satisfy \eqref{p_conditions}. By inserting \eqref{second_to_last_step} into the analogue of \eqref{second_to_final_step}, we obtain, 
	\begin{equation*}
		\abs{T - \mathring{T}} \leq g(t) N^{\xi_7} N^{d_N} \mathcal{M}^{r_{\sum}} \Bigg(\prod_{k=1}^{\#_{\beta}} \frac{N^{\Fasump \epsilon}}{\tau^{r_k - 1}} \Bigg) \Bigg(\frac{1}{N} \Big(\frac{1}{N \tilde{\eta}}\Big)^{\mathfrak{r}_2 + \mathfrak{r}_3 + n_2^{(\mathrm{tot})} - \sum_{j=1}^{\#_{\mathcal{I}(T)}-1} \max\{\hat{p}_j - 3, 0\}} \Bigg).
	\end{equation*} 	
	With this we have finished the proof of Lemma \ref{local_law_bound_lemma}. 
	\end{proof}

	Equipped with Lemma \ref{local_law_bound_lemma}, we are ready to show Lemma \ref{first_comprison_step_lemma}. 	
	\begin{proof}[Proof of Lemma \ref{first_comprison_step_lemma}]
	We recall that the terms in the expansion of $\dt \EX[\widetilde{F}(X(t))]$, \eqref{cumulant_expansion1} have the form \eqref{general_term} and \eqref{general_term_expansion} in particular, with $g(t) = O(\e^{-t})$. So if we can show that all terms $T$ in the expansion of $\dt \EX[\widetilde{F}(X(t))]$ satisfy \eg $\abs{T} \leq O(\e^{-t}N^{-7})$, then \eqref{first_comparison_step} follows by integrating $\dt \EX[\widetilde{F}(X(t))]$ from $0$ to $\infty$. For the rest of the proof, assume $T$ is any term of the form \eqref{general_term} from the expansion of $\dt \EX[\widetilde{F}(X(t))]$. We shall show that for sufficiently small $\epsilon > 0$ and sufficiently large $N$, we have for any small constant $c > 0$,
	\begin{equation}
		\abs{T} \leq \e^{-t}N^{-1/6 + c}.
		\label{goal_bound}
	\end{equation} 
		
	Recall the conditions on the $p_j$ and $\hat{p}_j$ from \eqref{p_conditions}. Since $\sum_{j=1}^{\#_{\mathcal{I}}} p_j \leq \mathfrak{r}_2 + 2\mathfrak{r}_3$, and $d \geq \mathfrak{r}_3$, implying that
	\begin{equation}
		d + \mathfrak{r}_2 + \mathfrak{r}_3 + \mathfrak{c} - \sum_{j=1}^{\#_{\mathcal{I}}} \max\{p_j - 3, 0\}	\geq \min\{d + \mathfrak{r}_2 + \mathfrak{r}_3 + \mathfrak{c}, 3\}. 
		\label{deg_lower_bound}
	\end{equation} 
	Similarly, we also have,
	\begin{equation}
		\mathfrak{r}_2 + \mathfrak{r}_3 + n_2^{(\mathrm{tot})} - \sum_{j=1}^{\#_{\mathcal{I}}-1} \max\{\hat{p}_j - 3, 0\}	\geq \min\{\mathfrak{r}_2 + \mathfrak{r}_3 + n_2^{(\mathrm{tot})}, 3\}. 
		\label{hat_deg_lower_bound}
	\end{equation} 
	In particular, \eqref{deg_lower_bound} and \eqref{hat_deg_lower_bound} imply that, for any $\xi_7 > 0$ and sufficiently large $N$ we have,
	\begin{equation}
		\abs{T} \leq g(t) N^{\xi_7} N^{d_N} \mathcal{M}^{r_{\sum}} \prod_{k=1}^{\#_{\beta}} \frac{N^{\Fasump \epsilon}}{\tau^{r_k - 1}} \Bigg(N^{(d + \mathfrak{r}_2 + \mathfrak{r}_3 + \mathfrak{c})(-1/3 + 14 \epsilon)} + N^{-1 + 42 \epsilon}\Bigg).
		\label{simplified_bound}
	\end{equation}
	
	Since $T$ must have inner factors, meaning $d + \mathfrak{r}_2 + \mathfrak{r}_3 + \mathfrak{c} \geq 1$, we see that if $T$ comes from \eqref{cumulant_expansion1} with $p + 1 \geq 4$, we have $d_N(T) \leq 0$. So then we get immediately from \eqref{simplified_bound} that \eqref{goal_bound} is satisfied. Hence, the only remaining case is when $d_N(T) > 0$. This can only happen if $T$ comes from $p+1 = 3$ and has a $\kappa_3$-factor. The rest of the proof focuses on this case. 

	If we have $d + \mathfrak{r}_2 + \mathfrak{r}_3 + \mathfrak{c} \geq 2$, \eqref{simplified_bound} is directly gives \eqref{goal_bound}, so the only remaining case to consider is when $d + \mathfrak{r}_2 + \mathfrak{r}_3 + \mathfrak{c} = 1$. This condition means that we can only have one inner factor, and only one $\Dim$-factor, and $d=0$. Further, since it comes from \eqref{cumulant_expansion1} with $p+1=3$, $a$ and $b$ both occur exactly three times. This leaves us with exactly two possible terms, 
	\begin{align}
		& g(t) \frac{N^{1/2}}{N^{2}} \sum_{a, b} \kappa_3(a, b) \sum_{\beta \in [[1, K]]} \EX\Bigg[(\partial_{\beta} F)(\widetilde{\func}_{\tau}(X(t))) \sum_{\mathcal{J} \in J} \mathfrak{w}_{\mathcal{J}} \frac{N}{\pi} \int_I \im (G_{qa} G_{bq} G_{aa} G_{bb}) \varrho_{\alpha} \di E\Bigg], \label{first_remaining_case} \\
		& g(t) \frac{N^{1/2}}{N^{2}} \sum_{a, b} \kappa_3(a, b) \sum_{\beta \in [[1, K]]} \EX\Bigg[(\partial_{\beta} F)(\widetilde{\func}_{\tau}(X(t))) \sum_{\mathcal{J} \in J} \mathfrak{w}_{\mathcal{J}} \frac{N}{\pi} \int_I \im (G_{qq}) \varrho_{\alpha}' \Dim (G_{aa} G_{ab} G_{bb})\di E\Bigg]. \label{second_remaining_case}
	\end{align}
	These can be expanded into terms with $d + \mathfrak{r}_2 + \mathfrak{r}_3 + \mathfrak{c} \geq 2$. The expansion technique we use is very similar to the one used for unmatched terms in \cite{schnelli_xu_generalized_22,Schnelli_Xu_2022,bucht_schnelli_xu_2025,green_function_comparison_reference}.

	We shall demonstrate how to expand \eqref{first_remaining_case}, expanding \eqref{second_remaining_case} is done in an analogous manner. For the expansion we use the identity,
	\begin{equation}
		G_{au} = \inner{e_a, u} \underline{G} + G_{au} \underline{HG} - \underline{G} \sum_{c} h_{ac} G_{cu},
		\label{green_function_identity1}
	\end{equation} where $u$ can be either an index or a unit vector. We expand $G_{qa}$ in \eqref{first_remaining_case} and obtain, 
	\begin{equation}
		\begin{split}
			& g(t) \frac{N^{1/2}}{N^{2}} \sum_{a, b} \kappa_3(a, b) \sum_{\beta \in [[1, K]]} \EX\Bigg[(\partial_{\beta} F)(\widetilde{\func}_{\tau}(X(t))) \sum_{\mathcal{J} \in J} \mathfrak{w}_{\mathcal{J}} \frac{N}{\pi} \int_I \im (G_{qa} G_{bq} G_{aa} G_{bb}) \varrho_{\alpha} \di E\Bigg] \\
			& = g(t) \frac{N^{1/2}}{N^{2}} \sum_{a, b} \kappa_3(a, b) \sum_{\beta \in [[1, K]]} \EX\Bigg[(\partial_{\beta} F)(\widetilde{\func}_{\tau}(X(t))) \sum_{\mathcal{J} \in J} \mathfrak{w}_{\mathcal{J}} \frac{N}{\pi} \int_I \im (\inner{e_a, q}\underline{G} G_{bq} G_{aa} G_{bb}) \varrho_{\alpha} \di E\Bigg] \\
			& \quad + g(t) \frac{N^{1/2}}{N^{3}} \sum_{a, b, c, d} \kappa_3(a, b) \sum_{\beta \in [[1, K]]} \EX\left[(\partial_{\beta} F)(\widetilde{\func}_{\tau}(X(t))) \sum_{\mathcal{J} \in J} \mathfrak{w}_{\mathcal{J}} \frac{N}{\pi} \int_I \im (h_{cd}(t) G_{cd} G_{qa} G_{bq} G_{aa} G_{bb}) \varrho_{\alpha} \di E\right] \\
			& \quad - g(t) \frac{N^{1/2}}{N^{3}} \sum_{a, b, c, d} \kappa_3(a, b) \sum_{\beta \in [[1, K]]}  \EX\left[(\partial_{\beta} F)(\widetilde{\func}_{\tau}(X(t))) \sum_{\mathcal{J} \in J} \mathfrak{w}_{\mathcal{J}} \frac{N}{\pi} \int_I \im (h_{ac}(t) G_{cq} G_{dd} G_{bq} G_{aa} G_{bb}) \varrho_{\alpha} \di E\right].
		\end{split}
		\label{unmatched_expansion_example}
	\end{equation} 
	Then, we cumulant expand the two last terms above. From these cumulant expansions, all terms but two satisfy $d + \mathfrak{r}_2 + \mathfrak{r}_3 + \mathfrak{c}_2 + \hat{d} \geq 2$, where the two exceptional terms arise when $\partial/ \partial h_{cd}(t)$ acts on $G_{cd}$ generating $G_{cc} G_{dd}$ and when $\partial/ \partial h_{ac}(t)$ acts on $G_{cq}$ generating $G_{cc} G_{aq}$. Thus, the leading terms from the cumulant expansions are
	\begin{equation*}
		\begin{split}
			& g(t) \frac{N^{1/2}}{N^{4}} \sum_{a, b, c, d} \frac{\kappa_3(a, b)}{(1 + \delta_{ab})^2}  \frac{\kappa_2(c, d)}{1 + \delta_{cd}} \sum_{\beta \in [[1, K]]}  \EX\left[(\partial_{\beta} F)(\widetilde{\func}_{\tau}(X(t))) \sum_{\mathcal{J} \in J} \mathfrak{w}_{\mathcal{J}} \frac{N}{\pi} \int_I \im (G_{cc} G_{dd} G_{qa} G_{bq} G_{aa} G_{bb}) \varrho_{\alpha} \di E\right] \\
			& - g(t) \frac{N^{1/2}}{N^{4}} \sum_{a, b, c, d}\frac{\kappa_3(a, b)}{(1 + \delta_{ab})^2}  \frac{\kappa_2(a, c)}{1 + \delta_{ac}} \sum_{\beta \in [[1, K]]}  \EX\left[(\partial_{\beta} F)(\widetilde{\func}_{\tau}(X(t))) \sum_{\mathcal{J} \in J} \mathfrak{w}_{\mathcal{J}} \frac{N}{\pi} \int_I \im (G_{cc}  G_{qa} G_{dd} G_{bq} G_{aa} G_{bb}) \varrho_{\alpha} \di E\right].
		\end{split}
	\end{equation*} 
	These terms have exactly the same $u_i, v_i$, but do not cancel exactly since we have different $\kappa$-factors in front. But we can replace $G_{dd} = m_{sc} + (G_{dd} - m_{sc})$, in both terms. For the terms with $G_{dd} - m_{sc}$, we obtain the same bound that we would have with $d = 1$, so they satisfy \eqref{goal_bound}, and the terms with $G_{dd}$ replaced by $m_{sc}$ cancel (almost) precisely since by the variance assumption \eqref{sum_condition} we have, \begin{equation*}
		\frac{1}{N} \sum_{d} \frac{\kappa_2(c, d)}{1 + \delta_{cd}} - \frac{1}{N} \sum_{d} 1 = O\Big(\frac{1}{N}\Big).
	\end{equation*} 
	Now we are almost done. The only remaining term is the first one of the right side of \eqref{unmatched_expansion_example}. When we bound it using the same techniques as in \eqref{local_law_bound}, we gain a factor $1/\sqrt{N}$ from the $\inner{e_a, q}$, yielding the bound,
	\begin{equation*}
		g(t) N^{\xi_7} N^{1/2} \mathcal{M} N^{\Fasump \epsilon} \Bigg(\frac{1}{\sqrt{N}}N^{-1/3 + 14 \epsilon} + N^{-1 + 42 \epsilon}\Bigg),
	\end{equation*}
	that directly implies \eqref{goal_bound}.

	As mentioned above, the argument for \eqref{second_remaining_case} works in the same manner so we omit it. Finally, note that we truncate the expansion \eqref{cumulant_expansion1} at a finite value, hence we have a constant finite number of terms. Thus, \eqref{goal_bound} holding for all terms in the expansion \eqref{general_term_expansion} concludes the proof of \eqref{first_comparison_step}. \end{proof}

	\section{Interpolating between Gaussian with variance profile and $\GOE$}
	\label{second_comparison_section}
	In this section we prove Lemma \ref{gaussian_to_gaussian_lemma} by using the flow $H^{(2)}(t)$ given in \eqref{gaussian_to_gaussian_interpolating_flow}. We will write $H(t) = H^{(2)}(t)$, and $G = G(t, z) = (H(t) - zI)^{-1}$. With this $H(t)$ we have \begin{equation}
		\dot{h}_{ab}(t) = -\frac{1}{2} \e^{-\frac{t}{2}} w_{ab} + \frac{\e^{-t}}{2 \sqrt{1 - \e^{-t}}} w_{ab}^{(\widetilde{v})}. 
		\label{gaussian_entries_time_derivative}
	\end{equation} Further, the cumulants of $w_{ab}$ and $w_{ab}^{(\widetilde{v})}$ are given by \begin{equation}
		c^{(2)}(w_{ab}) = \frac{1+\delta_{ab}}{N}, \quad c^{(2)} (w_{ab}^{(\widetilde{v})}) = (1 + \delta_{ab}) V_{ab}, \quad c^{(k)}(w_{ab}) = c^{(k)}(w_{ab}^{(\widetilde{v})}) = 0, \quad k = 1, 3, 4, \ldots.
	\end{equation} By using \eqref{gaussian_entries_time_derivative} we can cumulant expand \eqref{basic_time_derivative_computation}. We obtain \begin{equation}
		\begin{split}
		& \dt \EX[\widetilde{F}(X(t))] = \frac{\e^{-t}}{4} \sum_{a,b} \Big(V_{ab} - \frac{1}{N}\Big) (1 + \delta_{ab})^2 \EX \Bigg[\frac{\partial^2 (\widetilde{F}(X(t)))}{\partial h_{ab}(t)^2}\Bigg].
		\end{split}
		\label{cumulant_expanded_gaussians}
	\end{equation} 

	Using the derivative rules \eqref{derivative_rule_gen}--\eqref{func_deriv_rule} we can expand the derivatives. One of the terms we obtain come from the derivative hitting $\Dim G_{ab}$ in the second line of \eqref{basic_time_derivative_computation}. From the derivative hitting $\Dim G_{ab}$ we will get two terms, one of which is \begin{equation}
	\frac{\e^{-t}}{4} \sum_{a, b} \left(V_{ab} - \frac{1}{N}\right) \sum_{\beta \in [[1, K]]} \EX\left[(\partial_{\beta} F)(\widetilde{\func}_{\tau}(X(t))) \sum_{\mathcal{J} \in J} \partial_{\mathcal{J}} \func_\tau^{\beta}(X(t)) \frac{N}{\pi} \int_I (\im G_{qq}) \varrho_{\alpha(\mathcal{J})}'(Y(t)) \Dim (G_{aa} G_{bb}) \di E \right].	
	\label{example_term}
	\end{equation} 
	Note that the factor $(1 + \delta_{ab})^2$ is cancelled. An important property of the above term is that if the expression inside the expectation did not contain,   \eg the index $a$, the term evaluates to $0$, since $\sum_{a} (V_{ab} - 1/N) = 0$. For the remainder of the section, we omit writing out the arguments of $\partial_{\beta} F$, $\partial_{\mathcal{J}} \func_\tau^{\beta}$, and $\varrho_{\alpha(\mathcal{J})}$ (and its derivatives). They will always be $\widetilde{\func}_\tau(X(t))$, $X(t)$, and $Y(t)$, respectively.

	We use the same notation as \cite{schnelli_xu_generalized_22} and introduce, \begin{equation}
		\Pi := (1/N)_{1 \leq i , j \leq N}, \qquad \mathring{V} := V - \Pi. 
		\label{V_ring_def}
	\end{equation} In this notation, the factor $(V_{ab} - 1/N) = \mathring{V}_{ab}$. One can check that $\mathring{V}$ is a real symmetric matrix that commutes with $V$ and for any $k \geq 1$ \begin{equation}
		\mathring{V}^k = V^k - \Pi, \qquad \mathring{V}V^k = \mathring{V}^{k+1}. 
		\label{V_ring_powers_eq}
	\end{equation} Further we have the following result: \begin{lemma}[Lemma 4.2 from \cite{schnelli_xu_generalized_22}]
		For any $k \geq 1$, $\mathring{V}$ satisfies the following properties \begin{equation}
			\sum_{b = 1}^N (\mathring{V}^k)_{ab} = 0, \quad \norm{\mathring{V}^k}_{\max} \leq \frac{C_0}{N},
		\label{T_prop_lemma_eq}
		\end{equation} where $C_0 = 2(\Csup + 1)$ with $\Csup \geq \Cinf > 0$ given in \eqref{mean_field_condition}. Furthermore, there exists a constant $c_0$ with $\Cinf \leq c_0 \leq 1$ such that \begin{equation}
			\norm{\mathring{V}^k}_{\max} \leq \norm{\mathring{V}^k} \leq (1 - c_0)^k. 
			\label{T_powers_norm_decay}
		\end{equation}
		\label{T_prop_lemma}
	\end{lemma}

	We can use \eqref{T_prop_lemma_eq} to replace $G_{aa}$ in \eqref{example_term} with $G_{aa} - m_{\mathrm{sc}}$, since $G_{aa} = G_{aa} - m_{\mathrm{sc}} + m_{\mathrm{sc}}$. And if we replace $G_{aa}$ in \eqref{example_term} by $m_{\mathrm{sc}}$, which does not depend on $a$, we get $0$. We can repeat the same procedure for $b$. Hence, \begin{equation}
		\begin{split}
		& \frac{\e^{-t}}{4} \sum_{a, b} \left(V_{ab} - \frac{1}{N}\right) \sum_{\beta \in [[1, K]]} \EX\Bigg[(\partial_{\beta} F) \sum_{\mathcal{J} \in J} (\partial_{\mathcal{J}} \func_\tau^{\beta}) \frac{N}{\pi} \int_I (\im G_{qq}) \varrho_{\alpha(\mathcal{J})}' \Dim (G_{aa} G_{bb}) \di E \Bigg] \\
		& = \frac{\e^{-t}}{4} \sum_{a, b} \left(V_{ab} - \frac{1}{N}\right) \sum_{\beta \in [[1, K]]} \EX\Big[(\partial_{\beta} F) \sum_{\mathcal{J} \in J} (\partial_{\mathcal{J}} \func_\tau^{\beta}) \frac{N}{\pi} \int_I (\im G_{qq}) \varrho_{\alpha(\mathcal{J})}' \Dim (G_{aa} - m_{\mathrm{sc}}) (G_{bb} - m_{\mathrm{sc}}) \di E \Big].
		\end{split}
		\label{exapmle_expand_term}
	\end{equation}

	Call the term \eqref{example_term} $T$. Recall the random control parameter $\Psi$ introduced in \eqref{Psi_def}, and the local law bound on $\Psi$ in \eqref{Psi_bounds}. If we bound $A$ naively using the local law, the bound on $\partial_{\mathcal{J}} \func_{\tau}^{\beta}$ in \eqref{func_assumptions}, the bound on $\partial_{\beta} F$ in \eqref{deriv_F_bound_assump}, and the second equation in \eqref{T_prop_lemma_eq} we obtain 
	\begin{equation*}
		\begin{split}
		\abs{T} & \leq C \e^{-t} \sum_{a, b} \mathring{V}_{ab} \EX\left[\abs{\frac{N}{\pi} \int_I (\im G_{qq}) \varrho_{\alpha(\mathcal{J})}' \Dim (G_{aa} - \msc) (G_{bb} - \msc) \di E} \right] \\
		& \leq C \e^{-t} \sum_{a, b} \mathring{V}_{ab} N \cdot N^{-2/3 + \epsilon} \Psi^3 \\
		& \leq \e^{-t} N^{1/3 + C \epsilon}.  
		\end{split}	
	\end{equation*} This bound is too big. The idea from \cite{schnelli_xu_generalized_22} that can be adapted to our setting is to recursively expand in the indices with inhomogeneous weights, \ie the indices that occur in the $\mathring{V}$. 

	We demonstrate how to expand $A$. We use the identities $1  = -z \msc (z) - \msc(z)^2$ and $z G_{aa} = (HG)_{aa} - 1$. First, replace the $G_{aa} - \msc$ by $G_{aa}$ and insert a $1$. Doing this we obtain \begin{equation*}
		\begin{split}
		T & =  -\frac{1}{4} \e^{-t} \sum_{a, b} \mathring{V}_{ab} \sum_{\beta \in [[1, K]]} \EX\left[(\partial_{\beta} F)  \sum_{\mathcal{J}} (\partial_{\mathcal{J}} \func_\tau^{\beta}) \frac{N}{\pi} \int_I (\im G_{qq}) \varrho_{\alpha(\mathcal{J})}' \Dim (z \msc + \msc^2) G_{aa} (G_{bb} - \msc) \di E \right] \\
		& = \frac{1}{4} \e^{-t} \sum_{a, b, c} \mathring{V}_{ab} \sum_{\beta \in [[1, K]]} \EX\left[h_{ac}(t) (\partial_{\beta} F)  \sum_{\mathcal{J}} (\partial_{\mathcal{J}} \func_\tau^{\beta}) \frac{N}{\pi} \int_I (\im G_{qq}) \varrho_{\alpha(\mathcal{J})}' \Dim \msc G_{ca} (G_{bb} - \msc) \di E \right]  \\
		& \quad \, + \frac{1}{4} \e^{-t} \sum_{a, b, c} \mathring{V}_{ab} \sum_{\beta \in [[1, K]]}  \EX\left[(\partial_{\beta} F)  \sum_{\mathcal{J}} (\partial_{\mathcal{J}} \func_\tau^{\beta}) \frac{N}{\pi} \int_I (\im G_{qq}) \varrho_{\alpha(\mathcal{J})}' \Dim \msc^2 G_{aa} (G_{bb} - \msc) \di E \right] \\
		& = \frac{1}{4} \e^{-t} \sum_{a, b, c} \mathring{V}_{ab} c^{(2)}(h_{ac}(t)) \sum_{\beta \in [[1, K]]}  \EX\left[\frac{\partial}{\partial h_{ac}(t)} \left((\partial_{\beta} F)  \sum_{\mathcal{J}} (\partial_{\mathcal{J}} \func_\tau^{\beta}) \frac{N}{\pi} \int_I (\im G_{qq}) \varrho_{\alpha(\mathcal{J})}' \Dim \msc G_{ca} (G_{bb} - \msc) \di E \right) \right]  \\
		& \quad \, + \frac{1}{4} \e^{-t} \sum_{a, b} \mathring{V}_{ab} \sum_{\beta \in [[1, K]]}  \EX\left[(\partial_{\beta} F)  \sum_{\mathcal{J}} (\partial_{\mathcal{J}} \func_\tau^{\beta}) \frac{N}{\pi} \int_I (\im G_{qq}) \varrho_{\alpha(\mathcal{J})}' \Dim \msc^2 G_{aa} (G_{bb} - \msc) \di E \right] 
		\end{split}
	\end{equation*} 
	We note that the cumulant of $h_{ac}(t)$ is given by, \begin{equation*}
		c^{(2)}(h_{ac}(t)) = \Big(\frac{\e^{-t}}{N} + (1- \e^{-t}) V_{ac}\Big) (1 + \delta_{ac}). 
	\end{equation*} The derivatives with respect to $h_{ac}(t)$ will generate a factor $1/(1 + \delta_{ac})$. We introduce a real symmetric time-dependent matrix \begin{equation}
		(V(t))_{ac} := \frac{c^{(2)}(h_{ac}(t))}{1 + \delta_{ac}} = \frac{\e^{-t}}{N} + (1-\e^{-t}) V_{ac}, \quad 1 \leq a, c \leq N. 
		\label{V_scnd_cumulant}
	\end{equation} One can check that $V(t)$ is a symmetric and doubly stochastic matrix. Additionally, we have \begin{equation*}
		V(t) = (1 - \e^{-t}) \mathring{V} + \Pi, \quad V(t)\mathring{V} = \mathring{V}V(t) = (1 - \e^{-t}) \mathring{V}^2. 
	\end{equation*} 
	After computing $\frac{\partial}{\partial h_{ac}(t)}$ in the equation above we will get one leading term, and several non-leading. The leading term will be when the derivative hits $G_{ca}$ and generates $G_{aa} G_{cc}$. All other factors that the derivative hits will give rise to terms with at least 4 off-diagonal Green function entries, meaning that the naive local law estimate will be improved compared to the one we have for $T$. So the leading terms will be \begin{equation*}
		\begin{split}
			T & = - \frac{1}{4} \e^{-t} \sum_{a, b, c} \mathring{V}_{ab} (V(t))_{ac} \sum_{\beta \in [[1, K]]} \EX\Bigg[(\partial_{\beta} F)  \sum_{\mathcal{J}} \partial_{\mathcal{J}} \func_\tau^{\beta}\frac{N}{\pi} \int_I (\im G_{qq}) \varrho_{\alpha(\mathcal{J})}' \Dim \msc G_{aa} G_{cc} (G_{bb} - \msc) \di E \Bigg]  \\
		& \quad \, + \frac{1}{4} \e^{-t} \sum_{a, b} \mathring{V}_{ab} \sum_{\beta \in [[1, K]]}  \EX\Bigg[(\partial_{\beta} F)  \sum_{\mathcal{J}} \partial_{\mathcal{J}} \func_\tau^{\beta} \frac{N}{\pi} \int_I (\im G_{qq}) \varrho_{\alpha(\mathcal{J})}' \Dim \msc^2 G_{aa} (G_{bb} - \msc) \di E \Bigg] \\
		& \quad \, + O\Bigg(\e^{-t} N^{C \epsilon} N^{4/3} \Big(\frac{1}{N \tilde{\eta}}\Big)^{4}\Bigg). 
		\end{split}
	\end{equation*} 
	Continuing we use that $\sum_{c} (V(t))_{ac} = 1$, so we may insert it in the second term of $A$ to see cancellation. Ignoring the non-leading terms we get:
	\begin{equation*}
		\hspace*{-1cm}
		\begin{split}
			T & = - \frac{1}{4} \e^{-t} \sum_{a, b, c} \mathring{V}_{ab} (V(t))_{ac} \EX\left[(\partial_{\beta}\widetilde{F})  \sum_{\mathcal{J}} (\partial_{\mathcal{J}} \func_\tau^{\beta}) \frac{N}{\pi} \int_I (\im G_{qq}) \varrho_{\alpha(\mathcal{J})}' \Dim \msc G_{aa} (\msc - G_{cc}) (G_{bb} - \msc) \di E \right]  \\
			& = \frac{1}{4} \e^{-t} \sum_{a, b, c} \mathring{V}_{ab} (V(t))_{ac} \EX\left[ (\partial_{\beta}\widetilde{F})  \sum_{\mathcal{J}} (\partial_{\mathcal{J}} \func_\tau^{\beta}) \frac{N}{\pi} \int_I (\im G_{qq}) \varrho_{\alpha(\mathcal{J})}' \Dim \msc^2 (\msc - G_{cc}) (G_{bb} - \msc) \di E \right] 	\\
			& \quad \, + \frac{1}{4} \e^{-t} \sum_{a, b, c} \mathring{V}_{ab} (V(t))_{ac} \EX\left[ (\partial_{\beta}\widetilde{F}) \sum_{\mathcal{J}} (\partial_{\mathcal{J}} \func_\tau^{\beta}) \frac{N}{\pi} \int_I (\im G_{qq}) \varrho_{\alpha(\mathcal{J})}' \Dim \msc (G_{aa} - \msc) (\msc - G_{cc}) (G_{bb} - \msc) \di E \right]  \\
			& =	\frac{1}{4} \e^{-t} (1 - \e^{-t}) \sum_{b, c} (\mathring{V}^2)_{bc} \EX\left[ (\partial_{\beta}\widetilde{F})  \sum_{\mathcal{J}} (\partial_{\mathcal{J}} \func_\tau^{\beta}) \frac{N}{\pi} \int_I (\im G_{qq}) \varrho_{\alpha(\mathcal{J})}' \Dim \msc^2 (\msc - G_{cc}) (G_{bb} - \msc) \di E \right] 	\\
			& \quad \, + \frac{1}{4} \e^{-t} \sum_{a, b, c} \mathring{V}_{ab} (V(t))_{ac} \EX\left[ (\partial_{\beta}\widetilde{F}) \sum_{\mathcal{J}} (\partial_{\mathcal{J}} \func_\tau^{\beta}) \frac{N}{\pi} \int_I (\im G_{qq}) \varrho_{\alpha(\mathcal{J})}' \Dim \msc (G_{aa} - \msc) (\msc - G_{cc}) (G_{bb} - \msc) \di E \right].
		\end{split}
	\end{equation*} 
	In the last step we used that the expression inside the expectation no longer is dependent of $a$, since we replaced it my $\msc$. Because of this we can rewrite $\sum_a \mathring{V}_{ab} (V(t))_{ac} = (\mathring{V} V(t))_{bc} = (1 - \e^{-t}) (\mathring{V}^2)_{bc}$. Note that the second term in the above expression also is non-leading since we gained the (small) factor $(\msc - G_{cc})$ compared to the term $A$ we started with. 
	But also the first term in the above expression is almost the same as $A$, except additional factors $\msc^2$ and $(1 - \e^{-t})$, and $\mathring{V}$ replaced by $\mathring{V}^2$. The factors $\msc^2$ and $(1 - \e^{-t})$ cause no trouble, since $\abs{\msc(z)} \leq 1$ in the regime we are considering, and of course $\abs{1 - \e^{-t}} \leq 1$ too. 
	
	The key point in the above expansion is that $\mathring{V}^2$ is smaller than $\mathring{V}$, see \eqref{T_powers_norm_decay}. The idea is to recursively repeat the above expansion $\widetilde{k} := C \log N$ times, since then $\norm{\mathring{V}^{\widetilde{k}}}_{\max} \leq (1 - c_0)^{C \log N} \leq N^{-D}$, where $D$ can be chosen arbitrarily large by adjusting $C$ accordingly. Similar expansions can also be carried out on the non-leading terms. 
	
	The details are analogous to \cite{schnelli_xu_generalized_22}, but there are some extra technicalities mainly arising from the $G$-factors containing the $q_i$. To handle the recursive expansions described above we introduce a general form of terms, similar to \eqref{general_term}. 
	\begin{definition} [General form of term]
	Let $\mathcal{I}$ be a length $d$ multi-index of free summation indices $i_1, \dots, i_d$, with $d \geq 2$. Let $\mathfrak{m}_0$ be an integer with $2 \leq \mathfrak{m}_0 \leq d$, and introduce integers $k_p \geq 1$ for $1 \leq p \leq \mathfrak{m}_0 - 1$. As before, let $\#_{\beta}$ be an integer and $(r_k)_{k=1}^{\#_{\beta}} \in (\mathbb{Z}_{\geq 1})^{\#_{\beta}}$ be multi-indices. Finally let $g(t) : \mathbb{R}^+ \to \mathbb{R}$ be a time factor satisfying $g(t) \leq C \e^{-t}$. The above data defines the sequence of functions $\R^+ \to \R$,  \begin{equation}
			t \mapsto g(t) \frac{1}{N^{d- \mathfrak{m}_0}} \sum_{\mathcal{I}} \prod_{p=1}^{\mathfrak{m}_0 - 1} (\mathring{V}^{k_p})_{i_p i_{p+1}} \sum_{\beta \in [[1, K]]^{\#_{\beta}}} \EX\left[(\partial_{\beta} F) \prod_{k=1}^{\#_{\beta}} \Bigg(\sum_{\mathcal{J} \in J^{r_k}} (\partial_{\mathcal{J}} \func_\tau^{\beta_k}) \prod_{s=1}^{r_k} \inFac(k, s, \mathcal{I}, \mathcal{J}) \Bigg)\right],
		\label{general_term2}
	\end{equation} where $\inFac(k, s, \mathcal{I}, \mathcal{J})$ denotes inner factors (\cf \eqref{inner_fac_def}). Let $q_i$ and $q_j$ denote the basis vectors $q_{i(\mathcal{J}(s))}$ and $q_{j(\mathcal{J}(s))}$ respectively, and let $\alpha = \alpha(\mathcal{J}(s))$. Each inner factor is defined by the following data: 
	\begin{itemize}
		\item  Natural numbers $n_1^{(\mathrm{off})}$, $n_1^{(\mathrm{diag})}$, $n_1^{(\mathrm{inner})}$, $y_1$, $n_2$, and $\hat{y}_l$, $\hat{n}_l^{(\mathrm{off})}$, $\hat{n}_l^{(\mathrm{diag})}$ for $1 \leq l \leq n_2$. Further, we enforce the condition, 
		\begin{equation}
			n_1^{(\mathrm{inner})} + n_1^{(\mathrm{off})}+ n_1^{(\mathrm{diag})} + \sum_{l=1}^{n_2} \hat{n}_l^{(\mathrm{off})} + \hat{n}_l^{(\mathrm{diag})} \geq 2. 
			\label{no_empty_inner_fac}
		\end{equation}
		\item $\mathcal{I}\cup\{q_i, q_j\}$-valued indices $u_i, v_i , u_i \neq v_i,  1 \leq i \leq n_1^{(\mathrm{off})}$, and $w_i, 1 \leq i \leq n_1^{(\mathrm{diag})}$.
		\item $\mathcal{I}$-valued indices $x_i, 1 \leq i \leq n_1^{(\mathrm{inner})}$.
		\item For $1 \leq l \leq n_2$: $\mathcal{I}$-valued indices $\hat{u}_j, \hat{v}_j,\hat{u}_j \neq \hat{v}_j,  1 \leq j \leq \hat{n}_l^{(\mathrm{off})}$ and $\hat{w}_j, 1 \leq j \leq \hat{n}_l^{(\mathrm{diag})}$. 
		\item $\{i, j\}$-valued indices $z_l$ for $1 \leq l \leq n_1^{(\mathrm{inner})}$.
	\end{itemize}
		The above data uniquely defines the inner factor, 
	\begin{equation}
		\frac{N}{\pi} \int_{I} \prod_{i=1}^{n_1^{(\mathrm{inner})}} \inner{q_{z_i}, e_{x_i}} \im \Big((\msc)^{y_1} \prod_{i=1}^{n_1^{(\mathrm{off})}} G_{u_i v_i} \prod_{i=1}^{n_1^{(\mathrm{diag})}} (G_{w_i w_i}- \msc) \Big) \varrho^{(n_2)}_{\alpha} \prod_{l=1}^{n_2} \Dim \Big((\msc)^{\hat{y}_l} \prod_{j=1}^{\hat{n}_{l}^{(\mathrm{off})}} G_{\hat{u}_j \hat{v}_j} \prod_{j=1}^{\hat{n}_{l}^{(\mathrm{diag})}} (G_{\hat{w}_j \hat{w}_j} - \msc)\Big) \di E.
		\label{inner_fac_def2}
	\end{equation} 

	For each inner factor we restrict $q_i$ and $q_j$ to occur exactly once in the $q_{z_i}$, $u_i, v_i$, and $q_i$ and $q_j$ cannot occur together in a $(u_i, v_i)$-pairs unless $n_1^{(\mathrm{off})} = 1$ and $n_1^{(\mathrm{diag})} = 0$. Also, if we have equality in \eqref{no_empty_inner_fac} and $n_1^{(\mathrm{inner})} = 2$ we require $y_1 \geq 1$. Additionally, we restrict each index $i_j \in \mathcal{I}$ to occur exactly twice in the indices $x_i, u_i, v_i, w_i, \hat{u}_j, \hat{v}_j, \hat{w}_j$ among all inner factors. We say that the indices $i_1, \dots, i_{\mathfrak{m}_0}$ have inhomogeneous weights $\prod_{p=1}^{\mathfrak{m}_0} (\mathring{V}^{k_p})_{i_p i_{p+1}}$, and that the remaining indices have homogeneous weights $1/N$. The indices $i_1$ and $i_{\mathfrak{m}_0}$ are special, since for these indices $\sum_{i_1} \prod_{p=1}^{\mathfrak{m}_0 - 1} (\mathring{V}^{k_p})_{i_p i_{p+1}} = \sum_{i_{\mathfrak{m}_0}} \prod_{p=1}^{\mathfrak{m}_0 - 1} (\mathring{V}^{k_p})_{i_p i_{p+1}}  = 0$, in particular, if these indices are eliminated from the inner factors, the term is exactly zero. As in the previous section, we denote $\sum_{k=1}^{\#_{\beta}} r_k =: r_{\sum}$. 

	We will usually denote a term of the form \eqref{general_term2} as $T$, and denote the data related to $T$ as $d(T)$, $(\#_{\beta}) (T)$, $\mathfrak{m}_0(T)$, etc. 
	\label{general_term_def2}
	\end{definition} 

	Because of the requirement that each index $i_1, \dots, i_d$ occurs exactly twice among the Green function entries we have that \begin{equation*}
		2 \sum_{k=1}^{\#_{\beta}} \sum_{s = 1}^{r_k} \Big(- 1 +  n_1^{(\mathrm{inner})}(k, s) + n_1^{(\mathrm{off})}(k, s) + n_1^{(\mathrm{diag})}(k, s) + \sum_{l=1}^{n_2(k, s)} \hat{n}_l^{(\mathrm{off})}(k, s) + \hat{n}_l^{(\mathrm{diag})}(k, s)\Big) = 2d,
	\end{equation*} where the $-1$ in the sum comes from the $q_i$ and $q_j$ in the $u_i$ and $v_i$. To shorten notation, introduce \begin{equation*}
		\begin{split}
			& n_1^{(\mathrm{tot})} := \sum_{k=1}^{\#_{\beta}} \sum_{s=1}^{r_k} n_1^{(\mathrm{off})}(k, s) + n_1^{(\mathrm{diag})}(k, s), \\
			& n_2^{(\mathrm{tot})} := \sum_{k=1}^{\#_{\beta}} \sum_{s=1}^{r_k} \sum_{l=1}^{n_2(k, s)} \hat{n}_l^{(\mathrm{off})}(k, s) + \hat{n}_l^{(\mathrm{diag})}(k, s), \\
			& n_1^{(\mathrm{intot})} := \sum_{k=1}^{\#_{\beta}} \sum_{s=1}^{r_k} n_1^{(\mathrm{inner})}, \\ 
			& n_{\sum} := n_1^{\mathrm{(tot)}} + n_2^{\mathrm{(tot)}} + n_1^{\mathrm{intot}}. 
		\end{split}
		\end{equation*}
		Using this notation, we have \begin{equation}
			n_{\sum} - r_{\sum}= d.
			\label{n_d_relation}
		\end{equation} 

	Next we provide a naive local law bound of terms of the form \eqref{general_term2} analogous to Lemma \ref{local_law_bound_lemma}. 

	\begin{lemma}
		Let $T$ be a general term of the form \eqref{general_term2}, denote $\hat{k} := \max_{1 \leq p \leq \mathfrak{m}_0} k_p$, let $\mathcal{M} = N^{1/3 + \epsilon + \xi_6 \epsilon}/(N \tilde{\eta})$, for some constant $0 < \xi_6 \leq 1$. 
			In the case $n_1^{(\mathrm{intot})} < 4$ we have for sufficiently large $N$,  \begin{align}
			& \abs{T} \leq g(t) N \mathcal{M}^{r_{\sum}} \Big(\prod_{l=1}^{\#_{\beta}} \frac{N^{\Fasump \epsilon}}{\tau^{r_l - 1}} \Big) \Bigg( \Big(\frac{1}{N \tilde{\eta}}\Big)^{d} + \frac{1}{N^2 \tilde{\eta}}  \Bigg), \label{degree_small}\\
			& \abs{T} \leq g(t) N^2 (1 - c_0)^{\hat{k}} \mathcal{M}^{r_{\sum}}  \Big(\prod_{l=1}^{\#_{\beta}} \frac{N^{\Fasump \epsilon}}{\tau^{r_l - 1}} \Big) \Bigg(\Big(\frac{1}{N \tilde{\eta}}\Big)^{d} + \frac{1}{N^2 \tilde{\eta}} \Bigg). \label{many_expansions}
		\end{align}
		In the case $n_1^{(\mathrm{intot})} \geq 4$, we have for sufficiently large $N$,
		\begin{equation}
			\abs{T} \leq g(t) \frac{1}{N} \mathcal{M}^{r_{\sum}} \Big(\prod_{l=1}^{\#_{\beta}} \frac{N^{\Fasump \epsilon}}{\tau^{r_l - 1}} \Big).
			\label{big_n1_intot}
		\end{equation}
		\label{step2_local_law_bound_lemma}
	\end{lemma}
	\begin{proof}
	Because of the similarity to Lemma \ref{local_law_bound_lemma} we only provide an outline of the proof. We follow the same strategy as the proof of Lemma \ref{local_law_bound_lemma}. We use that each index $i_1, \dots, i_d$ only occurs twice, so the discussion on worst configuration of the right side of \eqref{worst_p_dist_analysis} is significantly simplified, the worst configuration is simply when $p_{\Psi}$ is as large as possible. For \eqref{degree_small} and \eqref{many_expansions} we technically bound the term with $\inner{q_{z_i}, e_{x_i}}$ replaced by $G_{q_{z_i} x_i}$. By bounding this altered term, we automatically obtain a bound for the original term since $\inner{q_{z_i}, e_{x_i}} \msc$ is the deterministic approximation of $G_{q_{z_i} x_i}$. 

	To show \eqref{degree_small}, use \eqref{T_prop_lemma_eq}, which corresponds to the case $d_N = 1$ in Lemma \ref{local_law_bound_lemma}. The error term $ \frac{1}{N^2 \tilde{\eta}}$ comes from the incidental index case, and that we cannot have empty inner factors, due to \eqref{no_empty_inner_fac}. 

	For \eqref{many_expansions}, we instead use \eqref{T_powers_norm_decay} on one of the $(\mathring{V}^{k_p})_{i_p i_{p+1}}$ of $P$, and \eqref{T_prop_lemma_eq} on the other $(\mathring{V}^{k_p})_{i_p i_{p+1}}$-factors, this loosely corresponds to the case $d_N = 2$ in Lemma \ref{local_law_bound_lemma}, but we have an extra factor $(1 - c_0)^{\hat{k}}$.  

	Finally, for \eqref{big_n1_intot}, we use the four factors of $\inner{q_{z_i}, e_{x_i}}$ to gain an $N^{-2}$ and use the $\im$ to obtain an $\mathcal{M}$-bound on the remaining part of the inner factors. 
	\end{proof}
	Crucially, the relation \eqref{n_d_relation} imply that if $d$ increases, the bounds in \eqref{degree_small} and \eqref{many_expansions} improve almost by a factor $1/(N \tilde{\eta})$. 

	Generalizing the example where we expanded \eqref{exapmle_expand_term} and the expansion mechanism in \cite{schnelli_xu_generalized_22} we show three expansion formulas. Let $\mathcal{T}_{d'}$ denote all terms $T$ of the form \eqref{general_term2} with $d(T) \geq d'$, and introduce, 
	\begin{equation}
		\widetilde{k} := \ceil{\frac{10 \log N}{\log (1 - c_0)^{-1}}},
		\label{log_times_expansion_iteration}
	\end{equation} 	
	which is the number of expansions we need to perform on the leading term before \eqref{many_expansions} becomes effective. We have the following different cases for our expansion mechanism: \begin{enumerate}[label=(\alph*)]
		\item ($i_1$ occurs in a diagonal Green function entry): In this case we can expand $T_d \in \mathcal{T}_d$ as \begin{equation}
			T_d = \sum_{T_{d'} \in \mathcal{T}_{d'}, d' = d + 1} T_{d'} + O(\e^{-t} N^{-8}). \label{diagonal_expansion}
		\end{equation} The sum above contains at most $11 \widetilde{k} d$ terms. The error $O(\e^{-t} N^{-8})$ comes from \eqref{many_expansions} with $\hat{k} \geq \widetilde{k}$. 
		\item ($i_1$ occurs in an off-diagonal Green function with another $i_j$): Here we can expand $T_d \in \mathcal{T}_d$ as \begin{equation}
			T_d = \sum_{T_{d'} \in \mathcal{T}_{d'}, d' = d + 1} T_{d'} + O\Bigg( \e^{-t} N^{\epsilon} \mathcal{M}^{r_{\sum}}  \Big(\prod_{l=1}^{\#_{\beta}} \frac{N^{\Fasump \epsilon}}{\tau^{r_l - 1}} \Big) \Big(\indicator (d \geq 3) \Big(\frac{1}{N \tilde{\eta}}\Big)^{d-2} + \frac{1}{N^2 \tilde{\eta}} \Big)+ \indicator (d = 2) \Big(\Big(\frac{1}{N \tilde{\eta}}\Big)^{d-1} \Big)\Bigg) + O(\e^{-t} N^{-8}). \label{off_diagonal_expansion}
		\end{equation} The sum above contains at most  $11 \widetilde{k} d$ terms.
		\item ($i_1$ and $i_{\mathfrak{m}_0}$ occur only in the Green function entries with $q_i$ or $q_j$, or in the $x_i$, and not only in the $x_i$): Assume wlog that $i_1$ occurs with $q_i$. Here we get \begin{equation}
			T_d = \sum_{T_{d'} \in \mathcal{T}_{d'}, d' = d + 1} T_{d'} + T_d[G_{q_i i_1} \to \inner{q_i, e_{i_1}} \msc] + O(\e^{-t} N^{-8}).
			\label{off_diagonal_q_expansion}
		\end{equation} The sum above contains at most  $11 \widetilde{k} d$ terms. 
	\end{enumerate}

	\begin{proof}[Proof of Lemma \ref{gaussian_to_gaussian_lemma}]
	We use \eqref{diagonal_expansion}--\eqref{off_diagonal_q_expansion} to show \eqref{second_comparison_step}. From the time derivative in \eqref{cumulant_expanded_gaussians} we get several terms with $d = 2$. From \eqref{diagonal_expansion}--\eqref{off_diagonal_q_expansion} we see that we can expand each term on the form \eqref{general_term2} where $i_1$ and $i_{\mathfrak{m}_0}$ not only occur in the $n_1^{(\mathrm{inner})}$-products. This means that we can expand all terms of the form \eqref{general_term2} that have $n_{1}^{(\mathrm{intot})} < 4$. We propose the following expansion scheme: Start with the terms given by \eqref{cumulant_expanded_gaussians}. There are a finite constant number of such terms, and they all have $d = 2$ and $n_1^{(\mathrm{intot})} = 0$. Fix some large integer $D$ and expand all terms using \eqref{diagonal_expansion}--\eqref{off_diagonal_q_expansion} until either they have $d = D$ or $n_1^{(\mathrm{intot})} = 4$. Each time we use \eqref{diagonal_expansion}--\eqref{off_diagonal_q_expansion}, the term we expand, $P_d$ is expressed as a linear combination of terms with either $n_1^{(\mathrm{intot})} \to n_1^{(\mathrm{intot})} + 1$ or $d \to d + 1$. We need to count the number of terms we get when carrying out this expansion. 

	Following the above expansion mechanism the number of terms $T$ with $$(n_1^{(\mathrm{intot})}(T), d(T)) = (n', d'), \; 0 \leq n' \leq 4, 2 \leq d' \leq D,$$ is bounded by \begin{equation*}
		(\# \mathrm{start terms}) \binom{n' + d' - 2}{d' - 2} (11 \widetilde{k} d)^{d-2},
	\end{equation*} since there are at most $11 \widetilde{k} d$ terms in the sums in \eqref{diagonal_expansion}--\eqref{off_diagonal_q_expansion}, and we get only one term with $n_{1}^{(\mathrm{intot})} \to n_1^{(\mathrm{intot})} + 1$ from \eqref{off_diagonal_q_expansion}. 
	
	First we bound the terms for which the terminating condition is $n_1^{(\mathrm{intot})} = 4$. We sum over the possible $d$ of the final term and use \eqref{big_n1_intot}. We also use that for all such terms, $g(t) \leq \e^{-t}$, $r_{\sum} \leq d$, $\#_{\beta} \leq d$ and since $1/\tau \leq N^{\Fasump \epsilon}$, the term for which $\Big(\prod_{l=1}^{\#_{\beta}} \frac{N^{\Fasump \epsilon}}{\tau^{r_l - 1}} \Big)$ is maximized is when $\#_{\beta} = d$. We get the bound \begin{equation*}
		B_1 := \sum_{d'=2}^{D-1} C (11\widetilde{k}d')^{d'-2} \e^{-t} \frac{1}{N} \mathcal{M}^{d'} N^{\Fasump d' \epsilon}. 
	\end{equation*}
	 Next we simplify the above expression and obtain, \begin{equation*}
		\begin{split}
		B_1 & \leq \frac{C \e^{-t} (D-2) (11(D-1) \widetilde{k})^{D-3} \mathcal{M}^{D-1} N^{\Fasump (D - 1) \epsilon}}{N} \\ 
		& \leq C_D \e^{-t} \frac{N^{\epsilon} (N^{16 \epsilon} N^{\Fasump \epsilon})^{D-1}}{N},
		\end{split}
	\end{equation*} where $C_D$ is a constant in $N$, depending on $D$, and we used that $\mathcal{M} \leq N^{16 \epsilon}$, and $\widetilde{k}^{D-3} \leq N^{\epsilon}$ for sufficiently large $N$.  

	Second we bound the terms with terminating condition $d = D$. We sum over the possible $n_1^{(\mathrm{intot})}$ and proceed as with $B_1$. We get \begin{equation*}
		\begin{split}
		B_2 & := \sum_{n_1^{(\mathrm{intot})=0}}^{3} C (11\widetilde{k}D)^{D-2} \e^{-t} N \mathcal{M}^{D} N^{\Fasump D \epsilon} \Bigg( \Big(\frac{1}{N \tilde{\eta}}\Big)^{D} + \frac{1}{N^2 \tilde{\eta}}  \Bigg) \\
		& \quad \leq C \e^{-t} \frac{\widetilde{k}^{D-2} N^{D (\Fasump + 16) \epsilon}}{N \tilde{\eta}} \leq C \e^{-t} \frac{N^{\epsilon} N^{D (\Fasump + 16) \epsilon + 14 \epsilon}}{N^{1/3}},
		\end{split}
	\end{equation*} 
	where we assumed that we choose $D$ sufficiently large for $1/(N \tilde{\eta})^D \leq 1/(N^2 \tilde{\eta})$. Also, we used that for sufficiently large $N$, $N^{\epsilon} \geq \widetilde{k}^{D-2}$.

	Finally, we need to account for the errors in the expansion formulas. To bound the total errors we count the number of terms in the expansion tree that get expanded. We get the bound \begin{equation*}
		\begin{split}
		B_3 & := \sum_{n' = 0}^{3} \sum_{d' = 2}^{D-1} C (11\widetilde{k}d')^{d'-2} \e^{-t} \Bigg(N^{\epsilon} \mathcal{M}^{d'} N^{\Fasump d' \epsilon} \frac{1}{N \tilde{\eta}} + N^{-8} \Bigg) \leq C_D \e^{-t} \frac{N^{2 \epsilon} N^{(\Fasump + 16)(D-1)\epsilon} N^{14 \epsilon}}{N^{1/3}}.
		\end{split}
	\end{equation*} The bounds $B_1$, $B_2$, $B_3$ are constructed so that \begin{equation*}
	\dt \EX[\widetilde{F}(X(t))] \leq B_1 + B_2 + B_3. 
	\end{equation*} By choosing $D$ large enough for $1/(N \tilde{\eta})^D \leq 1/(N^2 \tilde{\eta})$ to hold, we see that for sufficiently small $\epsilon > 0$ we have that $B_1 + B_2 + B_3 \leq \e^{-t} N^{-1/4}$, for sufficiently large $N$. So by integrating $\dt \EX[\widetilde{F}(X(t))]$ and using the bound on $B_1 + B_2 + B_3$, the proof is finished. 
	\end{proof}
	
	\subsubsection{Proof of \eqref{diagonal_expansion}}

	We start with a term $T_d$ on the form \eqref{general_term2}. Let $T_d^{(\mathrm{r})}$ denote the expression inside the expectation of $T_d$. By assumption there is a  factor $(G_{i_1 i_1} - \msc)$ in one of the inner factors. Replace the factor $G_{i_1 i_1} - \msc$ by $G_{i_1 i_1}$. Expand this inner factor using the identity $1 = -z \msc - (\msc)^2$.  Then on the term with $-z \msc$, use the identity $zG_{ii} = (HG)_{ii} - 1$. The term with the $-1$ contains no $i_1$ anymore, so it is exactly zero by \eqref{T_prop_lemma_eq}. Introduce a fresh index $c$ in the term with $(HG)_{i_1 i_1} = \sum_{c} h_{i_1 c}(t) G_{c i_1}$. Then we can cumulant expand using $h_{i_1 c}(t)$. Since $h_{i_1 c}(t)$ is Gaussian we only get the second cumulant term. If the derivative $\partial / \partial h_{i_1 c}(t)$ hits $G_{c i_1}$ we obtain one term with $-G_{cc} G_{i_1 i_1}$ and another with $-G_{c i_1} G_{c i_1}$. If the derivative hits $\partial_{\beta} F$, or $\partial_{\mathcal{J}_k} \func_{\tau}^{\beta_k}$, or any Green function factor inside the inner factors aside from $G_{c i_1}$, the new inner factor will be on the form \eqref{inner_fac_def2}, since $c$ and $i_1$ only occur in one Green function entry. Note also that the inner factor obtained with $-G_{c i_1} G_{c i_1}$ is of the form \eqref{inner_fac_def2}. Summarizing our argument so far we have
	\begin{align}
			T_d & = \frac{g(t)}{N^{d - \mathfrak{m}_0}} \sum_{\mathcal{I}} \prod_{p=1}^{\mathfrak{m}_0 - 1} (\mathring{V}^{k_p})_{i_p i_{p+1}} \EX[T_d^{(\mathrm{r})}] \nonumber \\
			& = \frac{g(t)}{N^{d - \mathfrak{m}_0}} \sum_{\mathcal{I} \cup \{c\}} \prod_{p=1}^{\mathfrak{m}_0 - 1} (\mathring{V}^{k_p})_{i_p i_{p+1}} (V(t))_{i_1 c} \EX\Big[T_d^{(\mathrm{r})}[G_{i_1 i_1} \to \msc G_{cc} G_{i_1 i_1}]\Big] \label{P_d_diag_expansion_leading_term} \\
			& \quad + \frac{g(t)}{N^{d - \mathfrak{m}_0}} \sum_{\mathcal{I} \cup \{c\}} \prod_{p=1}^{\mathfrak{m}_0 - 1} (\mathring{V}^{k_p})_{i_p i_{p+1}} (V(t))_{i_1 c} \EX\Big[T_d^{(\mathrm{r})}[G_{i_1 i_1} \to \msc G_{i_1 c} \mathrm{ + more off-diagonals}]\Big] \label{P_d_diag_expansion} \\
			& \quad + \frac{g(t)}{N^{d - \mathfrak{m}_0}} \sum_{\mathcal{I} \cup \{c\}} \prod_{p=1}^{\mathfrak{m}_0 - 1} (\mathring{V}^{k_p})_{i_p i_{p+1}} (V(t))_{i_1 c} \EX\Big[T_d^{(\mathrm{r})}[G_{i_1 i_1} \to - (\msc)^2 G_{i_1 i_1}]\Big]. \label{P_d_diag_msc2}
	\end{align} Note that in \eqref{P_d_diag_msc2} we used that $\sum_{c} (V(t))_{i_1 c} = 1$, and \eqref{V_scnd_cumulant}, so the second cumulant of $h_{i_1 c}(t)$ is contained in the $(V(t))_{i_1 c}$-factor.  The line \eqref{P_d_diag_expansion} contains all the resulting terms from the cumulant expansion except \eqref{P_d_diag_expansion_leading_term}. After splitting $(V(t))_{i_1 c} = (1 - \e^{-t})\mathring{V}_{i_1 c} + 1/N$ we obtain two new terms on the form \eqref{general_term2} for each term in \eqref{P_d_diag_expansion}, both with $d \to d + 1$. The first one will have $g(t) \to (1 - \e^{-t}) g(t)$, $\mathfrak{m}_0 \to \mathfrak{m}_0 + 1$, and $c$ as the new special index, and the second one will have $c$ as an index with homogeneous weight. 
	
	We need to analyze how many terms we obtain in \eqref{P_d_diag_expansion}. The derivative $\partial / \partial h_{i_1 c}(t)$ can either hit $\partial_{\beta} F$, or $\partial_{\mathcal{J}} \func_{\tau}^{\beta_k}$, or an inner factor. The first case gives 3 terms, the second case gives $3 (\#_{\beta}) \leq 3 d$ terms, and the third gives at most $2(2r_{\sum} + d) \leq 6d$ terms. The 2 $r_{\sum}$ comes from for the Green function factors with $q_i$, $q_j$ and the $\varrho$. So in total, \eqref{P_d_diag_expansion} consists of at most $3 + 3d + 6d \leq 11d$ terms (since $d \geq 2$). 

	Next we see cancellation between \eqref{P_d_diag_expansion_leading_term} and \eqref{P_d_diag_msc2}. We have 
	\begin{equation}
	\begin{split}
			\eqref{P_d_diag_expansion_leading_term} + \eqref{P_d_diag_msc2} & =  \frac{g(t)}{N^{d - \mathfrak{m}_0}} \sum_{\mathcal{I} \cup \{c\}} \prod_{p=1}^{\mathfrak{m}_0 - 1} (\mathring{V}^{k_p})_{i_p i_{p+1}} (V(t))_{i_1 c} \EX\Big[T_d^{(\mathrm{r})}[G_{i_1 i_1} \to \msc (G_{cc} - \msc)G_{i_1 i_1}]\Big] \\
			& = \frac{g(t)}{N^{d - \mathfrak{m}_0}} \sum_{\mathcal{I} \cup \{c\}} \prod_{p=1}^{\mathfrak{m}_0 - 1} (\mathring{V}^{k_p})_{i_p i_{p+1}} (V(t))_{i_1 c} \EX\Big[T_d^{(\mathrm{r})}[G_{i_1 i_1} \to \msc^2 (G_{cc} - \msc)]\Big] \\
			& \quad + \frac{g(t)}{N^{d - \mathfrak{m}_0}} \sum_{\mathcal{I} \cup \{c\}} \prod_{p=1}^{\mathfrak{m}_0 - 1} (\mathring{V}^{k_p})_{i_p i_{p+1}} (V(t))_{i_1 c} \EX\Big[T_d^{(\mathrm{r})}[G_{i_1 i_1} \to \msc (G_{cc} - \msc)(G_{i_1 i_1} - \msc)]\Big]. 
	\end{split} 
\label{k_1_increase_term} 
	\end{equation}
	As before we use $(V(t))_{i_1 c} = (1 - \e^{-t})\mathring{V}_{i_1 c} + 1/N$ on the last term above to obtain two terms with $d \to d + 1$ on the form \eqref{general_term2}. For the remaining term, where $G_{i_1 i_1}$ has been eliminated from the expression inside the expectation, we use precisely this, and sum with respect to $i_1$ and obtain, $$\sum_{i_1} (\mathring{V}^{k_1})_{i_1 i_2} (V(t))_{i_1 c} = (1- \e^{-t}) (\mathring{V}^{k_1 + 1})_{c i_2}.$$ Renaming $c$ to $i_1$ we have almost the same term, $T_d$ as we started with, the differences being $k_1 \to k_1 + 1$, and an extra $(\msc)^2$ and $(1 - \e^{-t})$. This term is the leading term, since its $d$ did not increase. 
	
	Repeating the above steps on the leading term,  \begin{equation*}
		\widetilde{k} := \ceil{\frac{10 \log N}{\log (1 - c_0)^{-1}}}
	\end{equation*} times, obtain at most $9\widetilde{k}d$ terms with $d \to d + 1$, and one term with the same $d$ as $T_d$, but $k_1 \to k_1 + \widetilde{k}$. Using the bound \eqref{many_expansions} on the term with $k_1 \to k_1 + \widetilde{k}$ we get the error $O(\e^{-t} N^{-8})$, and the proof of \eqref{diagonal_expansion} is concluded. 

	\subsubsection{Proof of \eqref{off_diagonal_expansion}}
	The strategy is similar to the diagonal case. Again we start with a term $T_d$, and we denote the expression inside the expectation by $T_d^{(\mathrm{r})}$. We assume that there is a $j \neq 1$ such that $G_{i_1 i_j}$ occurs in one of the inner factors. Assume that the other $i_1$ occurs together with $v_1$, the proof when the other $i_1$ occurs in the $\Dim$-products is identical.
	We comment at the end of the proof what happens when the other $i_1$ occurs in the $n_1^{(\mathrm{inner})}$-product. 
	We expand $T_d$ using the identity $1 = -z \msc - (\msc)^2$. On the term with $-z \msc$ we use the identity $zG_{i_1 i_j} = (HG)_{i_1 i_j} - \delta_{i_1 i_j}$. For the term with $(HG)_{i_1 i_j} = \sum_{c} h_{i_1 c}(t) G_{c i_j}$ we, cumulant expand. When the derivative does not hit $G_{c i_j}$ or $G_{i_1 v_1}$ We get one off-diagonal factor more, and after splitting $(V(t))_{i_1 c} = (1 - \e^{-t}) \mathring{V}_{i_1 c} + 1/N$ we obtain terms of the form \eqref{general_term2} with $d \to d + 1$. 
	
	So we need to handle the cases when $\partial / \partial h_{i_1 c}(t)$ hits $G_{c i_j}$ and $G_{i_1 v_1}$. When the derivative hits $G_{c i_j}$ we obtain $- G_{cc} G_{i_1 i_j}$ and $- G_{c i_1} G_{c i_j}$. The latter case again creates terms of the form \eqref{general_term2} with $d \to d + 1$, so that is fine. For the term with $- G_{cc} G_{i_1 i_j}$ we replace $G_{cc}$ by $\msc$ and obtain \begin{equation*}
		-G_{cc} G_{i_1 i_j} = -\msc G_{i_1 i_j} - (G_{cc} - \msc) G_{i_1 i_j}.  
	\end{equation*} We leave the term with $G_{cc}$ replaced by $\msc$ for now, we will come back to it later in \eqref{P_d_off_diag_exp_leading1}. The term with $(G_{cc} - \msc)$ is of form \eqref{general_term2} with $d \to d + 1$. 

	When $\partial / \partial h_{i_1 c}(t)$ hits $G_{i_1 v_1}$ we obtain two terms with $-G_{i_1 i_1} G_{c v_1}$ and $-G_{i_1 c} G_{c v_1}$, the latter is of form \eqref{general_term2} with $d \to d + 1$ so that is handled, but we still have $-G_{i_1 i_1} G_{c v_1}$, we treat it in \eqref{P_d_off_diag2}. 
	
	Summarizing our argument so far, we have \begin{align}
	T_d & = \frac{g(t)}{N^{d - \mathfrak{m}_0}} \sum_{\mathcal{I}} \prod_{p=1}^{\mathfrak{m}_0 - 1} (\mathring{V}^{k_p})_{i_p i_{p+1}} \EX[T_d^{(\mathrm{r})}] \nonumber \\
			& = \frac{g(t)}{N^{d - \mathfrak{m}_0}} \sum_{\mathcal{I} \cup \{c\}} \prod_{p=1}^{\mathfrak{m}_0 - 1} (\mathring{V}^{k_p})_{i_p i_{p+1}} (V(t))_{i_1 c} \EX\Big[T_d^{(\mathrm{r})}[G_{i_1 i_j} \to (\msc)^2 G_{i_1 i_j}]\Big] \label{P_d_off_diag_exp_leading1} \\
			& \quad + \frac{g(t)}{N^{d - \mathfrak{m}_0}} \sum_{\mathcal{I} \cup \{c\}} \prod_{p=1}^{\mathfrak{m}_0 - 1} (\mathring{V}^{k_p})_{i_p i_{p+1}} (V(t))_{i_1 c} \EX\Big[T_d^{(\mathrm{r})}[G_{i_1 v_1} \to G_{i_1 i_1} G_{c v_1} \mathrm{ + extra } \msc]\Big] \label{P_d_off_diag2} \\
			& \quad + \frac{g(t)}{N^{d - \mathfrak{m}_0}} \sum_{\mathcal{I} \cup \{c\}} \prod_{p=1}^{\mathfrak{m}_0 - 1} (\mathring{V}^{k_p})_{i_p i_{p+1}} (V(t))_{i_1 c} \EX\Big[T_d^{(\mathrm{r})}[G_{i_1 i_j} \to - (\msc)^2 G_{i_1 i_j}]\Big] \label{P_d_off_diag_msc2} \\
			& \quad +  \frac{g(t)}{N^{d - \mathfrak{m}_0}} \sum_{\mathcal{I}} \prod_{p=1}^{\mathfrak{m}_0 - 1} (\mathring{V}^{k_p})_{i_p i_{p+1}} \EX\Big[T_d^{(\mathrm{r})}[G_{i_1 i_j} \to \delta_{i_1 i_j}]\Big] \label{P_d_off_diag_delta2} \\
			& \quad + \sum_{T_{d'} \in \mathcal{T}_{d'}, d' = d + 1} T_{d'} .
	\end{align}
	The first thing to notice here is that \eqref{P_d_off_diag_exp_leading1} exactly cancels \eqref{P_d_off_diag_msc2}. For \eqref{P_d_off_diag2} we use the same argument as for \eqref{k_1_increase_term}, we replace $G_{i_1 i_1}$ by $\msc + (G_{i_1 i_1} - \msc)$. The term without $i_1$ is the same as $T_d$ except $k_1 \to k_1 + 1$ and $g(t) \to (1 - \e^{-t}) g(t)$, and we have an extra $\msc$ in one inner factor. 

	Next we need to handle \eqref{P_d_off_diag_delta2}. Replacing $G_{i_1 i_j}$ by $\delta_{i_1 i_j}$ is equivalent to replacing the remaining $i_j$ by $i_1$. If the other occurrence of $i_1$ and $i_j$ is $G_{i_1 i_j}$, this factor becomes $G_{i_1 i_1}$. Otherwise, the factors where $i_j$ and $i_1$ occur will still be on the form \eqref{inner_fac_def2}. Since we can bound $|G_{i_1 i_1}|$ by $1$ with probability $\geq 1 - N^{-D}$, we lose at most 2 factors of $1/(N \tilde{\eta})$, compared to \eqref{degree_small}. Additionally, because we have one index less, we gain a factor $1/N$. So the naive local law bound is (for sufficiently large $N$) \begin{equation*}
		\abs{ \frac{g(t)}{N^{d - \mathfrak{m}_0}} \sum_{\mathcal{I}} \prod_{p=1}^{\mathfrak{m}_0 - 1} (\mathring{V}^{k_p})_{i_p i_{p+1}} \EX\Big[T_d^{(\mathrm{r})}[G_{i_1 i_j} \to \delta_{i_1 i_j}]\Big] \label{P_d_off_diag_delta}} \leq g(t) N \mathcal{M}^{r_{\sum}} \Big(\prod_{l=1}^{\#_{\beta}} \frac{N^{\Fasump \epsilon}}{\tau^{r_l - 1}} \Big) \Bigg( \Big(\frac{1}{N \tilde{\eta}}\Big)^{d-2} + \frac{1}{N^2 \tilde{\eta}}  \Bigg). 	
	\end{equation*} This is the second term on the right of \eqref{off_diagonal_expansion}, and it is sufficiently small for $d \geq 3$. We need to analyze the case $d = 2$ closer. If the other occurrence of $i_1$ and $i_j$ is not together, then we only lose one factor of $1/(N \tilde{\eta})$, and so we get the third term bound in \eqref{off_diagonal_expansion}. The remaining possible $T_d$ for which $i_1$ and $i_j$ occur together, and we have $d = 2$ are when the inner factors of $T_d$ look like \begin{align*}
		& \frac{N}{\pi} \int_{I} \im G_{q_i q_j} \varrho'(Y(t)) \Dim G_{ab}^2 \di E, \\
		& \frac{N}{\pi} \int_{I} \im G_{q_i q_j} \varrho''(Y(t)) (\Dim G_{ab})^2 \di E, \\
		&\Big(\frac{N}{\pi} \int_{I} \im G_{q_i q_j} \varrho'(Y(t)) \Dim G_{ab} \di E \Big)^2. \\
	\end{align*} When one of the $G_{ab}$ is replaced by $\delta_{ab}$ these inner factors become \begin{align*}
		& \frac{N}{\pi} \int_{I} \im G_{q_i q_j} \varrho'(Y(t)) \Dim G_{aa} \di E, \\
		& \frac{N}{\pi} \int_{I} \im G_{q_i q_j} \varrho''(Y(t)) (\Dim G_{aa}) (\Dim 1)\di E, \\
		&\Big(\frac{N}{\pi} \int_{I} \im G_{q_i q_j} \varrho'(Y(t)) \Dim G_{aa} \di E \Big) \Big(\frac{N}{\pi} \int_{I} \im G_{qq} \varrho'(Y(t)) \Dim 1 \di E \Big).
	\end{align*} The two latter factors both contain a $\Dim 1$, so they are exactly $0$. And for the first case we can use the $\im G_{aa}$ to obtain the bound $\e^{-t} \mathcal{M} N^{\Fasump \epsilon}/(N \tilde{\eta})$, which is exactly what we want. 

	As we said in the beginning of the proof we need to comment on the case when the other occurrence of $i_1$ is part of the $n_1^{(\mathrm{inner})}$-factor. In fact, this only simplifies things, since then we do not have to handle the term \eqref{P_d_off_diag2}. 

	As with the proof for \eqref{diagonal_expansion}, we repeat the above iteration \eqref{log_times_expansion_iteration}, and we have the same bound on the number of terms in the sum $\sum_{T_{d'} \in \mathcal{T}_{d'}, d' = d + 1}$. Also, we added a factor $N^{\epsilon}$ in the error terms coming from the $\delta_{i_1 v_1}$ to accommodate the $\log N$ factor in \eqref{log_times_expansion_iteration}.  

	\subsubsection{Proof of \eqref{off_diagonal_q_expansion}} 

	This proof is similar to \eqref{off_diagonal_expansion} so we skip many details and only remark on the differences. The index $i_1$ occurs in $G_{i_1 q_i}$ (or $G_{i_1 q_j}$). Expand this inner factor using the identity $1 = -z \msc - (\msc)^2$ as before. On the term with the $-z \msc$, use the identity $z G_{i_1 q_i} = \sum_{c} h_{i_1 c}(t) G_{c q_i} - \inner{q_i, e_{i_1}}$. Next, perform cumulant expansion with the $h_{i_1 c}(t)$. The arguments are then the same as for \eqref{off_diagonal_expansion} with the exception that the term with $\delta_{i_1 v_1}$ no longer exist, instead that term is $T_d$ but with $G_{i_1 q_i} \to \inner{q_i, e_{i_1}}$.

\section*{Acknowledgements}
Z. G. Bao is partially supported by Hong Kong RGC Grant GRF 16304724
and 17304225.  T. Bucht and K. Schnelli are supported by Grant No. VR-2021-04703 from the Swedish Research Council.

\printbibliography

\begin{appendix}
\section{Proof sketch of Lemma \ref{diagonal_variance_perturbation_lemma}}
\label{easy_interpolation_lemma_proof_section}
The interpolation in \eqref{diagonal_var_interpolation_eq} has the nice properties of both the interpolations \eqref{interpolation_step_1} and \eqref{gaussian_to_gaussian_interpolating_flow}, \ie the variance terms almost cancel, and the matrices being interpolated are Gaussian. Therefore, we only briefly sketch the proof below.  
\begin{proof}[Proof of Lemma \ref{diagonal_variance_perturbation_lemma}]
Recall the interpolating flow \eqref{diagonal_var_interpolation_eq}, 
\begin{equation*}
	H(t) = \e^{-t/2} W^V + \sqrt{1 - \e^{-t}} W^{\widetilde{V}}.
\end{equation*} Entry-wise we denote the flow as,
\begin{equation*}
	h_{ab}(t) = \e^{-t/2} w_{ab}^{(v)} + \sqrt{1 - \e^{-t}} w_{ab}^{(\widetilde{v})}.
\end{equation*}
Recall that $w_{ab}^{(v)}, w_{ab}^{(\widetilde{v})}$ are centered Gaussians with second cumulants given by,
\begin{equation*}
	c^{(2)} (w_{ab}^{(v)}) = V_{ab}, \qquad c^{(2)}(w_{ab}^{(\widetilde{v})}) = (1 + \delta_{ab}) V_{ab}.
\end{equation*}

Proceeding as in the previous cases we estimate $\abs{\dF \EX[\widetilde{F}(X(t))]}$. Starting off with \eqref{basic_time_derivative_computation} yields,
\begin{equation}
	\begin{split}
		\dF \EX[\widetilde{F}(X(t))] & = \sum_{a, b} \frac{1 + \delta_{ab}}{2} \EX \left[\dot{h}_{ab}(t) \frac{\partial (\widetilde{F}(X(t)))}{\partial h_{ab}(t)}\right] \\ 
		& = \sum_{a,b} \frac{1+\delta_{ab}}{2} \, \mathbb{E} \!\left[ \left( -\frac12 \e^{-t/2} w_{ab}^{(v)} + \frac{\e^{-t}} {2\sqrt{1-\e^{-t}}} \,w_{ab}^{(\widetilde{v})} \right) \frac{\partial (\widetilde{F}(X(t)))}{\partial h_{ab}(t)} \right] \\ 
		& = -\sum_{a,b} \frac{1+\delta_{ab}}{4} \e^{-t} V_{ab}\,\mathbb{E}\!\left[\frac{\partial^2 (\widetilde{F}(X(t)))}{\partial h_{ab}(t)^2}\right] +\sum_{a,b}\frac{(1+\delta_{ab})^{2}}{4} \e^{-t}V_{ab}\,\mathbb{E}\!\left[\frac{\partial^2 (\widetilde{F}(X(t)))}{\partial h_{ab}(t)^2}\right].
	\end{split}
	\label{first_computations}
\end{equation} In the last equality we use cumulant expansion and since $W^V$ and $W^{\widetilde{V}}$ are centered Gaussian matrices, all cumulants except the second vanish. Next, we see that the terms on the last row of \eqref{first_computations} almost cancel, the terms in the sums for $a \neq b$ are identical. Hence, we obtain,
\begin{equation}
	\begin{split}
		\dF \EX[\widetilde{F}(X(t))] & = \sum_{a} \frac12 \e^{-t} V_{aa} \EX \left[\frac{\partial^2 (\widetilde{F}(X(t)))}{\partial h_{aa}(t)^2}\right].
		\label{second_computation}
	\end{split}
\end{equation} 

Finally, we conclude the proof of Lemma \ref{diagonal_variance_perturbation_lemma} by using that $V_aa \leq \Csup / N$. This means that we can bound the resulting terms in \eqref{second_computation} using Lemma \ref{local_law_bound} in the same way that we bounded the terms in the expansion \eqref{cumulant_expansion1} with $d_N \leq 0$. 
\end{proof}
\end{appendix}
				
\end{document}